\documentclass{birkjour}
\usepackage{hyperref}
\usepackage{amsmath}
\usepackage{amssymb}
\usepackage{amsfonts}
\usepackage{amsthm}
\usepackage{latexsym}
\usepackage{esint}
\usepackage{mathtools}
\usepackage{mathrsfs}
\usepackage{graphicx}
\usepackage{bbm}
\usepackage{color}
\usepackage{bm}
 \newtheorem{thm}{Theorem}[section]
 
 \newtheorem{lem}[thm]{Lemma}
 \newtheorem{prop}[thm]{Proposition}
 \theoremstyle{definition}
 \newtheorem{defn}[thm]{Definition}
 \theoremstyle{remark}
 \newtheorem{rem}[thm]{Remark}
 
 \numberwithin{equation}{section}

\begin{document}

%
%
%
%
%
%
%
%
%

\title{A multi-scale limit of a  randomly forced rotating  $3$-D compressible fluid}

\author{Prince Romeo Mensah}

\address{%
School of mathematical and computer science\\
Department of Mathematics\\
Heriot--Watt University\\
EH14 4AS\\
Edinburgh, U.K.}

\email{pm27@hw.ac.uk}

\thanks{The author acknowledges the financial support of the Department of Mathematics, Heriot--Watt University, through the James--Watt scholarship. }


\date{\today}

\begin{abstract}
We study  a singular limit of a scaled compressible Navier--Stokes--Coriolis system driven by both a deterministic and stochastic forcing terms in three dimensions. If the Mach number is comparable to the Froude number with both proportional to say $\varepsilon\ll 1$, whereas the Rossby number scales like $\varepsilon^m$ for $m>1$ large, then we show that any family of weak martingale solution to the $3$-D randomly  forced rotating compressible equation (under the influence of a deterministic centrifugal force) converges in probability, as $\varepsilon\rightarrow0$, to the $2$-D incompressible  Navier--Stokes system with a corresponding random forcing term. 
\end{abstract}

\maketitle

\section{Introduction}
Our aim is to study the following singular limit problem for rotating fluids
\begin{equation}
\begin{aligned}
\label{comprSPDE0}
\mathrm{d}\varrho + \mathrm{div}(\varrho \mathbf{u})\mathrm{d}t &= 0,
 \\
\mathrm{d}(\varrho \mathbf{u}) + \Big[\mathrm{div}(\varrho  &\mathbf{u}\otimes \mathbf{u})  +\frac{1}{\mathrm{Ro}}\varrho(\mathbf{e}_3 \times\mathbf{u})   + \frac{1}{\mathrm{Ma}^2}\nabla p(\varrho) \Big]\mathrm{d}t
\\
& = 
\mathrm{div} \,\mathbb{S}(\nabla\mathbf{u}) \,\mathrm{d}t 
+
\frac{1}{\mathrm{Fr}^2}\varrho\nabla G
+
 \Phi(\varrho,\varrho \mathbf{u})\mathrm{d}W
\end{aligned}
\end{equation}
where the density $\varrho$ and velocity vector field $\mathbf{u}$ takes its values from the space $\mathcal{O}=\mathbb{R}^2\times(0,1)$. The term  $\frac{1}{\mathrm{Ro}}\varrho(\mathbf{e}_3 \times\mathbf{u})$ in \eqref{comprSPDE0} above accounts for rotation  in the fluid due to Coriolis forces  and the factor $\frac{1}{\mathrm{Ro}}$ - which is the reciprocal of the Rossby number - measures the intensity or the speed of this rotation.  $\mathbf{e}_3=(0,0,1)$ is the unit vector in the vertical $x_3$-direction. The centrifugal force term is essentially of the form $\nabla G\approx \nabla(\vert x_1\vert^2+\vert x_2\vert^2)$ with $(x_1,x_2)\in \mathbb{R}^2$ and with $\frac{1}{\mathrm{Fr}^2}$ - the squared reciprocal of the  Froude number - quantifying the level of stratification in the fluid. Here, $p(\varrho)=\varrho^\gamma$ with $\gamma>\frac{3}{2}$ is the isentropic pressure, $\mathrm{Ma}$ is the Mach number 
 and the viscous stress tensor is
\begin{align}
 \mathbb{S}(\nabla\mathbf{u}) :=  \nu\left(\nabla \mathbf{u} + \nabla^T\mathbf{u} - \frac{2}{3}\mathrm{div}\,\mathbf{u}\mathbb{I} \right) 
\end{align}  
with viscosity coefficient satisfying $\nu>0$
. 
As would be made clear in Section \ref{subsec:energyineq}, we intend to use a version of Korn's inequality and as such, it is of technical importance to omit the bulk part of the viscous stress tensor. 

A prototype for the stochastic forcing term will be 
\begin{equation}
\label{diffCoeciffients}\Phi(\varrho,\varrho \mathbf{u})\,\mathrm{d}W \approx \varrho \,\mathrm{d}W^1 +\varrho \mathbf{u} \,\mathrm{d}W^2
\end{equation}
for a pair of identically distributed independent Wiener processes $W^1$ and $W^2$.
We give the precise assumptions on the noise term in Section \ref{sec:noiseAssump}.

If we set the Rossby number $\mathrm{Ro}=\varepsilon$, the Froude number $\mathrm{Fr}=\varepsilon$ and the Mach number $\mathrm{Ma}=\varepsilon^m$ for some $m\gg 1$, then given a sequence $(\varrho_\varepsilon, \mathbf{u}_\varepsilon)$  of \textit{weak martingale solutions} to \eqref{comprSPDE0} (see Definition \ref{def:martSolution} for the precise definition), we show that its limit   $\mathbf{U}=[\mathbf{U}_h(x_1,x_2),0]$, solves the 2-dimensional Navier--Stokes system
\begin{equation}
\begin{aligned}
\label{2dIncom}
\mathrm{d}\mathbf{U}_h +\left[ \mathrm{div}_h(\mathbf{U}_h\otimes\mathbf{U}_h)  + \nabla_h \pi- \nu\Delta_h\mathbf{U}_h\right] \,\mathrm{d}t = \mathcal{P} \Phi(1,\mathbf{U}_h) \, \mathrm{d}W, 
\\
\mathrm{div}_h\mathbf{U}_h=0.
\end{aligned}
\end{equation}
Here, $\pi$ is an associated pressure term, $\mathcal{P} $ represents Helmholtz decomposition onto solenoidal  vector fields and the subscript $h$ which stands for `horizontal', represents the first two component of a $3$-D vector. The precise statement of this result is given in Theorem \ref{thm:mainRo}.

The convergence leading to \eqref{2dIncom} is obtained in probability which is stronger than the convergence in law obtained in the purely incompressible limit result studied in  \cite{breit2015incompressible} on the torus and \cite[Theorem 2]{mensah2016existence} in $\mathbb{R}^3$. The additional Coriolis term, in a sense, can therefore be seen to have some regularizing effect on the system. This is a consequence of the uniqueness result which is available for the 2-dimensional system \eqref{2dIncom}. We state this uniqueness result in Theorem \ref{thm:2duniqueness}.

In the deterministic setting, the analysis of incompressible rotating fluids have been studied by several authors. For an extensive review and introduction on the topic, the reader might want to see \cite{chemin2006mathematical}. However, important contributions include the works of Babin et al \cite{babin1999global, babin20013d} where they study certain class of solutions to the  system using amongst other techniques,  the Littlewood--Paley dyadic decomposition and further tools from algebraic geometry.

More recent work include \cite{ngo2016dispersive}, where the authors study a deterministic homogeneous compressible inviscid system on the whole space $\mathbb{R}^3$. They analyse the system under fast rotations with isotropic scale corresponding to when $\mathrm{Ro}=\mathrm{Ma}=\varepsilon$. This involved decomposing the system into a linear part  and non-linear part. Using Strichartz-type estimates, they establish the convergence to zero for the linear part. The non-linear part is then analysed using bootstrapping methods and some harmonic analysis tools including paradifferential calculus.

In \cite{feireisl2012a}, the authors study the deterministic counterpart of this limit problem with $\mathrm{Ro}=\mathrm{Ma}=\varepsilon\rightarrow0$ and with no centrifugal force effect. Using the so called RAGE theorem, they established the convergence to zero of the acoustic energy. The subsequent limit system is then given as a \textit{stream function} for the incompressible $2$-D Navier--Stokes system.

For a more general scalings of the form $\mathrm{Ro}=\varepsilon$, $\mathrm{Ma}=\varepsilon^m$ where $m\geq 1$ and $\mathrm{Fr}=\varepsilon$, which is more in line with what we study in this paper, the authors in \cite{feireisl2012multi} then study the limit problem under the influence of centrifugal force. If $m=1$, they obtain a $2$-D linear system with radially symmetric solutions whereas the multi-scale limit problem corresponding to $m\gg1$ converges to the $2$-D Navier--Stokes system. In this later case, the choice of $m$ subsequently eliminates the effect due to the centrifugal force.

There is very little results for  stochastic problems involving rotation. In \cite{flandoli2012stochastic}, they study averaging results for the $3$-D stochastic incompressible Navier--Stokes equation under fast rotation on a periodic domain. Here, an additive white noise is considered and the limit variables solves the so-called $3$-D stochastic resonant averaged equation.

As far as we can tell, there are no available results for compressible rotating fluids with stochastic forcing. Indeed, apart from the low-Mach number result in
\cite{breit2015incompressible, mensah2016existence}, the other result pertaining to such singular limits, we believe, is contained in \cite{breit2015compressible}. In \cite{breit2015compressible}, the combined effect of the low Mach number regime and the high Reynolds number $\mathrm{Re}$ is studied on a torus. From a $3$-D stochastic compressible Navier--Stokes equation (without rotation), they obtain in the limit, a $3$-D stochastic incompressible Euler equation.

We now give a brief outline of the paper. First of all, unless otherwise stated, the assumptions that we make in Section \ref{sec:prelim} will apply throughout the paper. 
We also define in that section, the various concept of solutions in Section \ref{subsec:solution}, as well as state the main result in Section \ref{subsec:main}.

We will then devote the entirety of Section \ref{sec:proofofmain} and Section \ref{sec:mainThmContin} to the proof  of our main Theorem \ref{thm:mainRo}. Our compactness arguments in Section \ref{subsec:compactness} will start by first establishing uniform estimates in Section \ref{subsec:energyineq}. Obtaining such uniform bounds will rely on the relative energy inequality introduced in \cite{breit2015compressible} in the context of stochastic compressible fluids.
We then show in Lemma \ref{lem:tighness}, tightness of the joint law on the path space defined on uniformly bounded sequences obtained in Section \ref{subsec:energyineq}. Finally, we conclude the section by using the Jakubowski--Skorokhod theorem \cite{jakubowski1998short} to establish almost sure convergent subsequence in the topology of the path space mentioned above.

Section \ref{subsec:limits} through to Section \ref{sec:convectiveTerm} will involve the justification of the limit system by treating the most important terms separately. A crucial part of  the analysis involves the corresponding acoustic wave equation which we study  in Section \ref{subsec:acoustic}. In this regard, the proof of the crucial result, Lemma \ref{lem:uniformGrad}, will rely on Fourier analysis, semigroup theory and regularization to obtain the mild form of the acoustic system. We then use Strichartz-type estimates to obtain uniform bounds for the (rescaled) gradient part of momentum. By scaling back, we eventually show that this part of the momentum vanishes in the limit. We then follow this by showing in Section \ref{sec:verticalSolenoidal} that the vertical average of the solenoidal part of momentum converges in the limit, to the full velocity.

In Section \ref{sec:Coriolis} we show that by considering the vertical averages, one can conclude that the Coriolis term is a gradient vector field and thus weakly solenoidal. Also, as mentioned in the previous paragraph, we study in Section \ref{sec:verticalSolenoidal},  just the vertical average of the solenoidal part of momentum. This leads us to justify in Section \ref{sec:oscillatoryMomentum} that any residual or oscillatory term obtained after the taking of vertical averages does not contribute to the limit system.  We will then devote Section \ref{sec:convectiveTerm}  to the proof of Lemma \ref{lem:convectiveConvergence} which identifies the limit in the  non-linear convective term.

Finally, we complete the proof of Theorem \ref{thm:mainRo} in Section \ref{subsec:pathwise} by using the unique (pathwise) solvability of the limit problem \eqref{2dIncom}.  The main tool is based on the recent result by Breit et al \cite[Theorem 2.10.3]{breit2017stoch} that  extends the original Gy\"ongy--Krylov's characterization  of  convergence in probability on Polish spaces \cite{gyongy1996existence} to  quasi-Polish spaces (this includes Banach spaces with weak topology). Having established convergence in law in Section \ref{sec:proofofmain} and with 2-D uniqueness Theorem \ref{thm:2duniqueness} in hand, we gain convergence in probability to the limit problem.

\section{Preliminaries}
\label{sec:prelim}
\subsection{Notations and definitions}
Let us start with a few notations. We set  $Q_T=(0,T)\times \mathcal{O}$ for fixed $T>0$ and consider the following microscopic state variables. For $x=(x_1,x_2,x_3)\in \mathcal{O}$, we let $x_h=(x_1,x_2)\in\mathbb{R}^2$ represent its first two or `horizontal' component  and with the third or `vertical' component  $x_3\in(0,1)$. We now define on $\mathcal{O}$, the macroscopic state variables $\varrho=\varrho(t,x)$ and $\mathbf{u}=\mathbf{u}(t,x)$ which are respectively, a non-negative scaler  and a three dimensional Euclidean vector valued functions representing the density and velocity fields. The vector valued function $(\varrho\mathbf{u})=(\varrho\mathbf{u})(t,x)$ represents the momentum.
%
%
%

We shall reserve the following short-hand notation $\Vert \cdot \Vert_{L^p_{x}}$ for the globally defined norms  on the whole space $L^p(\mathbb{R}^3)$. In this case, the integral over the whole space $\mathbb{R}^3$ is to be understood as the extension by zeroes outside of $\mathcal{O}$ whenever the function is only defined on $\mathcal{O}$.   An extension of this notation will be $\Vert \cdot \Vert_{L^p_{t,x}}$ or $\Vert \cdot \Vert_{L^q_\omega L^p_{t,x}}$ which will refer to the norms on $L^p((0,T)\times\mathbb{R}^3)$ and $L^q(\Omega ;L^p(0,T)\times\mathbb{R}^3))$ respectively, as well as similar variants. However, integrals over proper subsets of $\mathbb{R}^3$ will be made explicit. For example, we shall write $\Vert \cdot \Vert_{L^p(K)}$ for the usual Lebesgue norm whenever $K\subset\mathbb{R}^3$. Finally, we write $a \lesssim b$ and $a \lesssim_p b$ if there exists respectively,  generic constants $c>0$ and $c(p)>0$ such that $a \leq c b$ and $a \leq c(p) b$ holds respectively.

\subsection{Assumptions on the stochastic force}
\label{sec:noiseAssump}
Throughout this paper,  we assume that $(\Omega,\mathscr{F},(\mathscr{F}_t)_{t\geq0},\mathbb{P})$ is a stochastic basis with a complete right-continuous filtration $(\mathscr{F}_t)_{t\geq0}$.
$W$ is a $(\mathscr{F}_t)$-cylindrical Wiener process, that is, there exist a sequence of mutually independent $1$-D Brownian motions $(\beta_k)_{k\in\mathbb{N}}$ and an orthonormal basis $(e_k)_{k\in\mathbb{N}}$ of a separable Hilbert space $\mathfrak{U}$ such that
\begin{align*}
W(t) =  \sum_{k\in\mathbb{N}}\beta_k(t)e_k, \quad t\in[0,T].
\end{align*}
Now set $\mathbf{m}:=\varrho\mathbf{u}$ and assume that there exists a compact set $\mathcal{K} \subset\mathbb{R}^2$ for which we set $K:=\mathcal{K}\times[0,1]\subset \overline{\mathcal{O}}$. We then assume the existence of some $C^1$-functions $\mathbf{g}_k: \mathcal{O}\times \mathbb{R}_+ \times \mathbb{R}^3  \rightarrow \mathbb{R}$ whose decompositions are made up of functions  $\underline{\mathbf{g}}_k: \mathcal{O}\times \mathbb{R}_+  \rightarrow \mathbb{R}$  and  $\alpha_k:=\alpha_k(x) :\mathcal{O} \rightarrow \mathbb{R}$ such that 
\begin{align}
 \mathbf{g}_k(x, \varrho, \mathbf{m})=  \underline{\mathbf{g}}_k(x, \varrho)
+
\alpha_k(x)\, \mathbf{m},  \quad k\in \mathbb{N}.
\end{align}
These coefficients are assumed to satisfy the uniform  bounds
\begin{equation}
\begin{aligned}
\label{stochCoeffBound}
&\sum_{k\in\mathbb{N}}  \left\vert \alpha_k   \right\vert^2   < \infty,
\quad
\sum_{k\in\mathbb{N}}  \left\vert \underline{\mathbf{g}}_k(x, \varrho)   \right\vert^2   \lesssim \varrho^2 ,
\quad
\sum_{k\in\mathbb{N}}  \left\vert\nabla_{\varrho} \,\underline{\mathbf{g}}_k(x, \varrho) \right\vert^2   \lesssim 1.
\end{aligned}
\end{equation}
Then if we define the map $\Phi(\varrho,  \mathbf{m}):\mathfrak{U}\rightarrow    L^1(K)$ by
\begin{align}
\Phi(\varrho,  \mathbf{m}) e_k   = \mathbf{g}_k(\cdot, \varrho(\cdot), \mathbf{m} (\cdot))=  \underline{\mathbf{g}}_k(\cdot, \varrho(\cdot))
+
\alpha_k(\cdot)\, \mathbf{m} (\cdot),
\end{align}
where
\begin{align}
\label{noiseSupport}
\mathrm{spt}(\mathbf{g}_k)\Subset K, \quad \text{for any } k\in\mathbb{N},
\end{align}
we can use   the embedding $L^1(K)\hookrightarrow  W^{-l,2}(K)$ where $l>\frac{3}{2}$, and \eqref{stochCoeffBound}--\eqref{noiseSupport} to show that
\begin{equation}
\begin{aligned}
\label{noiseEst}
&\Vert\Phi(\varrho, \mathbf{m})  \Vert_{L_2(\mathfrak{U};W^{-l,2}(\mathcal{O}))}^2
=\sum_{k\in\mathbb{N}}\Vert \mathbf{g}_k(x, \varrho, \mathbf{m})   \Vert_{W^{-l,2}(\mathcal{O})}^2 
\\&
= \sum_{k\in\mathbb{N}}\Vert \mathbf{g}_k(x, \varrho, \mathbf{m})  \Vert_{W^{-l,2}({K})}^2 
\lesssim \sum_{k\in\mathbb{N}}\Vert \mathbf{g}_k(x, \varrho, \mathbf{m})   \Vert_{L^1(K)}^2 
\\&
\lesssim (\varrho)_{{K}}\,   \int_{{K}} \sum_{k\in\mathbb{N}} \varrho^{-1}\vert \mathbf{g}_k(x, \varrho, \mathbf{m}) \vert^2\, \mathrm{d}x  
\\&
\lesssim (\varrho)_{{K}}\,   \int_{{K}} \sum_{k\in\mathbb{N}} \big(\varrho^{-1}\vert \underline{\mathbf{g}}_k(x, \varrho)\vert^2
+
\vert
\alpha_k\vert^2\,\varrho\,\vert \mathbf{u}   \vert^2\big)\, \mathrm{d}x 
\\&
\lesssim (\varrho)_{{K}}\,  \int_{{K}} \big(1  + \varrho^\gamma + \varrho  \vert \mathbf{u} \vert^2 \big)\mathrm{d}x.
\end{aligned}
\end{equation}
where $(\varrho)_{{K}}$ represents the average density over the compact set ${K}$ and where we have used $\varrho\leq 1+\varrho^\gamma$ in the last step. The left-hand side of \eqref{noiseEst} is therefore uniformly bounded provided  $\varrho  \in  L^\gamma_{\mathrm{loc}}(\mathcal{O})$ and  $\sqrt{\varrho}\mathbf{u} \in L^2_{\mathrm{loc}}(\mathcal{O})$. If so, then the stochastic integral $\int_0^\cdot\Phi(\varrho,\mathbf{m})\mathrm{d}W$ is a well-defined $(\mathscr{F}_t)$-martingale taking value in $W^{-l,2}(\mathcal{O})$.

As already mentioned in the introduction, we expect in the limit, a process that solves the 2-D Navier--Stokes system, Eq. \eqref{2dIncom}. As a result,
we assume in analogy to \eqref{noiseSupport}--\eqref{noiseEst}, a diffusion coefficient of the kind
\begin{align*}
 \Psi( \mathbf{U}_h)e_k  = \mathbf{g}_{k,h}\big(\cdot,  \mathbf{U}_h(\cdot) \big) =   \underline{\mathbf{g}}_{k,h}(\cdot)
+
\alpha_k(\cdot)\, \mathbf{U}_h(\cdot)  
\end{align*} 
with $\Psi(\mathbf{U}_h) = \mathcal{P}\Phi(1,\mathbf{U}_h)$ and coefficients  satisfying the bound
\begin{equation}
\begin{aligned}
\sum_{k\in\mathbb{N}}  \left\vert \alpha_k   \right\vert^2  
+
 \sum_{k\in\mathbb{N}} \left\vert \underline{\mathbf{g}}_{k,h}(x_h, 1)   \right\vert^2   \lesssim 1.
\end{aligned}
\end{equation}
Subsequently, the estimate for the noise term  in Eq. \eqref{2dIncom} becomes:
\begin{equation}
\begin{aligned}
\label{noiseEst2D}
\Vert\Psi( \mathbf{U}_h)  \Vert&_{L_2(\mathfrak{U};W^{-l,2}(\mathbb{R}^2))}^2
\lesssim \sum_{k\in\mathbb{N}}\Vert \underline{\mathbf{g}}_{k,h}(x_h)
+
\alpha_k(x_h)\, \mathbf{U}_h   \Vert_{L^1(\mathcal{K})}^2 
\\
&\lesssim \int_{\mathcal{K}} \sum_{k\in\mathbb{N}}\big(\vert \underline{\mathbf{g}}_{k,h}(x_h)\vert^2
+
\vert\alpha_k\, \mathbf{U}_h   \vert^2 \big) \mathrm{d}x
\\
&\lesssim \int_{\mathcal{K}} (1 +  \vert \mathbf{U}_h \vert^2)\mathrm{d}x
\end{aligned}
\end{equation}
where $\mathcal{K}$ is the same compact set hidden in \eqref{noiseSupport} above and where we have used the continuity of the operator $\mathcal{P}$. 

Lastly, we define the auxiliary space $\mathfrak{U}_0 \supset  \mathfrak{U}$ via
\begin{align*}
\mathfrak{U}_0  =\left\{\mathbf{u}= \sum_{k\in\mathbb{N}}c_ke_k\,;\quad  \sum_{k\in\mathbb{N}} \frac{c^2_k}{k^2}<\infty \right\}
\end{align*}
and endow it with the norm
\begin{align*}
\Vert  \mathbf{u}  \Vert^2_{\mathfrak{U}_0}  =  \sum_{k\in\mathbb{N}} \frac{c^2_k}{k^2}, \quad   \mathbf{u}=\sum_{k\in\mathbb{N}}c_ke_k.
\end{align*}
Then it can be shown that $W$ has $\mathbb{P}$-a.s. $C([0,T];\mathfrak{U}_0)$ sample paths with the Hilbert-Schmidt embedding $\mathfrak{U}\hookrightarrow \mathfrak{U}_0$. See \cite{da2014stochastic}.

\subsection{Boundary and far reach conditions}
Since we are working on an semi bounded/unbounded spatial domain, we supplement our system  \eqref{comprSPDE0} with the far reach condition
\begin{align}
\label{boundaryCond}
\varrho\rightarrow\overline{\varrho}_{\varepsilon},\quad \mathbf{u}\rightarrow0 \quad\text{as} \quad \vert x_h\vert\rightarrow\infty, \quad \mathbb{P}\text{-a.s.}
\end{align}
for some time independent function $\overline{\varrho}_{\varepsilon}=\overline{\varrho}_{\varepsilon}(x)>0$ as well as the complete slip boundary condition for the velocity field
\begin{equation}
\begin{aligned}
\label{boundaryCond1}
\mathbf{u}\cdot\mathbf{n}\big\vert_{\partial\mathcal{O}}=\pm \,u_3\big\vert_{\partial\mathcal{O}}=0,
\\ \left(\left[\mathbb{S}(\nabla \mathbf{u})\cdot\mathbf{n}\right]\times\mathbf{n}\right)\big\vert_{\partial\mathcal{O}}
=  \left(S_{23},-S_{13},0\right)\big\vert_{\partial\mathcal{O}}
=0
\end{aligned}
\end{equation}
where $\mathbf{n}=(0,0,\pm1)$ is the outer normal vector to the boundary; so as  to entirely eliminate the influence of boundary effects.
\subsection{The relative energy functional}
We now introduce the relative energy functional which compares `solutions' of \eqref{comprSPDE0} with some smooth functions $r$ and $\mathbf{U}$. Let   start by  first defining the following. For the \textit{isentropic} pressure function $p(z)=z^\gamma$ with $p\in C^1[0,\infty)\cap C^2(0,\infty)$,
\begin{align}
\label{densityPotential}
P(\varrho)=\varrho \int_1^\varrho z^{\gamma-2}\,\mathrm{d}z
\end{align}
represent the corresponding  \textit{pressure potential}. 

Now we assume that the (smooth) functions $r,\mathbf{U}$ are random variables that are adapted to the filtration $(\mathscr{F}_t)_{t\geq0}$ and satisfies:
\begin{align}
r>0, \quad (r-\overline{\varrho}_{\varepsilon})\in C_c^\infty([0,T]\times \overline{\mathcal{O}}), \quad \mathbf{U} \in C_c^\infty([0,T]\times \mathcal{O}),
\end{align}
$\mathbb{P}$-a.s. Additionally, we assume that $\overline{\varrho}_{\varepsilon}=\overline{\varrho}_{\varepsilon}(x)$ solves  the static problem
\begin{align}
\label{staticProblem}
\nabla \,\overline{\varrho}_{\varepsilon}^\gamma =\varepsilon^{2(m-1)}\overline{\varrho}_{\varepsilon}\,\nabla G \quad
 \text{in} \quad \mathcal{O}
\end{align} 
for a  non-negative time independent deterministic force $G$ that satisfy
\begin{equation}
\begin{aligned}
\label{centrifugal}
G=G(x)\geq 0, \quad  G\in W^{1,1}(\mathcal{O}) \cap W^{1,\infty}(\mathcal{O}).
\end{aligned}
\end{equation} 
We now set
\begin{align}
\label{entropyRel}
{H} \left(\varrho,r \right)=P(\varrho) -P'(r)(\varrho - r) -P(r)
\end{align}
and define
\begin{equation}
\begin{aligned}
\label{entropy}
\mathcal{E}&\left(\varrho ,\mathbf{u} \left\vert \right. r, \mathbf{U}  \right) (t,\cdot) 
:=
\int_{\mathcal{O}}\Big[\frac{\varrho}{2} \vert \mathbf{u} - \mathbf{U}\vert^2  
+ \frac{1}{\mathrm{Ma}^2}{H} \left(\varrho,r \right) \Big]\,(t,\cdot)\,\mathrm{d}x
\end{aligned}
\end{equation}
to be the \textit{relative energy functional}.

\begin{rem}
By using the  identity  \eqref{densityPotential}, one can easily check that \eqref{staticProblem} is equivalent to solving $\nabla \,P'(\overline{\varrho}_{\varepsilon}) =\varepsilon^{2(m-1)}\,\nabla G$ so that
\begin{align}
\label{staticProblem1}
P'\big(\overline{\varrho}_{\varepsilon}(x)\big) = \varepsilon^{2(m-1)}G(x) + P'(1).
\end{align} 
Since $G$ is non-negative,  it follows from \eqref{staticProblem1} that for any $x\in \mathcal{O}$ and $m>1$,
\begin{align*}
1 \leq \overline{\varrho}_{\varepsilon}(x) <c \quad\text{and} \quad\overline{\varrho}_{\varepsilon}(x)\rightarrow 1
\end{align*} 
as $\varepsilon\rightarrow0$ for some $c>0$. Furthermore, by using the Lipschitz continuity of $G$, we can deduce from \eqref{staticProblem1} that for any $x\in \mathcal{O}$ with $\vert x_h\vert\leq k\varepsilon^{-\alpha}$,  $k>0$ and $0\leq \alpha\leq m-1$, 
\begin{align}
\label{densityAndOne}
\vert \overline{\varrho}_{\varepsilon}(x) -1 \vert \lesssim_k \varepsilon^{2(m-1-\alpha)}.
\end{align}
\end{rem}
\subsection{Concepts of solution}
\label{subsec:solution}
We now define different notions of solution that will be considered in this paper. We follow the approach of  \cite{Hof} where we now define on our special geometry $\mathcal{O}$, a corresponding solution to \eqref{comprSPDE0} which is weak in both the probabilistic and PDE sense. This is given in Definition \ref{def:martSolution} below. 
\begin{defn}
\label{def:martSolution}
Let $\overline{\varrho}>0$. If $\Lambda$ is a Borel probability measure on $L^\gamma_{\mathrm{loc}}(\mathcal{O})\times L^\frac{2\gamma}{\gamma+1}_{\mathrm{loc}}(\mathcal{O})$, then we say that
$\left[(\Omega,\mathscr{F},(\mathscr{F}_t),\mathbb{P});\varrho, \mathbf{u}, W  \right]$
is a \textit{finite energy weak martingale solution} of equation \eqref{comprSPDE0} with initial law $\Lambda$ provided
\begin{enumerate}
\item $(\Omega,\mathscr{F},(\mathscr{F}_t),\mathbb{P})$ is a stochastic basis with a complete right-continuous filtration,
\item $W$ is a $(\mathscr{F}_t)$-cylindrical Wiener process,
\item the density $\varrho$ satisfies $\varrho\geq 0,$ $t\rightarrow \langle \varrho(t,\cdot), \phi   \rangle  \,\in \,C[0,T]$ for any $\phi  \in C^\infty_c (\mathcal{O})$ $\mathbb{P}$-a.s., the function  $t \mapsto \langle \varrho(t,\cdot), \phi   \rangle$ is progressively measurable, and
\begin{align*}
\mathbb{E}\left[  \sup_{t\in[0,T]}\Vert \varrho(t,\cdot) \Vert^p_{L^{\gamma}(K)}  \right]   <   \infty   
\end{align*}
for all $1\leq p   <\infty$ and for all compact set $K\subseteq {\mathcal{O}}$,
\item the momentum $ \varrho\mathbf{u}$ satisfies $t\rightarrow \langle  \varrho\mathbf{u}, \bm{\phi}  \rangle  \,\in \,C[0,T]$ for any $\bm{\phi} \in C^\infty_c (\mathcal{O})$ $\mathbb{P}$-a.s., the function  $t \mapsto \langle  \varrho\mathbf{u}, \bm{\phi}  \rangle$ is progressively measurable and for all  $1\leq p   <\infty$
\begin{align*}
\mathbb{E}\left[  \sup_{t\in[0,T]}\Vert \varrho\mathbf{u}   \Vert^p_{L^{\frac{2\gamma}{\gamma+1}}(K)}  \right]   <   \infty   ,
\end{align*}
for all compact set $K\subseteq \mathcal{O}$,
\item the velocity field $\mathbf{u}$ is $(\mathscr{F}_t)$-adapted, $\mathbf{u}\in  L^p\left( \Omega; L^2\left(0,T;W^{1,2}_{\mathrm{loc}}(\mathcal{O})  \right)   \right)$ 
and,
\begin{align*}
\mathbb{E}\left[\left( \int_0^T \Vert  \mathbf{u}  \Vert^2_{W^{1,2}(K)} \mathrm{d}t   \right)^p  \right]  <   \infty  
\end{align*}
for all $1\leq p   <\infty$ and for all compact set $K\subseteq \mathcal{O}$,
\item there exists $\mathscr{F}_0$-measurable random variables $(\varrho_0,  \varrho_0\mathbf{u}_0)=(\varrho(0),  \varrho\mathbf{u}(0))$ such that $\Lambda = \mathbb{P}\circ (\varrho_0,  \varrho_0\mathbf{u}_0)^{-1}$,
\item for all $\psi \in C^\infty_c (\mathcal{O})$ and $\bm{\phi}  \in C^\infty_c (\mathcal{O})$ and all $t\in [0,T]$, it holds $\mathbb{P}$-a.s.\footnote{Here and throughout the rest of this paper, we shall use the notations $\langle u,v \rangle=\int_{\mathcal{O}}uv\,\mathrm{d}x$  and $\langle u,v \rangle_h=\int_{\mathbb{R}^2}uv\,\mathrm{d}x_h$  in 3-D and 2-D respectively and where the 2-D version corresponds to the horizontal or first two Cartesian components of the 3-D version.}
\begin{align*}
\langle \varrho&(t)\, ,\, \psi \rangle   =  \langle \varrho_0, \psi \rangle  +  \int_0^t \langle  \varrho\mathbf{u}, \nabla \psi \rangle  \mathrm{d}r,
\\
\langle & \varrho\mathbf{u}(t)\, ,\, \bm{\phi}  \rangle   =  \langle \varrho_0\mathbf{u}_0\, ,\, \bm{\phi}  \rangle  +  \int_0^t \langle  \varrho\mathbf{u}\otimes \mathbf{u}, \nabla \bm{\phi}  \rangle  \mathrm{d}r -  \frac{1}{\mathrm{Ro}}  \int_0^t \langle \varrho(\mathbf{e}_3\times\mathbf{u}) \, ,\, \bm{\phi} \rangle\mathrm{d}r
\\
&-\int_0^t \langle \mathbb{S}(\nabla\mathbf{u}) \, ,\, \mathrm{div}\bm{\phi}  \rangle  \mathrm{d}r    +    \frac{1}{\mathrm{Ma}^2}  \int_0^t \langle p(\varrho)\, ,\, \mathrm{div}\bm{\phi} \rangle  \mathrm{d}r
\\
&+    \frac{1}{\mathrm{Fr}^{2}}  \int_0^t \langle \varrho\,\nabla G\, ,\,  \bm{\phi}  \rangle  \mathrm{d}r
+\int_0^t \langle \Phi(\varrho,\varrho\mathbf{u})\mathrm{d}W\, ,\, \bm{\phi} \rangle  ,
\end{align*}
\item for a.e. $\tau\in (0,T)$, the following inequality
\begin{equation}
\begin{aligned}
\label{relativeEnergy}
 \int_{\mathcal{O}}&\bigg[\frac{\varrho}{2} \vert \mathbf{u} \vert^2  
+ \frac{H(\varrho, \overline{\varrho})}{\mathrm{Ma}^2} \bigg](\tau)\,\mathrm{d}x
+
\int_0^\tau \int_{\mathcal{O}}\mathbb{S}(\nabla \mathbf{u}):  \nabla \mathbf{u} \,\mathrm{d}x  \,\mathrm{d}t
\\
&\leq
\int_{\mathcal{O}}\bigg[\frac{\vert \varrho_0\mathbf{u}_0 \vert^2}{2\varrho_0}  
+ \frac{H(\varrho_0, \overline{\varrho})}{\mathrm{Ma}^2} \bigg]\,\mathrm{d}x
\\&
+\frac{1}{2}\int_0^\tau \,\bigg(\int_{\mathcal{O}}\sum_{k\in\mathbb{N}}\varrho^{-1}\vert\mathbf{g}_k(x,\varrho,\varrho\mathbf{u})\vert^2\,\mathrm{d}x\bigg)\mathrm{d}t
+M_{R}(\tau)
\end{aligned}
\end{equation}
holds $\mathbb{P}$-a.s. for a real-valued martingale $M_{R}$ given by 
\begin{align}
M_R(\tau)=\int_0^\tau \int_\mathcal{O} \mathbf{u} \cdot \Phi(\varrho,\varrho\mathbf{u}) \, \mathrm{d}x \, \mathrm{d}W
\end{align}
and satisfying the estimate
\begin{align}
\label{realMartingale}
\mathbb{E}\left(\sup_{t\in[0,T]}\vert M_{R}\vert^p \right) \leq c_p\left(1+ \mathbb{E}\left[\int_\mathcal{O}\left(\frac{\vert \varrho_0\mathbf{u}_0\vert^2}{2\varrho_0}+\frac{H(\varrho_0, \overline{\varrho}) }{\mathrm{Ma}^2}   \right)\,\mathrm{d}x\right]^p \right).
\end{align}
for all $p\in[1,\infty)$.
\item In addition, \eqref{comprSPDE0}$_1$ holds in the renormalized sense. That is, for any $\phi\in \mathcal{D}'(\mathcal{O})$ and $b \in C^0[0,\infty)\cap  C^1(0,\infty)$ such that $\vert b'(t) \vert \leq ct^{-\lambda_0}$, $t\in(0,1]$, $\lambda_0<1$ and $\vert b'(t) \vert  \leq ct^{\lambda_1}$, $t\geq1$ where $c>0$ and $-1< \lambda_1<\infty$, we have that
\begin{align}
\label{renormalizedCont}
\mathrm{d}\langle b(\varrho),\phi \rangle = \langle b(\varrho)\mathbf{u},\nabla\phi \rangle\mathrm{d}t  -  \langle \left(b(\varrho) - b'(\varrho)\varrho\right)\mathrm{div}\mathbf{u},  \phi \rangle\mathrm{d}t.
\end{align} 
\end{enumerate}
\end{defn}

\begin{rem}
\label{rem:exist}
Following a similar argument as in the proof of \cite[Theorem 1]{mensah2016existence}, one can establish  on $\mathbb{T}^2_L\times\mathbb{T}_1$  instead of $\mathbb{T}^3_L$, the existence of a finite energy weak martingale solution to \eqref{comprSPDE0} in the sense of Definition \ref{def:martSolution} under the assumption that \eqref{stochCoeffBound} and \eqref{noiseSupport} holds. Here, $\mathbb{T}^p_L = \left([-L,L]\vert_{\{L,L\}}\right)^p$ and $\mathbb{T}_1 = [-1,1]\vert_{\{-1,1\}}$ are the $p$-D and $1$-D  flat tori with periods $2L\geq1$ and $2$ respectively. 
For $L$ fixed, existence of a {finite energy weak martingale solution}  follows from  \cite{Hof}. The aim will then be to pass to limit as $L\rightarrow \infty$ in analogy to the proof of \cite[Theorem 1]{mensah2016existence}. However this will yield a result posed on $\mathcal{O}=\mathbb{R}^2\times\mathbb{T}_1$ instead of the original geometry $\mathcal{O}=\mathbb{R}^2\times(0,1)$. This aforementioned reformulation into a purely periodic problem with a corresponding boundary condition  is allowed  after a special symmetrization of our density and velocity vector fields as given in \cite[Eq. 1.7]{feireisl2012multi}. This was originally proposed in \cite{ebin1983viscous} and for completeness, we state them below.
\begin{equation}
\begin{aligned}
\label{oddAndEven}
\varrho\left( \cdot\,, x_h,-x_3 \right)=\varrho\left( \cdot\,,x_h,x_3 \right),\quad  \mathbf{u}_h\left( \cdot\,,x_h,-x_3 \right)=\mathbf{u}_h\left( \cdot\,,x_h,x_3 \right),\\
-u_3\left( \cdot\,,x_h,-x_3 \right)=u_3\left( \cdot\,,x_h,x_3 \right).
\end{aligned}
\end{equation}
That is, the horizontal component of velocity and density are extended from $(0,1)$ to $\mathbb{T}_1$ as an even function in $x_3$ whereas the vertical component of velocity is extended to an odd function in $x_3$.

In accordance with \eqref{oddAndEven}, the functions $\mathbf{g}_k(x,\varrho,\mathbf{m})$ in \eqref{noiseSupport} are assumed to satisfy 
\begin{equation}
\begin{aligned}
\label{oddAndEvenNoise}
-g_{k,3}\left(x_h,-x_3,\cdot, \mathbf{m}_h,-m_3 \right)=g_{k,3}\left( x_h,x_3,\cdot, \mathbf{m}_h,{m}_3, \right),
\\
\mathbf{g}_{k,h}\left( x_h,-x_3,\cdot, \mathbf{m}_h,-m_3 \right)=\mathbf{g}_{k,h}\left( x_h,x_3,\cdot, \mathbf{m}_h,m_3 \right),
\end{aligned}
\end{equation}
where $g_{k,3}$ and $\mathbf{g}_{k,h}$, agrees correspondingly, to the `vertical' and `horizontal' components of the noise term. Lastly, we extend the potential of the centrifugal force  to $\mathbb{T}_1$ as an even function in $x_3$ , i.e.,
\begin{equation}
\begin{aligned}
\label{centrifugal1}
G(x_h, -x_3)=G(x_h,x_3).
\end{aligned}
\end{equation} 
These symmetric assumptions \eqref{oddAndEven}, \eqref{oddAndEvenNoise} and \eqref{centrifugal1} are thus, implicitly implied throughout the rest of this paper.
\end{rem}

%
\begin{rem}
\label{rem:orthogo}
Notice that the term $\varrho(\mathbf{e}_3\times\mathbf{u}) $ is orthogonal to $\mathbf{u}$ and so it vanishes during the compactness argument when we test the momentum equation with the velocity. Also, the centrifugal forcing term is easily controlled by a similar estimate as in \eqref{centForceDissipate} below.
\end{rem}
The nature of the limit system \eqref{2dIncom} naturally leads to a corresponding definition of a solution in $2$-D. Typically, this can either be simultaneously weak in the probabilistic and PDE sense, in analogy to Definition \ref{def:martSolution} above, or strong in at least one of these senses. The former notion is stated in Definition \ref{def:weakSol} below. For the later, which we state in Definition \ref{def:pathSol}, the solutions are weak in the PDE sense but strong in the sense of probability. This follows from uniqueness in 2-D which is currently unavailable for the 3-D counterpart. In this later case, the underlying probability space is fixed in advance. Consequently, existence of solution in the sense of the later yields the former. However, the analysis involved in this paper is such that, both versions are required.
\begin{defn}
\label{def:weakSol}
Let $\Lambda$ be a Borel probability measure on $L^2_{\mathrm{div}}(\mathbb{R}^2)$\footnote{Here and below, $C^\infty_{c,\mathrm{div}}(\mathbb{R}^2) := \{\mathbf{v} \in C^\infty_c(\mathbb{R}^2)  \, : \, \mathrm{div}(\mathbf{v})=0  \}$ and $L^2_{\mathrm{div}}(\mathbb{R}^2)=\overline{C^\infty_{c,\mathrm{div}}(\mathbb{R}^2)}^{\Vert\cdot\Vert_{L^2}}$ and similarly for the Sobolev space. }. Then we say that $[(\Omega,\mathscr{F},(\mathscr{F}_t),\mathbb{P}),\mathbf{u},W]$ is a
\textit{weak martingale solution} of equation \eqref{2dIncom} with initial datum $\Lambda$ provided:
\begin{enumerate}
\item $(\Omega,\mathscr{F},(\mathscr{F}_t),\mathbb{P})$ is a stochastic basis with a complete right-continuous filtration,
\item $W$ is an $(\mathscr{F}_t)$-cylindrical Wiener process,
\item $\mathbf{u}$ is $(\mathscr{F}_t)$-adapted, $\mathbf{u}\in  C_w\left([0,T];L^2_\mathrm{div}(\mathbb{R}^2)  \right)$ $\cap\, L^2(0,T;W^{1,2}_\mathrm{div}(\mathbb{R}^2))$ $\mathbb{P}$-a.s. 
and for all $p\in[1,\infty)$,
\begin{equation}
\begin{aligned}
\label{energy2d}
\mathbb{E}\left[\sup_{t\in(0,T)} \Vert  \mathbf{u}  \Vert^2_{L^2(\mathbb{R}^2)}  \right]^p
+
\mathbb{E}\left[\left( \int_0^T \Vert  \mathbf{u}  \Vert^2_{W^{1,2}(\mathbb{R}^2)} \mathrm{d}t   \right)^p  \right]  <   \infty,
\end{aligned}
\end{equation}
\item $\Lambda = \mathbb{P}\circ (\mathbf{u}(0))^{-1}$,
\item for all $\bm{\phi} \in C^{\infty}_{c,\mathrm{div}} (\mathbb{R}^2)$ and all $t\in [0,T]$, it holds $\mathbb{P}$-a.s.,
\begin{equation}
\begin{aligned}
\label{weak2d}
\langle \mathbf{u}(t)\,,\, \bm{\phi}\rangle_h   &=  \langle \mathbf{u}(0) \,,\,\bm{\phi}\rangle_h -  \int_0^t \langle \mathbf{u}\otimes \mathbf{u}\,,\, \nabla\bm{\phi}\rangle_h  \,  \mathrm{d}r
+\nu \int_0^t\langle\nabla\mathbf{u}\,,\, \bm{\phi}\rangle_h \,  \mathrm{d}r    
\\
&+\int_0^t  \langle\mathcal{P} \Phi(1,  \mathbf{u})\mathrm{d}W\,,\, \bm{\phi}\rangle_h
\end{aligned}
\end{equation}
\end{enumerate}
where $\mathcal{P} $ is the Helmholtz decomposition onto the space of solenoidal vector fields.
\end{defn}
An even stronger notion of solution for the incompressible stochastic Navier--Stokes equation is the concept of  \textit{weak pathwise solution} given below.

\begin{defn}
\label{def:pathSol}
Let 
$\left(\Omega,\mathscr{F},(\mathscr{F}_t),\mathbb{P}\right)$ be a stochastic basis with an $(\mathscr{F}_t)$- cylindrical Wiener process $W$. Let $\mathbf{u}_0$ be an $\mathscr{F}_0$-measurable random variable. Then we say that $\mathbf{u}$ is a
\textit{weak pathwise solution} of the Navier--Stokes equation \eqref{2dIncom} with initial datum $\mathbf{u}_0$ provided:
\begin{enumerate}
\item the velocity $\mathbf{u}$ is $(\mathscr{F}_t)$-adapted, $\mathbf{u}\in  C_w\left([0,T];L^2_\mathrm{div}(\mathbb{R}^2)  \right) \cap $\\$L^2(0,T;W^{1,2}_\mathrm{div}(\mathbb{R}^2))$  $\mathbb{P}$-a.s. 
and for all $p\in[1,\infty)$, \eqref{energy2d} holds true,
\item the equality $\mathbf{u}(0)\,= \mathbf{u}_0$ holds $\mathbb{P}$-a.s.,
\item for all $\bm{\phi } \in C^{\infty}_{c,\mathrm{div}} (\mathbb{R}^2)$ and all $t\in [0,T]$, Eq. \eqref{weak2d} holds $\mathbb{P}$-a.s.
\end{enumerate}
\end{defn}

\begin{thm}
\label{thm:2duniqueness}
Let $(\Omega, \mathscr{F}, (\mathscr{F}_t),\mathbb{P})$ be a stochastic basis with an $(\mathscr{F}_t)$-cylindrical Wiener process $W$ and let $\mathbf{U}_0$ be an $\mathscr{F}_0$-measurable random variable belonging to the space $L^p(\Omega;L^2_{\mathrm{div}}(\mathbb{R}^2))$ for all $p\in[1,\infty)$.
If \eqref{noiseEst2D} holds,
 then there exists a unique weak pathwise solution to \eqref{2dIncom} in the sense of Definition \ref{def:pathSol} with initial condition $\mathbf{U}_0$
\end{thm}
\begin{proof}
See \cite{menaldi2002stochastic} with $\mathfrak{U}=l^2(\mathbb{R}^2)$.
\end{proof}


\subsection{Main result}
\label{subsec:main}
We now state the main result of this paper. Theorem \ref{thm:mainRo} below corresponds to the simultaneous low Rossby - low Mach - low Froude number limit result of the stochastic compressible Navier--Stokes--Coriolis equation taking into account, the influence of centrifugal force.
\begin{thm}
\label{thm:mainRo}
Set $\mathrm{Ma}=\varepsilon^m$, for $m>10$ and $\mathrm{Ro}=\mathrm{Fr}=\varepsilon$  in \eqref{comprSPDE0}.
Let $\gamma>\frac{3}{2}$ and assume that $\mathbf{U}_0=[\mathbf{U}_{h,0},0]\in L^2(\mathbb{R}^{3})$ is $\mathscr{F}^\varepsilon_0$ measurable. Consider the following initial data $\big(\varrho_{0,\varepsilon}, (\varrho\mathbf{u})_{0,\varepsilon}\big)\in L^\gamma(\mathcal{O}) \times  L^{\frac{2\gamma}{\gamma+1}}(\mathcal{O}) $ satisfying
\begin{equation}
\begin{aligned}
&
\varrho_{0,\varepsilon} = \overline{\varrho}_{\varepsilon}+\varepsilon^m\varrho_{0,\varepsilon}^{(1)}>0, \quad
\, \big\{ \sqrt{\overline{\varrho}_{\varepsilon}}\,\mathbf{u}_{0,\varepsilon}\big\}_{\varepsilon>0}\in L^2(\mathcal{O}),\,
\\&
\big\{\overline{\varrho}_{\varepsilon}^\frac{\gamma-2}{2} \varrho_{0,\varepsilon}^{(1)}\big\}_{\varepsilon>0}\in  L^2(\mathcal{O}),
 \quad
\big\{ \varrho_{0,\varepsilon}^{(1)}\big\}_{\varepsilon>0}\in L^\infty(\mathcal{O})\cap L^1(\mathcal{O}),
\\
&
\left\vert(\varrho\mathbf{u})_{0,\varepsilon}-\mathbf{U}_0 \right\vert + \left\vert \varrho_{0,\varepsilon}-\overline{\varrho}_{\varepsilon} \right\vert\leq\varepsilon^m\, M, 
\end{aligned}
\end{equation}
for a constant $M>0$ and for $\overline{\varrho}_{\varepsilon}>0$ solving \eqref{staticProblem}. 
%
If the collection\\ $\left[(\Omega,\mathscr{F},(\mathscr{F}_t),\mathbb{P});\varrho_\varepsilon, \mathbf{u}_\varepsilon, W  \right]$ is a family of finite energy weak martingale solution of \eqref{comprSPDE0} in the sense of Definition \ref{def:martSolution} with initial law $
\Lambda_\varepsilon=\mathbb{P}\circ [ \varrho_{0,\varepsilon}, (\varrho\mathbf{u})_{0,\varepsilon} ]^{-1}$, $\varepsilon\in(0,1)$
and uniformly bounded moment estimate
\begin{equation}
\begin{aligned}
\label{momentsBounded}
&\int_{L_x^\gamma\times L_x^{\frac{2\gamma}{\gamma+1}}}  \left\Vert \frac{1}{2}\frac{\vert \varrho\mathbf{u}\vert^2}{\varrho}  +\frac{1}{\varepsilon^{2m}}{H} \left(\varrho, \overline{\varrho}_{\varepsilon} \right) \right\Vert^p_{L^1_x}   \mathrm{d}\Lambda_\varepsilon(\varrho,  \varrho\mathbf{u}) \lesssim 1
\end{aligned}
\end{equation}
for all $p\in[1,\infty)$
then
\begin{equation}
\begin{aligned}
\label{almostSureConv}
\varrho_\varepsilon \rightarrow 1 \quad&\text{in}\quad L^\infty(0,T;L^{\min\{2,\gamma\}}_{\mathrm{loc}}(\mathcal{O})),
\\
\mathbf{u}_\varepsilon \rightarrow \mathbf{U} \quad&\text{in}\quad \big(L^2(0,T;W^{1,2}(\mathcal{O})) ,w\big), 
\end{aligned}
\end{equation}
in probability and where $\mathbf{U}=[\mathbf{U}_h(t,x_h),0]$ is the  unique weak pathwise solution of \eqref{2dIncom} in the sense of Definition \ref{def:pathSol} with the initial condition  $\mathbf{U}_0 = \mathbf{U}_{h,0}(x_h)$.\footnote{For a topological space $X$, we write $(X,w)$ if it is equipped with the weak topology.}
\end{thm}

\begin{rem}
We remark that, one may actually formulate  Theorem \ref{thm:mainRo} above  for the general class of solutions $\left[(\Omega^\varepsilon,\mathscr{F}^\varepsilon,(\mathscr{F}^\varepsilon_t),\mathbb{P}^\varepsilon);\varrho_\varepsilon, \mathbf{u}_\varepsilon, W_\varepsilon  \right]_{\varepsilon>0}$ rather than the restriction to $\left[(\Omega,\mathscr{F}
,(\mathscr{F}_t),\mathbb{P});\varrho_\varepsilon, \mathbf{u}_\varepsilon, W  \right]_{\varepsilon>0}$. However, without loss of generality, it always suffices to consider the latter even when the former is given. This is because the subsequent application of stochastic compactness argument due to Jakubowski \cite{jakubowski1998short} will yield respectively, the existence of some probability spaces or  a probability space. However, whatever the case may be (either the space is plural or singular ), it is shown by Jakubowski that they (it)  can be considered as the \textit{standard probability space} $\big( [0,1], \overline{\mathcal{B}([0,1])}, \mathcal{L}\big)$. So although the theorem will hold for $(\Omega^\varepsilon,\mathscr{F}^\varepsilon,(\mathscr{F}^\varepsilon_t),\mathbb{P}^\varepsilon)$, without loss of generality, it is enough to consider $(\Omega,\mathscr{F}
,(\mathscr{F}_t),\mathbb{P})$. The same loss of generality justifies considering $W$ rather than $W_\varepsilon$ for all $(\varrho_\varrho, \mathbf{u}_\varepsilon)$.
\end{rem}

\section{Uniform estimates and compactness arguments}
\label{sec:proofofmain}
This section is devoted to preparations towards the proof of our main theorem. We start by establishing a dissipative estimate for the energy of the compressible system \eqref{comprSPDE0}.
%
\subsection{Relative energy inequality and uniform bounds}
\label{subsec:energyineq}
Since the collection  $\left[(\Omega,\mathscr{F},(\mathscr{F}_t),\mathbb{P});\varrho_\varepsilon, \mathbf{u}_\varepsilon, W  \right]$ is a sequence of finite energy weak martingale solution of \eqref{comprSPDE0},  we can deduce from \eqref{relativeEnergy} that
\begin{equation}
\begin{aligned}
\label{relativeEnergyEst}
 \int_{\mathcal{O}}&\bigg[\frac{\varrho_\varepsilon}{2} \vert \mathbf{u}_\varepsilon \vert^2  
+ \frac{H(\varrho_\varepsilon, \overline{\varrho}_\varepsilon)}{\varepsilon^{2m}} \bigg](\tau)\,\mathrm{d}x
+
\int_0^\tau \int_{\mathcal{O}}\mathbb{S}(\nabla \mathbf{u}_\varepsilon):  \nabla \mathbf{u}_\varepsilon \,\mathrm{d}x  \,\mathrm{d}t
\\
&\leq
\int_{\mathcal{O}}\bigg[\frac{\vert\varrho_{\varepsilon,0} \mathbf{u}_{\varepsilon,0} \vert^2}{2\varrho_{\varepsilon,0}}  
+ \frac{H(\varrho_{\varepsilon,0}, \overline{\varrho}_\varepsilon)}{\varepsilon^{2m}} \bigg]\,\mathrm{d}x
\\&
+\frac{1}{2}\int_0^\tau\int_{\mathcal{O}}\sum_{k\in\mathbb{N}}\varrho_\varepsilon^{-1}\vert\mathbf{g}_k(x,\varrho_\varepsilon,\varrho_\varepsilon\mathbf{u}_\varepsilon)\vert^2\,\mathrm{d}x\,\mathrm{d}t
+M_{R}^\varepsilon(\tau)
\end{aligned}
\end{equation}
holds $\mathbb{P}$-a.s. with $M_{R}^\varepsilon$ given by
\begin{align}
M_R^\varepsilon(\tau)=\int_0^\tau \int_\mathcal{O} \mathbf{u}_\varepsilon \cdot \Phi(\varrho_\varepsilon,\varrho_\varepsilon\mathbf{u}_\varepsilon) \, \mathrm{d}x \, \mathrm{d}W.
\end{align}
By the mass compatibility condition
\begin{align}
\int_\mathcal{O}(\varrho_\varepsilon -\overline{\varrho}_\varepsilon) \, \mathrm{d}x=0,
\end{align}
it follows from \eqref{stochCoeffBound}--\eqref{noiseEst} that
\begin{equation}
\begin{aligned}
\label{noiseToPotential}
\sum_{k\in\mathbb{N}}&
\int_0^\tau \int_{\mathcal{O}}\frac{\varrho^{-1}_\varepsilon}{2}\vert\mathbf{g}_k(x,\varrho_\varepsilon,\varrho_\varepsilon\mathbf{u}_\varepsilon)\vert^2\,\mathrm{d}x\,\mathrm{d}t
\\
&\lesssim 
\int_0^\tau \int_{K}
\frac{1}{2}\big( \varrho_\varepsilon + \varrho_\varepsilon  \vert \mathbf{u}_\varepsilon \vert^2\big)
\mathrm{d}x\,\mathrm{d}t  
\\
&\lesssim 
\int_0^\tau \int_{\mathcal{O}}
\frac{1}{2}  \varrho_\varepsilon  \vert \mathbf{u}_\varepsilon\vert^2\, \mathrm{d}x\,\mathrm{d}t  +\int_0^\tau \int_{\mathcal{O}} {\varrho}_\varepsilon\,
\mathrm{d}x \,\mathrm{d}t 
\\
&\lesssim_\tau
\int_0^\tau \int_{\mathcal{O}}
\frac{1}{2}  \varrho_\varepsilon  \vert \mathbf{u}_\varepsilon\vert^2\, \mathrm{d}x\,\mathrm{d}t  + \int_{\mathcal{O}} \overline{\varrho}_\varepsilon\,
\mathrm{d}x.
 \end{aligned}
\end{equation}
Finally, we give a formal clarification of the apparent loss of the deterministic forcing term. In order to derive the  estimate \eqref{relativeEnergyEst}, one applies It\^o's formula to the functionals on the left-hand side of it. This is analogous to testing the momentum balance equation with the velocity vector so that in the case of rotating fluids \eqref{comprSPDE0}, one expects that the following term
\begin{equation}
\begin{aligned}
\int_0^\tau \int_{\mathcal{O}}\frac{1}{\varepsilon^{2}}\varrho_\varepsilon\nabla G\cdot  \mathbf{u}_\varepsilon \,\mathrm{d}x  \,\mathrm{d}t
\end{aligned}
\end{equation}
appears. However, assuming that all the terms below are integrable and regular enough, then we note that
\begin{equation}
\begin{aligned}
\label{centForceDissipate}
\int_0^\tau \int_{\mathcal{O}}&\frac{1}{\varepsilon^{2}}\varrho_\varepsilon\nabla G\cdot  \mathbf{u}_\varepsilon \,\mathrm{d}x  \,\mathrm{d}t
=
- \int_{\mathcal{O}}\int_0^\tau\frac{G}{\varepsilon^{2}}\mathrm{div}(\varrho_\varepsilon \mathbf{u}_\varepsilon)   \,\mathrm{d}t\,\mathrm{d}x
\\
&=
  \int_{\mathcal{O}}\frac{P'(\overline{\varrho}_{\varepsilon})}{\varepsilon^{2m}}\big(\varrho_\varepsilon(\tau) - \varrho_{\varepsilon,0}\big)   \,\mathrm{d}x
 \\
&=
  \int_{\mathcal{O}}\frac{P'(\overline{\varrho}_{\varepsilon})}{\varepsilon^{2m}}\big(\varrho_\varepsilon(\tau) -\overline{\varrho}_{\varepsilon}\big)   \,\mathrm{d}x
 -
  \int_{\mathcal{O}}\frac{P'(\overline{\varrho}_{\varepsilon})}{\varepsilon^{2m}} \big( \varrho_{\varepsilon,0} -\overline{\varrho}_{\varepsilon} \big)   \,\mathrm{d}x
\end{aligned}
\end{equation}
where we have used the continuity equation and \eqref{staticProblem1}. However, the right-hand terms in \eqref{centForceDissipate} are precisely, the first order Taylor expansion terms hidden in $H(\varrho_\varepsilon,\overline{\varrho}_\varepsilon)$ and $H(\varrho_{\varepsilon,0},\overline{\varrho}_\varepsilon)$ respectively so that in fact, information given by $G$ is captured in \eqref{relativeEnergyEst}.\\
Moving on,  by applying Gronwall's lemma, we can combine \eqref{relativeEnergyEst}, and \eqref{noiseToPotential} to get
\begin{equation}
\begin{aligned}
\label{relativeEnergyEst1}
&\sup_{t\in(0,T)}\mathcal{E} \left(\varrho_\varepsilon,\mathbf{u}_\varepsilon\left\vert \right. \overline{\varrho}_{\varepsilon}, \mathbf{0}  \right)(t,x)
+
\int_0^T \int_{\mathcal{O}}\mathbb{S}(\nabla \mathbf{u}_\varepsilon): \nabla \mathbf{u}_\varepsilon \,\mathrm{d}x  \,\mathrm{d}t
\\
&
\lesssim
1+\mathcal{E}\left(\varrho_\varepsilon,\mathbf{u}_\varepsilon\left\vert \right. \overline{\varrho}_{\varepsilon}, \mathbf{0}  \right)(0,x)
+\sup_{t\in(0,T)}\vert M_{R}^\varepsilon \vert
\end{aligned}
\end{equation}
$\mathbb{P}$-a.s.  By invoking the estimate \eqref{realMartingale}, we get by taking $p$-th moments in \eqref{relativeEnergyEst1} that
\begin{equation}
\begin{aligned}
\label{relativeEnergyEst2}
&\mathbb{E}\bigg[\sup_{t\in(0,T)}\int_\mathcal{O}
\frac{\varrho_\varepsilon}{2}\vert \mathbf{u}_\varepsilon\vert^2 (t,\cdot)\,\mathrm{d}x\bigg]^p
+
\mathbb{E}\bigg[\sup_{t\in(0,T)}\int_\mathcal{O}
\frac{1}{\varepsilon^{2m}}H(\varrho_\varepsilon,\overline{\varrho}_{\varepsilon}) (t,\cdot)\,\mathrm{d}x\bigg]^p
\\
&+
\mathbb{E}\bigg[\int_0^T \int_{\mathcal{O}}\mathbb{S}(\nabla \mathbf{u}_\varepsilon): \nabla \mathbf{u}_\varepsilon \,\mathrm{d}x  \,\mathrm{d}t\bigg]^p
\lesssim
\mathbb{E}\big[1+\mathcal{E}\left(\varrho_\varepsilon,\mathbf{u}_\varepsilon\left\vert \right. \overline{\varrho}_{\varepsilon}, \mathbf{0}  \right)(0)\big]^p. 
\end{aligned}
\end{equation}
Finally, we observe that for any such $p\in[1,\infty)$, the inequality
\begin{align*}
\mathbb{E}\big[1&+\mathcal{E}\left(\varrho_\varepsilon,\mathbf{u}_\varepsilon\left\vert \right. \overline{\varrho}_{\varepsilon}, \mathbf{0}  \right)(0,x)\big]^p
\\&
\lesssim_p 1+ \mathbb{E}\bigg(\int_{\mathcal{O}}  \bigg[\frac{\vert\varrho_{\varepsilon,0}\mathbf{u}_{\varepsilon,0} \vert^2}{2\varrho_{\varepsilon,0}}  +\frac{1}{\varepsilon^{2m}}{H} \left(\varrho_{\varepsilon,0}, \overline{\varrho}_{\varepsilon} \right) \bigg] \mathrm{d}x\bigg)^p
\\&=
c(p)\Bigg[ 1+ \int_{L_x^\gamma\times L_x^{\frac{2\gamma}{\gamma+1}}}  \left\Vert \frac{1}{2}\frac{\vert\mathbf{m}\vert^2}{\varrho}  +\frac{1}{\varepsilon^{2m}}{H} \left(\varrho, \overline{\varrho}_{\varepsilon} \right) \right\Vert^p_{L^1_x}   \mathrm{d}\Lambda_\varepsilon(\varrho,  \mathbf{m}) \Bigg]
\end{align*}
holds. Subsequently, from the boundedness assumption on the initial law (i.e., the moment estimate \eqref{momentsBounded}), we can  conclude from \eqref{relativeEnergyEst2} that for any $p\in[0,\infty)$,
\begin{equation}
\begin{aligned}
\label{entropyRel1}
\mathbb{E}\left(\sup_{t\in(0,T)}\int_{\mathcal{O}}\frac{1}{\varepsilon^{2m}}\, H(\varrho_\varepsilon,\overline{\varrho}_{\varepsilon})\,\mathrm{d}x\right)^p \, 
\lesssim_p 1
\end{aligned}
\end{equation}
and
 \begin{equation}
\begin{aligned}
\label{uniBounds1}
\mathbb{E}\left(\sup_{t\in(0,T)}\int_{\mathcal{O}}\frac{\varrho_\varepsilon}{2}\vert \mathbf{u}_\varepsilon \vert^2\,\mathrm{d}x\right)^p \, 
\lesssim_p 1,\\
\mathbb{E}\left(\int_{Q_T}\mathbb{S}( \nabla\mathbf{u}_\varepsilon):\nabla\mathbf{u}_\varepsilon\,\mathrm{d}x  \,\mathrm{d}t\,\right)^p \, 
\lesssim_p 1
\end{aligned}
\end{equation}
uniformly in $\varepsilon$. 
\\
In the following, unless otherwise specified, we shall always refer to `balls' as 3-D objects of the form
\begin{align}
\label{specialball}
B_k:= \big\{x\in\mathcal{O} \, :\, \vert x_h\vert\leq k \big\}.
\end{align}
Also, we follow \cite[Page 144]{feireisl2009singular} and define $\mathscr{O}_{\mathrm{ess}}$ and $\mathscr{O}_{\mathrm{res}}$ to be fixed subsets of  $(0,\infty)$ defined by
\begin{align*}
\mathscr{O}_{\mathrm{ess}} &:= \big\{ \varrho\in (0,\infty) \, :\,{\overline{\varrho}}/{2} < \varrho < 2\overline{\varrho}\big\}
\\
&= \big\{ \varrho\in (0,\infty) \, :\, |\varrho - 5\overline{\varrho}/4 |< 3\overline{\varrho}/4 \big\},
\\
\mathscr{O}_{\mathrm{res}} &:=(0,\infty) \setminus \mathscr{O}_{\mathrm{ess}} 
\end{align*}
respectively. Then for fixed $\omega \in \Omega$, we define the measurable subsets of $\Omega \times(0,T) \times \mathcal{O}$ by
\begin{align*}
\mathcal{M}^\varepsilon_{\mathrm{ess}} &:= \big\{ (\omega,t,x) \in \Omega \times(0,T) \times \mathcal{O} \, :\, \varrho_\varepsilon(\omega,t,x) \in \mathscr{O}_{\mathrm{ess}}  \big\},
\\
\mathcal{M}^\varepsilon_{\mathrm{res}} &:= \big( \Omega \times (0,T) \times \mathcal{O} \big) \setminus \mathcal{M}^\varepsilon_{\mathrm{ess}}
\end{align*}
respectively.
Subsequently, the decomposition of an integrable function $h$ on the random time-space cylinder into its \textit{essential} and \textit{residual} parts, i.e., 
\begin{align*}
h=[h]_\mathrm{ess} +[h]_\mathrm{res},\quad\text{where}\quad [h]_\mathrm{ess}= h\mathbbm{1}_{\mathcal{M}^\varepsilon_{\mathrm{ess}} },
\end{align*}
holds. Given the above definitions, 
%
we can use \eqref{entropyRel1} and \cite[Eq. 4.4]{breit2015compressible} to show that for $m>1+\alpha$,
\begin{equation}
\begin{aligned}
\label{uniBounds2}
\mathbb{E}\left(\sup_{t\in(0,T)}\int_{B_{k\varepsilon^{-\alpha}}} \left[\frac{ \varrho_\varepsilon -\overline{\varrho}_{\varepsilon}}{\varepsilon^m}\right]_{\mathrm{ess}}^2\,\mathrm{d}x  \right)^p\, 
\lesssim_{p,k} 1,
\\
\mathbb{E}\left(\sup_{t\in(0,T)}\int_{B_{k\varepsilon^{-\alpha}}} \big[1+ \varrho_\varepsilon^\gamma \big]_{\mathrm{res}} \,\mathrm{d}x  \right)^p\, 
\lesssim_{p,k} \varepsilon^{2m} 
\end{aligned}
\end{equation}
for balls $B_{k\varepsilon^{-\alpha}}\subset\overline{\mathcal{O}}$ of radius ${k\varepsilon^{-\alpha}}>0$. Furthermore, we can show the following lemma.
\begin{lem}
\label{lem:taylorIsentropicPressure}
For all $p\in[1,\infty)$, we have 
\begin{align*}
\mathbb{E}\left(\sup_{t\in(0,T)}\int_{B_{k\varepsilon^{-\alpha}}} \frac{ 1}{\varepsilon^{2m}}\big[\varrho_\varepsilon^\gamma -\gamma(\varrho_\varepsilon -\overline{\varrho}_{\varepsilon})  -\overline{\varrho}_{\varepsilon}^\gamma\big]\,\mathrm{d}x  \right)^p\, 
\lesssim_{p,k} 1+\varepsilon^{2(m-1-\alpha)}
\end{align*}
uniformly in $\varepsilon$.
\end{lem}
\begin{proof}
We first notice that
\begin{equation}
\begin{aligned}
\varrho_\varepsilon^\gamma -\gamma(\varrho_\varepsilon -\overline{\varrho}_{\varepsilon})  -\overline{\varrho}_{\varepsilon}^\gamma
&=
\big[\varrho_\varepsilon^\gamma -\gamma\overline{\varrho}_{\varepsilon}^{\gamma-1}(\varrho_\varepsilon -\overline{\varrho}_{\varepsilon})  -\overline{\varrho}_{\varepsilon}^\gamma\big]
\\&
+
\gamma(\overline{\varrho}_{\varepsilon}^{\gamma-1}-1)(\varrho_\varepsilon - \overline{\varrho}_{\varepsilon})
\\
&= 
(\gamma-1) H(\varrho_\varepsilon ,\overline{\varrho}_{\varepsilon})
+ 
\gamma(\overline{\varrho}_{\varepsilon}^{\gamma-1}-1)(\varrho_\varepsilon - \overline{\varrho}_{\varepsilon})
\end{aligned}
\end{equation}
where for $x\in B_{k\varepsilon^{-\alpha}}$, we get from \eqref{densityAndOne} that
\begin{align*}
\big\vert \overline{\varrho}_{\varepsilon}(x)^{\gamma-1} -1\big\vert  \lesssim \big\vert \overline{\varrho}_{\varepsilon}(x) -1\big\vert  \lesssim \varepsilon^{2(m-1-\alpha)}.
\end{align*}
The claim then follows from \eqref{entropyRel1} and \eqref{uniBounds2}.
\end{proof}

\begin{lem}
\label{uniformBounds}
For all $p\in[1,\infty)$ and ball $B\subset \mathcal{O}$, we have that
\begin{align}
&\mathbb{E} \,\bigg\vert \bigg( \int_0^T \big\Vert \mathbf{u}_\varepsilon \big\Vert_{W^{1,2}(\mathcal{O})}^2 \, \mathrm{d}t \bigg)^\frac{1}{2} \bigg\vert^p  \lesssim 1,
\label{velo}\\
&\mathbb{E} \,\bigg\vert \sup_{t\in[0,T]}  \big\Vert \sqrt{\varrho_\varepsilon} \mathbf{u}_\varepsilon \big\Vert_{L^2(\mathcal{O})}\bigg\vert^p  \lesssim 1,
\label{halfMomen}\\
&\mathbb{E} \,\Bigg\vert \sup_{t\in[0,T]}  \bigg\Vert \frac{\varrho_\varepsilon - \overline{\varrho}_{\varepsilon}}{\varepsilon^m} \bigg\Vert_{L^{\min\{2,\gamma\}}(B)}\Bigg\vert^p  \lesssim 1,
\label{limDense} 
\end{align}
and
\begin{align}
&\mathbb{E} \,\bigg\vert \sup_{t\in[0,T]}  \big\Vert \varrho_\varepsilon \big\Vert_{L^\gamma(B)}\bigg\vert^p  \lesssim 1,
\label{densegammaRot}\\
&\mathbb{E} \,\bigg\vert \sup_{t\in[0,T]}  \big\Vert \varrho_\varepsilon \mathbf{u}_\varepsilon \big\Vert_{L^\frac{2\gamma}{\gamma+1}(B)}\bigg\vert^p  \lesssim 1,
\label{momentum}\\
&\mathbb{E} \,\bigg\vert \bigg( \int_0^T \big\Vert \varrho_\varepsilon \mathbf{u}_\varepsilon\otimes  \mathbf{u}_\varepsilon \big\Vert_{L^\frac{6\gamma}{4\gamma+3}(B)}^2 \, \mathrm{d}t \bigg)^\frac{1}{2} \bigg\vert^p  \lesssim 1,\label{convectiv}
\end{align}
uniformly in $\varepsilon$.
\end{lem}
\begin{proof}
The first two follows immediately from \eqref{uniBounds1} and a version of Korn's inequality \cite[Theorem 10.17]{feireisl2009singular}, c.f. \cite[Eq. 2.20]{feireisl2012multi}. The bound \eqref{limDense} follows from \eqref{uniBounds2} and \eqref{densityAndOne}.
The last three \eqref{densegammaRot}--\eqref{convectiv} can be found in \cite{mensah2016existence}[Eq. 23].
\end{proof}

\begin{lem}
\label{lem:strongDensity}
For all $p\in[1,\infty)$, we have that
\begin{align*}
\varrho_\varepsilon \, \rightarrow \, 1
&\text{ in } L^p(\Omega;L^\infty(0,T;L^{\min\{2,\gamma\}}_{\mathrm{loc}}(\mathcal{O})))
\end{align*}
as $\varepsilon\rightarrow0$.
\end{lem}
\begin{proof}
This is a direct consequence of \eqref{limDense} and \eqref{densityAndOne}.
\end{proof}
Now  let set $\mathrm{Ma}=\varepsilon^m$,  $\mathrm{Ro}=\mathrm{Fr}=\varepsilon$. Then  we observe that by setting $r_\varepsilon= \frac{\varrho_\varepsilon-\overline{\varrho}_{\varepsilon}}{\varepsilon^m}$, \footnote{This quantity is sometimes referred to as the density fluctuation.}  we derive from Eq. \eqref{comprSPDE0} the following:
\begin{equation}
\begin{aligned}
\label{acousticSPD00}
 \varepsilon^m\mathrm{d}r_\varepsilon &+  \mathrm{div}\,\left(\varrho_\varepsilon\mathbf{u}_\varepsilon\right) \mathrm{d}t   =0,
\\
\varepsilon^m \mathrm{d}\left(\varrho_\varepsilon\mathbf{u}_\varepsilon\right) +\big[\varepsilon^{m-1}(\mathbf{e}_3&\times \varrho_\varepsilon\mathbf{u}_\varepsilon)
 +
 \gamma \nabla r_\varepsilon \big]\mathrm{d}t
 = 
 \varepsilon^m \textbf{F}_\varepsilon\,\mathrm{d}t
 \\
 &+
 \varepsilon^{2(m-1)}r_\varepsilon\nabla G\,\mathrm{d}t
+ \varepsilon^m \Phi(\varrho_\varepsilon, \varrho_\varepsilon\mathbf{u}_\varepsilon)\mathrm{d}W  
\end{aligned}
\end{equation}
in the sense of distributions and where we have used \eqref{staticProblem},
\begin{align}
\label{allForces}
\textbf{F}_\varepsilon  &:=  \mathrm{div}\left(\mathbb{S}(\nabla\mathbf{u}_\varepsilon)\right)
-
\mathrm{div}( \varrho_\varepsilon\mathbf{u}_\varepsilon\otimes \mathbf{u}_\varepsilon)
- \frac{1}{\varepsilon^{2m}}\nabla[\varrho_\varepsilon^\gamma -\gamma(\varrho_\varepsilon -\overline{\varrho}_{\varepsilon})  -\overline{\varrho}_{\varepsilon}^\gamma ] 
\end{align}
and for any $K\Subset\mathcal{O}$
\begin{align}
\label{allForces1}
\textbf{F}_\varepsilon  \in L^p\left( \Omega; L^2\left(0,T;W^{-l,2}(K) \right)\right)
\end{align}
uniformly in $\varepsilon$ for $l>5/2$. The uniform estimate \eqref{allForces1} follows from Lemma \ref{lem:taylorIsentropicPressure}, \eqref{velo} and \eqref{convectiv}  and is similar to the proof of \cite[Eq. 44]{mensah2016existence}.
\\
Finally, one can also infer from \eqref{limDense} and   \eqref{centrifugal} that
\begin{align}
\label{limDenseAndCentri}
 r_\varepsilon\nabla G &\in
 L^p(\Omega;L^\infty(0,T;L_{\mathrm{loc}}^{\min\{2,\gamma\}}(\mathcal{O}))).
\end{align}

\subsection{Analysis of the Coriolis term}
\label{sec:Coriolis}
We wish to show in this section that the Coriolis term is a gradient vector field provided we consider its vertical average. cf. \cite[Section 3]{feireisl2012multi} and \cite[Section 3]{gallagher2006weak}. To see this, let first consider the following notation: 
\begin{align}
\label{verticalAverage}
\ulcorner {g}\urcorner=\fint_{\mathbb{T}_1} {g}\,\mathrm{d}x_3
=\frac{1}{\vert\mathbb{T}_1\vert}\int_{\mathbb{T}_1} g\,\mathrm{d}x_3
\end{align}
for any function $g$ defined on $\mathcal{O}$. Then we observe that if we set ${\mathbf{Y}}_{\varepsilon}:=\mathcal{P}\left({\varrho}_{\varepsilon}{\mathbf{u}}_{\varepsilon}\right)$, we have that
 $\mathrm{div}( \ulcorner{\mathbf{Y}}_{\varepsilon}\urcorner)
=\partial_{x_1}\ulcorner {Y}^1_{\varepsilon}\urcorner+ \partial_{x_2}\ulcorner {Y}^2_{\varepsilon}\urcorner=0$. As such,
\begin{align}
\label{curl}
\mathrm{curl}\,\left(\mathbf{e}_3\times\left\ulcorner {\mathbf{Y}}_{\varepsilon} \right\urcorner \right)=\left(0, 0,\partial_{x_1}\left\ulcorner {Y}^1_{\varepsilon}\right\urcorner+ \partial_{x_2}\left\ulcorner {Y}^2_{\varepsilon}\right\urcorner \right)=\bm{0}.
\end{align}
\begin{rem}
It is crucial at this point to consider the vertical average of the solenoidal part of momentum since otherwise, the curl of the Coriolis term for the full momentum is not zero. This taking of the vertical average is a reason why we require the special geometry $\mathcal{O}$ rather than the whole space $\mathbb{R}^3$.
\end{rem}

We also observe that for any potential $\Psi$, the following identity
\begin{align}
\label{commutativity} 
\ulcorner \mathbf{e}_3\times \nabla \Psi \urcorner
=
 \mathbf{e}_3\times \ulcorner\nabla \Psi\urcorner
 =
 \mathbf{e}_3\times \nabla\ulcorner \Psi\urcorner
\end{align}
holds. Subsequently, we will use any of the identities in \eqref{commutativity} interchangeably throughout the rest of this paper.


\subsection{Compactness}
\label{subsec:compactness}As in \cite[Eq. 3.3]{feireisl2012multi}, we introduce the smooth family of cut-off functions $\eta_\varepsilon$ satisfying
\begin{equation}
\begin{aligned}
\label{newCufOff}
&\eta_\varepsilon \in C^\infty_c(\mathbb{R}^2), \quad 0\leq \eta_\varepsilon \leq 1, \quad \eta_\varepsilon (x_h)\equiv1 \text{ in }  B_{\varepsilon^{-\alpha}}, 
\\
&\eta_\varepsilon (x_h)=0 \text{ if }  \vert x_h\vert \geq 2\varepsilon^{-\alpha},\quad  \big\vert \nabla \eta_\varepsilon (x_h) \big\vert \leq 2 \varepsilon^\alpha  \text{ for }x_h\in\mathbb{R}^2
\end{aligned}
\end{equation}
where since $m>10$,  we can choose $\alpha$ in \eqref{newCufOff} such that
\begin{align}
\label{alpha}
1+\frac{m}{2}<\alpha< \frac{3m}{4}-\frac{3}{2}.
\end{align}
To explore compactness, let define the following spaces:
\begin{align*}
\chi_{\ulcorner \varrho\mathbf{u}\urcorner} &= C_w\left([0,T];L^{\frac{2\gamma}{\gamma+1}}_{\mathrm{loc}}(\mathcal{O})\right), 
\\
\chi_\mathbf{u} &= \left(L^2(0,T;W^{1,2}(\mathcal{O})),w\right),
\\
\chi_\varrho &= C_w\left([0,T];L^\gamma_{\mathrm{loc}}(\mathcal{O})\right),
\\
\chi_W &= C\left([0,T];\mathfrak{U}_0\right) ,
\end{align*}
and let
\begin{enumerate}
\item $\mu_{\ulcorner\mathcal{P}(\eta_\varepsilon\varrho_\varepsilon\mathbf{u}_\varepsilon)\urcorner}$ be the law of $\ulcorner \mathcal{P}\,(\eta_\varepsilon\varrho_\varepsilon\mathbf{u}_\varepsilon)\urcorner$ on the space $\chi_{ \ulcorner \varrho\mathbf{u}\urcorner}$,
\item $\mu_{\mathbf{u}_\varepsilon}$ be the law of $\mathbf{u}_\varepsilon$ on $\chi_{\mathbf{u}}$,
\item $\mu_{\varrho_\varepsilon}$ be the law of $\varrho_\varepsilon$ on the space $\chi_{\varrho}$,
\item $\mu_{W}$ be the law of $W$ on the space $\chi_{W}$.
\end{enumerate}
Now we let $\mu^{\varepsilon,\delta}$ and $\nu^{\varepsilon,\delta}$ be the joint laws of 
\begin{align*}
(\varrho_\varepsilon,\mathbf{u}_\varepsilon,\ulcorner\mathcal{P}\,(\eta_\varepsilon\varrho_\varepsilon\mathbf{u}_\varepsilon)\urcorner, \varrho_\delta,\mathbf{u}_\delta,\ulcorner\mathcal{P}\,(\eta_\delta\varrho_\delta\mathbf{u}_\delta)\urcorner)
\end{align*}
and 
\begin{align*}
(\varrho_\varepsilon,\mathbf{u}_\varepsilon,\ulcorner\mathcal{P}\,(\eta_\varepsilon\varrho_\varepsilon\mathbf{u}_\varepsilon)\urcorner, \varrho_\delta,\mathbf{u}_\delta,\ulcorner\mathcal{P}\,(\eta_\delta\varrho_\delta\mathbf{u}_\delta)\urcorner, W) 
\end{align*}
respectively on the path space $\chi = \chi_\varrho \times \chi_\mathbf{u} \times \chi_{\ulcorner\varrho\mathbf{u}\urcorner} \times\chi_\varrho \times \chi_\mathbf{u} \times \chi_ {\ulcorner\varrho\mathbf{u}\urcorner}$ and $\chi^J = \chi \times \chi_W$ respectively.


\begin{lem}
\label{lem:tighness}
The collection $\{ \mu_{\ulcorner\mathcal{P}(\eta_\varepsilon\varrho_\varepsilon\mathbf{u}_\varepsilon)\urcorner} \, ; \, \varepsilon \in(0,1) \}$ is tight on $\chi_{\ulcorner\varrho\mathbf{u}\urcorner}$. 
\end{lem}
\begin{proof}
We have shown in Section \ref{sec:Coriolis} that the vertical average of the Coriolis term is curl-free meaning that it is a gradient vector. We now combine the approach of  \cite[Proposition 3.6]{breit2015incompressible} and \cite[Section 5.4.2]{feireisl2009singular} and consider the projection of \eqref{acousticSPD00} onto solenoidal fields. This is done by the special choice of divergence-free test function $\mathcal{P}\bm{\phi} $, $\bm{\phi} \in C^\infty_{c} (\mathbb{R}^3)$. With this test function, we get by integration by part that the vertical average of the distributional form of the momentum equation \eqref{acousticSPD00} is
\begin{equation}
\begin{aligned}
\ulcorner &\big\langle (\eta_\varepsilon\varrho_\varepsilon \mathbf{u}_\varepsilon)(t), \mathcal{P}\bm{\phi} \big\rangle \urcorner 
= \ulcorner  \big\langle \eta_\varepsilon\varrho_{\varepsilon,0}\mathbf{u}_{\varepsilon,0}, \mathcal{P}\bm{\phi} \big\rangle \urcorner
\\&+
 \int_0^{t} \ulcorner  \big\langle\eta_\varepsilon( \varrho_\varepsilon \mathbf{u}_\varepsilon \otimes \mathbf{u}_\varepsilon), \nabla \mathcal{P}\bm{\phi} \big\rangle \urcorner\,\mathrm{d}s
-
\int_0^{t} \ulcorner \big\langle\nu\,\eta_\varepsilon\nabla \mathbf{u}_\varepsilon\,, \nabla \mathcal{P}\bm{\phi} \big\rangle \urcorner \mathrm{d}s 
\\& 
+
 \int_0^{t}\ulcorner  \big\langle \varepsilon^{m-2}\eta_\varepsilon(r_\varepsilon\nabla G) ,\mathcal{P} \bm{\phi} \big\rangle \urcorner\,\mathrm{d}s
+
 \int_0^{t}\ulcorner  \big\langle \mathcal{R}_\varepsilon , \mathcal{P}\bm{\phi} \big\rangle \urcorner\,\mathrm{d}s
\\&
 +
 \int_0^{t} \ulcorner \big\langle \eta_\varepsilon\Phi(\varrho_\varepsilon ,\varrho_\varepsilon \mathbf{u}_\varepsilon), \mathcal{P}\bm{\phi} \big\rangle \urcorner\,\mathrm{d}W
\end{aligned}
\end{equation}
$\mathbb{P}$-a.s. for all $t\in [0,T]$ and where
\begin{equation}
\begin{aligned}
\label{residualLLower}
\mathcal{R}_\varepsilon
=
\nabla\eta_\varepsilon \cdot ( \varrho_\varepsilon \mathbf{u}_\varepsilon \otimes \mathbf{u}_\varepsilon)
-
\nabla\eta_\varepsilon \cdot \nu\,\mathbf{u}_\varepsilon.
\end{aligned}
\end{equation}
To proceed, we make the following denotations:
\begin{align*}
\ulcorner T_\varepsilon(t) \urcorner
&:= 
 \ulcorner  \big\langle \eta_\varepsilon\varrho_{\varepsilon,0}\mathbf{u}_{\varepsilon,0}, \mathcal{P}\bm{\phi} \big\rangle \urcorner
+
 \int_0^{t} \ulcorner  \big\langle\eta_\varepsilon( \varrho_\varepsilon \mathbf{u}_\varepsilon \otimes \mathbf{u}_\varepsilon), \nabla \mathcal{P}\bm{\phi} \big\rangle \urcorner\,\mathrm{d}s
\\&
-
\int_0^{t} \ulcorner \big\langle\nu\,\eta_\varepsilon\nabla \mathbf{u}_\varepsilon\,, \nabla \mathcal{P}\bm{\phi} \big\rangle \urcorner \mathrm{d}s 
+
 \int_0^{t}\ulcorner  \big\langle \varepsilon^{m-2}\eta_\varepsilon(r_\varepsilon\nabla G) ,\mathcal{P} \bm{\phi} \big\rangle \urcorner\,\mathrm{d}s
\\&
+
 \int_0^{t}\ulcorner  \big\langle \mathcal{R}_\varepsilon , \mathcal{P}\bm{\phi} \big\rangle \urcorner\,\mathrm{d}s
\\
\ulcorner R_\varepsilon(t)\urcorner&:= \int_0^{t} \ulcorner \big\langle \eta_\varepsilon\Phi(\varrho_\varepsilon ,\varrho_\varepsilon \mathbf{u}_\varepsilon), \mathcal{P}\bm{\phi} \big\rangle \urcorner\,\mathrm{d}W
\end{align*}
for all $t\in[0,T]$. Now consider any compact set $K=\overline{K}\times \mathbb{T}_1$. Then by using the continuity of the operator $\mathcal{P}$, \eqref{convectiv} and the continuous embedding $W^{-1,\frac{6\gamma}{4\gamma+3}}(K)\hookrightarrow  W^{-l,2}(K)$ which holds true provided $l>\frac{5}{2}$, we get   that
\begin{equation}
\begin{aligned}
\label{convT}
\ulcorner \mathcal{P}\,\mathrm{div}(\eta_\varepsilon\varrho_\varepsilon\mathbf{u}_\varepsilon\otimes \mathbf{u}_\varepsilon)  \urcorner\in L^p\left( \Omega;L^2\left(0,T; W^{-l,2}(K) \right) \right), \quad l>\frac{5}{2},
\end{aligned}
\end{equation}
uniformly in $\varepsilon$ for all $p\in[1,\infty)$. Furthermore, by using \eqref{velo} and  the continuity of $\mathcal{P}$, we also get that
\begin{equation}
\begin{aligned}
\label{velT}
\ulcorner \mathcal{P}\,(\nu \, \mathrm{div}(\eta_\varepsilon\nabla \mathbf{u}_\varepsilon) \urcorner \in L^p\left( \Omega;L^2\left(0,T; W^{-1,2}(K) \right) \right)
\end{aligned}
\end{equation}
uniformly in $\varepsilon$ for all $p\in[1,\infty)$. Also, the continuous embedding\\ $L^\infty(0,T;L^1(K))$ $\hookrightarrow L^2(0,T;W^{-l,2}(K))$  gives 
\begin{equation}
\begin{aligned}
\label{centT}
\ulcorner \mathcal{P}\,(\eta_\varepsilon r_\varepsilon\nabla  G) \urcorner \in L^1\left( \Omega;L^2\left(0,T; W^{-l,2}(K) \right) \right).
\end{aligned}
\end{equation}
Indeed, by using the aforementioned embedding, the continuity of $\mathcal{P}$, \eqref{centrifugal}, H\"{o}lder's inequality and \eqref{limDense},  we have that
\begin{equation}
\begin{aligned}
\mathbb{E} &\Vert  \ulcorner \mathcal{P}(\eta_\varepsilon r_\varepsilon\nabla  G) \urcorner \Vert_{L^2\left(0,T; W^{-l,2}(K) \right)}
\\&
\lesssim
\mathbb{E} \bigg(
\Vert \nabla G\Vert_{L^\infty_x}\sup_{t\in(0,T)} \Vert  \ulcorner r_\varepsilon \urcorner \Vert_{ L^1(K)}\bigg)
\\
&\lesssim\,\mathbb{E}\, 
\sup_{t\in(0,T)} \Vert  r_\varepsilon  \Vert_{ L^{\min\{2,\gamma\}}(K)}\lesssim 1
\end{aligned}
\end{equation}
uniformly in $\varepsilon
$. The residual term \eqref{residualLLower} is comparable to \eqref{convT}, \eqref{velT} and   is in fact, of lower order. Subsequently, by combining \eqref{convT}, \eqref{velT} and \eqref{centT}, it follows that
\begin{equation}
\begin{aligned}
\label{tT}
\partial_t\ulcorner T_\varepsilon(t)\urcorner\in L^1\left( \Omega;L^2\left(0,T; W^{-l,2}(K) \right) \right)
\end{aligned}
\end{equation}
uniformly in $\varepsilon$ since $\varepsilon^{m-2}\ll 1$ for $m>3$. 
It follows from \eqref{tT}  that  the mean
\begin{align}
\label{kolmogorovEst1}
\mathbb{E}\, \left\Vert \,\ulcorner T_\varepsilon\urcorner \, \right\Vert_{C^\vartheta\left([0,T];W^{-l,2}(K)  \right)}\lesssim 1
\end{align}
is bounded uniformly in $\varepsilon$ for $\vartheta\in[0,\frac{1}{2}]$. 
Again, much like the proof of \cite[Proposition 3.6]{breit2015incompressible}, one gets by using \eqref{stochCoeffBound} and \eqref{noiseEst} that
\begin{align*}
\mathbb{E}&\, \left\Vert   \ulcorner R_\varepsilon(t)\urcorner  -   \ulcorner R_\varepsilon(s)\urcorner\right\Vert^\theta_{W^{-l,2}(K)}
\lesssim
\mathbb{E}\, \left\Vert  \int_s^t  \Phi\left( \varrho_\varepsilon, \varrho_\varepsilon\mathbf{u}_\varepsilon \right)  \,\mathrm{d}W  \right\Vert^\theta_{W^{-l,2}(K)}
\\
&\lesssim
\mathbb{E}\Bigg(
 \int_s^t \sum_{k\in\mathbb{N}}  \left\Vert \mathbf{g}_k\left( x,\varrho_\varepsilon, \varrho_\varepsilon\mathbf{u}_\varepsilon \right)       \right\Vert^2_{W^{-l,2}(K)}\,\mathrm{d}\tau
\Bigg)^\frac{\theta}{2}
\\
&\lesssim
\mathbb{E}\Bigg(
 \int_s^t \sum_{k\in\mathbb{N}}  \left\Vert  \mathbf{g}_k \left( x,\varrho_\varepsilon, \varrho_\varepsilon\mathbf{u}_\varepsilon \right)      \right\Vert^2_{L^1(K)}\,\mathrm{d}\tau
\Bigg)^\frac{\theta}{2}
\\
&\lesssim
\mathbb{E}\Bigg(
 \int_s^t  \int_{{K}} (1  + \varrho^\gamma_\varepsilon + \varrho _\varepsilon \vert \mathbf{u}_\varepsilon \vert^2)\mathrm{d}x\,\mathrm{d}\tau
\Bigg)^\frac{\theta}{2}
\\
&\lesssim \vert t-s\vert^\frac{\theta}{2}
\Bigg(
1  + \mathbb{E}\sup_{t\in[0,T]}\Vert\varrho_\varepsilon\Vert_{L^\gamma(K)}^{\theta\gamma/2} + \mathbb{E}\sup_{t\in[0,T]}\Vert\sqrt{\varrho}_\varepsilon  \mathbf{u}_\varepsilon \Vert_{L^2(K)}^{\theta}
\Bigg)^\frac{\theta}{2}
\\&
\lesssim \vert t-s\vert^\frac{\theta}{2}.
\end{align*}
In the last estimate above, we have used \eqref{halfMomen} and \eqref{densegammaRot}. We now apply Kolmogorov's continuity criterion and then combining with \eqref{kolmogorovEst1}, we get that
\begin{align}
\label{kolmogorovEst2}
\mathbb{E}\, \left\Vert \ulcorner \mathcal{P}\,(\eta_\varepsilon \varrho_\varepsilon\mathbf{u}_\varepsilon)(t)\urcorner \right\Vert_{C^\vartheta\left([0,T];W^{-l,2}(K)  \right)}\lesssim 1.
\end{align}
Finally, we use the compact embedding (see \cite[Corollary B.2]{ondrejat2010stochastic})
\begin{equation}
\begin{aligned}
\label{kolmogorovEst2aa}
L^\infty\big( 0,T;L^\frac{2\gamma}{\gamma+1}(K) \big)\cap C^\vartheta \big([0,T];W^{-l,2}(K)\big) \hookrightarrow 
C_w\big([0,T];L^\frac{2\gamma}{\gamma+1}(K)\big)
\end{aligned}
\end{equation}
and \eqref{momentum} to finish the proof.
\end{proof}

\begin{lem}
The collection $\{  \nu^{\varepsilon,\delta} \, ; \, \varepsilon,\delta\in(0,1) \}$ is tight on $\chi^J$. 
\end{lem}
\begin{proof}
This is similar \cite[Lemma 6]{mensah2016existence} or \cite[Corollary 3.7]{breit2015incompressible}.
\end{proof}

\begin{prop}[Jakubowski-Skorokhod representation theorem]
\label{prop:Jakubow0}
For any subsequence $\{\nu^{\varepsilon_n,\delta_n}\,;\,n\in\mathbb{N}\}$, there exists a further subsequence (not relabelled), a probability space $(\tilde{\Omega}, \tilde{\mathscr{F}}, \tilde{\mathbb{P}})$ with $\chi^J$-valued  random variables 
\begin{align*}
(\hat{\varrho}, \hat{\mathbf{U}},\hat{\mathbf{m}}, \check{\varrho},\check{\mathbf{U}}, \check{\mathbf{m}} , \tilde{W})\text{ and }  (\hat{\varrho}_{\varepsilon_n}, \hat{\mathbf{u}}_{\varepsilon_n},\hat{\mathbf{m}}_{\varepsilon_n}, \check{\varrho}_{\delta_n},\check{\mathbf{u}}_{\delta_n}, \check{\mathbf{m}}_{\delta_n}, \tilde{W}_n),\quad n\in\mathbb{N}
\end{align*}
and $\varepsilon_n,\delta_n\in(0,1)$ such that as $\varepsilon_n,\delta_n\rightarrow0$ (corresponding to when $n\rightarrow\infty$), we have
\begin{enumerate}
\item $ \tilde{\mathbb{P}}((\hat{\varrho}_{\varepsilon_n}, \hat{\mathbf{u}}_{\varepsilon_n},\hat{\mathbf{m}}_{\varepsilon_n}, \check{\varrho}_{\delta_n},\check{\mathbf{u}}_{\delta_n}, \check{\mathbf{m}}_{\delta_n},\tilde{W}_n)\in \cdot)= \nu^{\varepsilon_n,\delta_n}(\cdot)$,
\item $ \tilde{\mathbb{P}}((\hat{\varrho}, \hat{\mathbf{U}},\hat{\mathbf{m}}, \check{\varrho},\check{\mathbf{U}}, \check{\mathbf{m}},\tilde{W})\in \cdot)= \nu(\cdot)$ is a Radon measure,
\item the sequences $(\hat{\varrho}_{\varepsilon_n}, \hat{\mathbf{u}}_{\varepsilon_n},\hat{\mathbf{m}}_{\varepsilon_n}, \check{\varrho}_{\delta_n},\check{\mathbf{u}}_{\delta_n}, \check{\mathbf{m}}_{\delta_n}, \tilde{W}_n)$ converges $\tilde{\mathbb{P}}$-a.s. to \newline $(\hat{\varrho}, \hat{\mathbf{U}},\hat{\mathbf{m}}, \check{\varrho},\check{\mathbf{U}}, \check{\mathbf{m}}, \tilde{W})$
in the topology of $\chi^J$.
\end{enumerate}
In particular,  the joint law of $(\hat{\varrho}_{\varepsilon_n}, \hat{\mathbf{u}}_{\varepsilon_n},\hat{\mathbf{m}}_{\varepsilon_n}, \check{\varrho}_{\delta_n},\check{\mathbf{u}}_{\delta_n}, \check{\mathbf{m}}_{\delta_n})$, i.e. $\mu^{\varepsilon_n,\delta_n}$, converges weakly to the measure $\mu(\cdot) =\tilde{\mathbb{P}}((\hat{\varrho}, \hat{\mathbf{U}},\hat{\mathbf{m}}, \check{\varrho},\check{\mathbf{U}}, \check{\mathbf{m}})\in \cdot)$.
\end{prop}

To extend this new probability space $(\tilde{\Omega}, \tilde{\mathscr{F}}, \tilde{\mathbb{P}})$ into a stochastic basis, we endow it with a filtration. To do this, let us first define a restriction operator $\textbf{r}_t$ by
\begin{align}
\label{continuousFunction}
\textbf{r}_t:X\rightarrow X\vert_{[0,t]}, \quad f\mapsto f\vert_{[0,t]},
\end{align}
for $t\in[0,T]$ and $X\in  \{ \chi_\varrho, \chi_{\mathbf{u}},  \chi_W \}$. We observe that $\textbf{r}_t$ is a continuous map. We can therefore construct $\tilde{\mathbb{P}}$-augmented canonical filtrations for 
\begin{align*}
(\hat{\varrho}_{\varepsilon_n}, \hat{\mathbf{u}}_{\varepsilon_n},
 \check{\varrho}_{\delta_n},\check{\mathbf{u}}_{\delta_n}, 
 \tilde{W}_n) 
 \text{ and }
  (\hat{\varrho}, \hat{\mathbf{U}},
  \check{\varrho},\check{\mathbf{U}}, 
   \tilde{W})
\end{align*}
 respectively by setting
\begin{align*}
\tilde{\mathscr{F}}^n_t &= \sigma\left( \sigma(\textbf{r}_t\hat{\varrho}_{\varepsilon_n}, \textbf{r}_t\hat{\mathbf{u}}_{\varepsilon_n},
 \textbf{r}_t\check{\varrho}_{\delta_n},\textbf{r}_t\check{\mathbf{u}}_{\delta_n}, 
  \textbf{r}_t\tilde{W}_n)  \cup  \{N\in\tilde{\mathscr{F}};\,\tilde{\mathbb{P}}(N)=0  \} \right),
\\
\tilde{\mathscr{F}}_t &= \sigma\left( \sigma(\textbf{r}_t\hat{\varrho}, \textbf{r}_t\hat{\mathbf{U}},
 \textbf{r}_t\check{\varrho},\textbf{r}_t\check{\mathbf{U}}, 
 \textbf{r}_t\tilde{W})  \cup  \{N\in\tilde{\mathscr{F}};\,\tilde{\mathbb{P}}(N)=0  \} \right),
\end{align*}
for $t\in[0,T]$.

\subsection{Identification of the limit}
\label{subsec:limits}
Having established the limits of the family of sequences in Proposition \ref{prop:Jakubow0}, we now identify them with  weak martingale solutions of \eqref{2dIncom}. In fact, as a consequence of the $2$-D uniqueness theorem given in Theorem \ref{thm:2duniqueness}, we show that the corresponding random variables coincides. We state this in the theorem below.

\begin{thm}
\label{the:important}
The pair $
 [(\tilde{\Omega}, \tilde{\mathscr{F}}_t,(\tilde{\mathscr{F}}_t),\tilde{\mathbb{P}}),\hat{\mathbf{U}}, \tilde{W}]$ and $[(\tilde{\Omega}, \tilde{\mathscr{F}}_t,(\tilde{\mathscr{F}}_t),\tilde{\mathbb{P}}),\check{\mathbf{U}}, \tilde{W}]$
are each a weak martingale solution of \eqref{2dIncom} in the sense of Definition \ref{def:weakSol}  defined on the same stochastic basis. Moreover
\begin{align}
\label{uniqueRV}
\hat{\varrho}=\check{\varrho}=1,\quad \hat{\mathbf{m}}=\hat{\mathbf{U}}_h,\quad \check{\mathbf{m}}=\check{\mathbf{U}}_h
\end{align}
$\tilde{\mathbb{P}}$-a.s. where $\hat{\mathbf{U}}_h=\hat{\mathbf{U}}_h(x_h)$ satisfies $\hat{\mathbf{U}}=[\hat{\mathbf{U}}_h(x_h),0]$ and similarly for $\check{\mathbf{U}}$.
\end{thm}

The proof of Theorem \ref{the:important} follows several steps. First of all, we show that on the new probability
space $(\tilde{\Omega},\tilde{\mathscr{F}_t},\tilde{\mathbb{P}})$, any pair of approximate subsequence of functions $\left(\hat{\varrho}_{\varepsilon_n},\hat{\mathbf{u}}_{\varepsilon_n}\right)$ and $\left( \check{\varrho}_{\delta_n},  \check{\mathbf{u}}_{\delta_n}\right)$ also solves the system \eqref{comprSPDE0}.

\begin{lem}
\label{lem:WeinerSeq}
The following subsequence which are defined on the same stochastic basis
\begin{align*}
 [ (\tilde{\Omega}, \tilde{\mathscr{F}}_t,(\tilde{\mathscr{F}}^n_t),\tilde{\mathbb{P}}),\hat{\varrho}_{\varepsilon_n}, \hat{\mathbf{u}}_{\varepsilon_n}, \tilde{W}_n ] \text{ and } [ (\tilde{\Omega}, \tilde{\mathscr{F}}^n_t,(\tilde{\mathscr{F}}_t),\tilde{\mathbb{P}}),\check{\varrho}_{\delta_n},\check{\mathbf{u}}_{\delta_n}, \tilde{W}_n ]
\end{align*} 
are each a finite energy weak martingale solutions of \eqref{comprSPDE0} with initial law $\Lambda_n$
\end{lem}
\begin{proof}
This follows from the equality of laws
from Proposition \ref{prop:Jakubow0} and a general theorem given in \cite[Theorem 2.9.1]{breit2017stoch}.
\end{proof}

Until stated otherwise, we now concentrate our attention on the analysis of the sequence $ \big(\hat{\varrho}_{\varepsilon_n}, \hat{\mathbf{u}}_{\varepsilon_n}, \tilde{W}_{n}\big)$, which as stated earlier, shares the same stochastic basis with the sequence $\big(\check{\varrho}_{\delta_n},\check{\mathbf{u}}_{\delta_n}, \tilde{W}_{n}\big)$. Analysis of the latter is mere repetition.

\section{Proof of Theorem \ref{thm:mainRo}}
\label{sec:mainThmContin}
In the previous section, we have obtained two limits of a pair of sequences that have been shown to be weak martingale solutions to \eqref{2dIncom} in the sense of Definition \ref{def:weakSol}. In this section, we primarily wish to apply the $2$-D uniqueness Theorem \ref{thm:2duniqueness} to obtain  convergence of our original sequence to a limit random variable that will solve \eqref{2dIncom} in the pathwise sense given by Definition  \ref{def:pathSol}.

To do this however, we first give a rigorous analysis into the momentum function by studying its various components. We will then follow this by passing to the limit in the convective term before finally, completing the proof of our main theorem.

\subsection{Acoustic equation and its Strichartz estimates}
\label{subsec:acoustic}
In this section, we establish dispersive estimates for the acoustic wave equation obtained by projecting the vector quantities in the momentum balance equation  unto gradient vector fields via Helmholtz decomposition.
Our main result in this section is the proof of the following lemma.

\begin{lem}
Let $\Delta^{-1}_{\mathcal{O}}$ represent the inverse of the Laplace operator on $\mathcal{O}=\mathbb{R}^2\times\mathbb{T}_1$, let $\eta_\varepsilon$ be as defined in \eqref{newCufOff} and set $\mathcal{Q}= \nabla\Delta^{-1}_{\mathcal{O}}\mathrm{div}$.
Then there exist a subsequence (not relabelled) such that the following $\tilde{\mathbb{P}}$-a.s convergence holds
\begin{align}
\label{acousticMomen}
\mathcal{Q}(\eta_\varepsilon \hat{\varrho}_{\varepsilon_n}\hat{\mathbf{u}}_{\varepsilon_n}) \, \rightarrow \, 0 &\quad\text{in}\quad L^2(0,T;L^\frac{2\gamma}{\gamma+1}_{\mathrm{loc}}(\mathcal{O})), 
\end{align}
as $n\rightarrow \infty$. 
\end{lem}
\begin{proof}
Here, we follow the approach of \cite[Section 3.2.1]{feireisl2012multi} applied to the mild form of the mass and momentum balance equation. 
\\Given Proposition \ref{prop:Jakubow0} and Lemma \ref{lem:WeinerSeq}, we expect $(\hat{\varrho}_{\varepsilon_n},\hat{\mathbf{u}}_{\varepsilon_n})$ to be a weak solution of \eqref{acousticSPD00}. By multiplying the continuity equation \eqref{acousticSPD00}$_1$ with the cut-off function $\eta_\varepsilon $ introduced in \eqref{newCufOff}, we expect $(\hat{\varrho}_{\varepsilon_n},\hat{\mathbf{u}}_{\varepsilon_n})$ to be a distributional solution of
\begin{equation}
\begin{aligned}
\label{contEqCutoff}
\varepsilon^m \,\mathrm{d}\big[\eta_\varepsilon  \hat{r}_{\varepsilon_n}\big] +  \mathrm{div}\big[\eta_\varepsilon  (\hat{\varrho}_{\varepsilon_n} \hat{\mathbf{u}}_{\varepsilon_n})\big]\mathrm{d}t   &=\nabla \eta_\varepsilon \cdot \big(\hat{\varrho}_{\varepsilon_n} \hat{\mathbf{u}}_{\varepsilon_n}\big)\,\mathrm{d}t.
\end{aligned}
\end{equation}
However by using \eqref{newCufOff}, in particular $\vert \nabla \eta_\varepsilon (x_h) \vert \leq 2 \varepsilon^\alpha $, we can conclude from \eqref{momentum} and Proposition \ref{prop:Jakubow0} that
\begin{align}
\label{oddInequality}
 \nabla \eta_\varepsilon \cdot \big(\hat{\varrho}_{\varepsilon_n} \hat{\mathbf{u}}_{\varepsilon_n}\big)
= \varepsilon^{\alpha}
 \, \hat{F}_{\varepsilon_n}
\end{align}
for some $\hat{F}_{\varepsilon_n}$ belonging to
\begin{align}
\label{forceScaler}
 \hat{F}_{\varepsilon_n} \in L^p \big(  \tilde{\Omega}; L^\infty(0,T;L^\frac{2\gamma}{\gamma+1}_{\mathrm{loc}}\mathcal{O})\big)
\end{align}
uniformly in $\varepsilon_n$. 
We can therefore rewrite \eqref{contEqCutoff} as
\begin{equation}
\begin{aligned}
\label{contEqCutoff1}
\varepsilon^m 
\big\langle \eta_\varepsilon  \hat{r}_{\varepsilon_n}(t) \,,\, \varphi \big\rangle 
&- 
\int_0^t \big\langle \eta_\varepsilon  (\hat{\varrho}_{\varepsilon_n} \hat{\mathbf{u}}_{\varepsilon_n}) \,,\,\nabla \varphi
\big\rangle\mathrm{d}s   
=
\varepsilon^m 
\big\langle \eta_\varepsilon  \hat{r}_{\varepsilon_n}(0) \,,\, \varphi \big\rangle
\\&+
\varepsilon^{\alpha} 
\int_0^t \big\langle \hat{F}_{\varepsilon_n} \,,\, \varphi \big\rangle\mathrm{d}s
\end{aligned}
\end{equation}
for all $t\in[0,T]$ and $\varphi \in C^\infty_c(\mathbb{R}^3)$.
Similar to \eqref{contEqCutoff}, the corresponding momentum balance equation \eqref{acousticSPD00}$_2$ becomes
\begin{equation}
\begin{aligned}
\label{momEqCutoff}
 &\varepsilon^m\, \mathrm{d}\big[\eta_\varepsilon (\hat{\varrho}_{\varepsilon_n} \hat{\mathbf{u}}_{\varepsilon_n})\big] + \gamma \nabla \big[\eta_\varepsilon \hat{r}_{\varepsilon_n}\big]\mathrm{d}t 
 =  \varepsilon^m \, \mathrm{div}\big[\eta_\varepsilon \mathbb{S}(\nabla \hat{\mathbf{u}}_{\varepsilon_n})\big]\mathrm{d}t
 \\&-
 \varepsilon^m \, \mathrm{div}\big[ \eta_\varepsilon  (\hat{\varrho}_{\varepsilon_n} \hat{\mathbf{u}}_{\varepsilon_n} \otimes \hat{\mathbf{u}}_{\varepsilon_n})\big]\mathrm{d}t
 -
 \varepsilon^{m-1}\eta_\varepsilon  \big(\mathbf{e}_3\times \hat{\varrho}_{\varepsilon_n} \hat{\mathbf{u}}_{\varepsilon_n} \big)\,\mathrm{d}t
\\
&-  
\frac{1}{\varepsilon^m}\nabla \big[\eta_\varepsilon \big(\hat{\varrho}_{\varepsilon_n}^\gamma -\gamma(\hat{\varrho}_{\varepsilon_n}-\overline{\varrho}_{\varepsilon}) - \overline{\varrho}_{\varepsilon}^\gamma\big)\big]\mathrm{d}t
 \\
& +
 \varepsilon^{2(m-1)}\eta_\varepsilon (\hat{r}_{\varepsilon_n} \nabla G)\,\mathrm{d}t
 +
  \varepsilon^m\,  \eta_\varepsilon  \,\Phi(\hat{\varrho}_{\varepsilon_n},\hat{\varrho}_{\varepsilon_n} \hat{\mathbf{u}}_{\varepsilon_n})\mathrm{d}\tilde{W}_{n}
 \\
 &-
 \varepsilon^m\,\nabla \eta_\varepsilon  \cdot \big[ \mathbb{S}(\nabla \hat{\mathbf{u}}_{\varepsilon_n})\big] \mathrm{d}t
+
 \varepsilon^m\,\nabla \eta_\varepsilon  \cdot 
 (\hat{\varrho}_{\varepsilon_n} \hat{\mathbf{u}}_{\varepsilon_n} \otimes \hat{\mathbf{u}}_{\varepsilon_n}) \mathrm{d}t
\\
 &+ \gamma \nabla \eta_\varepsilon  \, \hat{r}_{\varepsilon_n}\mathrm{d}t
 +
  \nabla\eta_\varepsilon \,\frac{1}{\varepsilon^m} \big[\hat{\varrho}_{\varepsilon_n}^\gamma -\gamma(\hat{\varrho}_{\varepsilon_n}-\overline{\varrho}_{\varepsilon}) - \overline{\varrho}_{\varepsilon}^\gamma\big]\mathrm{d}t
    \\&
  =:\sum_{j=1}^{10}I_j
\end{aligned}
\end{equation}
which is to be understood in the distributional sense, i.e. for $t\in[0,T]$, the equality
\begin{equation}
\begin{aligned}
\label{momEqCutoff}
 &\varepsilon^m\big\langle\eta_\varepsilon (\hat{\varrho}_{\varepsilon_n} \hat{\mathbf{u}}_{\varepsilon_n})(t),\bm{\varphi}\big\rangle
-\gamma
\int_0^t \big\langle\eta_\varepsilon \hat{r}_{\varepsilon_n}, \mathrm{div}\bm{\varphi}\big\rangle \,\mathrm{d}s 
\\& 
= 
 \varepsilon^m
  \big\langle\eta_\varepsilon (\hat{\varrho}_{\varepsilon_n} \hat{\mathbf{u}}_{\varepsilon_n})(0) ,\bm{\varphi}\big\rangle
- 
\varepsilon^m
\int_0^t \big\langle  \eta_\varepsilon \mathbb{S}(\nabla \hat{\mathbf{u}}_{\varepsilon_n}) , \nabla\bm\varphi \big\rangle\mathrm{d}s
\\&
+
\varepsilon^m
\int_0^t \big\langle   \eta_\varepsilon  (\hat{\varrho}_{\varepsilon_n} \hat{\mathbf{u}}_{\varepsilon_n} \otimes \hat{\mathbf{u}}_{\varepsilon_n}) , \nabla \bm{\varphi}\big\rangle \mathrm{d}s
\\& 
-
\varepsilon^{m-1}
\int_0^t \big\langle \eta_\varepsilon  \big(\mathbf{e}_3\times \hat{\varrho}_{\varepsilon_n} \hat{\mathbf{u}}_{\varepsilon_n} \big) ,\bm{\varphi} \big\rangle\mathrm{d}s
\\&
+  
\frac{1}{\varepsilon^m}
\int_0^t \big\langle
\eta_\varepsilon \big(\hat{\varrho}_{\varepsilon_n}^\gamma -\gamma(\hat{\varrho}_{\varepsilon_n}-\overline{\varrho}_{\varepsilon}) - \overline{\varrho}_{\varepsilon}^\gamma\big) , \mathrm{div}\bm{\varphi} \big\rangle \mathrm{d}s
 \\& 
+
 \varepsilon^{2(m-1)}
 \int_0^t \big\langle
\eta_\varepsilon (\hat{r}_{\varepsilon_n} \nabla G) , \bm{\varphi} \big\rangle\mathrm{d}s
+
 \varepsilon^{m}
\int_0^t \big\langle
 \eta_\varepsilon  \,\Phi(\hat{\varrho}_{\varepsilon_n},\hat{\varrho}_{\varepsilon_n} \hat{\mathbf{u}}_{\varepsilon_n})
  \mathrm{d}\tilde{W}_{n} , \bm{\varphi} \big\rangle
 \\
 &
 - 
 \varepsilon^{m}\int_0^t \big\langle
\nabla \eta_\varepsilon  \cdot  \mathbb{S}(\nabla \hat{\mathbf{u}}_{\varepsilon_n}) , \bm{\varphi} \big\rangle \mathrm{d}s
+
\varepsilon^{m}
\int_0^t \big\langle
\nabla \eta_\varepsilon  \cdot 
 (\hat{\varrho}_{\varepsilon_n} \hat{\mathbf{u}}_{\varepsilon_n} \otimes \hat{\mathbf{u}}_{\varepsilon_n}) , \bm{\varphi} \big\rangle \mathrm{d}s
\\
 &+ 
 \gamma
 \int_0^t \big\langle
 \nabla \eta_\varepsilon  \, \hat{r}_{\varepsilon_n} , \bm{\varphi} \big\rangle
 \mathrm{d}s
 +
\frac{1}{\varepsilon^{m}} \int_0^t \big\langle
  \nabla\eta_\varepsilon  \big[\hat{\varrho}_{\varepsilon_n}^\gamma -\gamma(\hat{\varrho}_{\varepsilon_n}-\overline{\varrho}_{\varepsilon}) - \overline{\varrho}_{\varepsilon}^\gamma\big] , \bm{\varphi} \big\rangle \mathrm{d}s
    \\&
  =:\sum_{j=1}^{11}I_j
\end{aligned}
\end{equation}
holds for any $\bm{\varphi} \in C^\infty_c(\mathbb{R}^3)$. Regularity for the terms in \eqref{momEqCutoff} follows from the uniform estimates shown in Section \ref{subsec:energyineq} and the equality of law given by Proposition \ref{prop:Jakubow0}.  The terms $I_8, \ldots,I_{11}$ in \eqref{momEqCutoff} are even of lower order.  Again by using \eqref{newCufOff}, we can do a similar analysis as in \eqref{oddInequality} for the momentum equation \eqref{momEqCutoff} to get
\begin{equation}
\begin{aligned}
\label{momEqCutoff1}
 \varepsilon^m& \big\langle \eta_\varepsilon (\hat{\varrho}_{\varepsilon_n} \hat{\mathbf{u}}_{\varepsilon_n})(t) , \bm{\varphi} \big\rangle 
-
  \gamma 
 \int_0^t \big\langle
 \eta_\varepsilon \hat{r}_{\varepsilon_n}, \mathrm{div} \bm{\varphi} \big\rangle \mathrm{d}s 
\\&=
\varepsilon^m\big\langle \eta_\varepsilon (\hat{\varrho}_{\varepsilon_n} \hat{\mathbf{u}}_{\varepsilon_n})(0) \,,\, \bm{\varphi} \big\rangle 
-
 \big( \varepsilon^m + \varepsilon^{2(m-1-\alpha)}\big)
  \int_0^t \big\langle
 \hat{\mathbb{F}}_{\varepsilon_n}, \nabla \bm{\varphi} \big\rangle \mathrm{d}s
\\& + 
 \int_0^t \big\langle
\big( \varepsilon^{m-1}+\varepsilon^\alpha + \varepsilon^{2(m-1-\alpha)}\big)\,\hat{\mathbf{F}}_{\varepsilon_n}, \bm{\varphi} \big\rangle \mathrm{d}s
\\& +
  \varepsilon^m  \int_0^t \big\langle \eta_\varepsilon\Phi(\hat{\varrho}_{\varepsilon_n},\hat{\varrho}_{\varepsilon_n} \hat{\mathbf{u}}_{\varepsilon_n}) \mathrm{d}\tilde{W}_n , \bm{\varphi} \big\rangle
\end{aligned}
\end{equation}
for any $t\in[0,T]$ and any $\bm{\varphi} \in C^\infty_c(\mathbb{R}^3)$ and where
\begin{equation}
\begin{aligned}
\label{forcetensorAndmatrix}
\hat{\mathbb{F}}_{\varepsilon_n} &\in L^p \big(  \tilde{\Omega}; L^2(0,T;L^1_{\mathrm{loc}}\mathcal{O})\big),
 \\
 \hat{\mathbf{F}}_{\varepsilon_n} &\in L^p \big(  \tilde{\Omega}; L^2(0,T;L^1_{\mathrm{loc}}\mathcal{O})\big)
\end{aligned}
\end{equation}
uniformly in $\varepsilon_n$. We can mollify \eqref{contEqCutoff1} and \eqref{momEqCutoff1} by convolution with the usual mollifier $\wp_\kappa$ to get  for a.e. $(\omega, t,x)\in \tilde{\Omega}\times[0,T]\times\mathcal{O}$,
\begin{equation}
\begin{aligned}
\label{contEqCutoff2}
\varepsilon^m \,\mathrm{d}\big[\eta_\varepsilon  \hat{r}_{\varepsilon_n}\big]_\kappa +  \mathrm{div}\big[\eta_\varepsilon  (\hat{\varrho}_{\varepsilon_n} \hat{\mathbf{u}}_{\varepsilon_n})\big]_\kappa\mathrm{d}t   &=\varepsilon^{\alpha}\hat{F}_{\varepsilon_n,\kappa} \,\mathrm{d}t
\end{aligned}
\end{equation}
and
\begin{equation}
\begin{aligned}
\label{momEqCutoff2}
 &\varepsilon^m\, \mathrm{d}\big[\eta_\varepsilon (\hat{\varrho}_{\varepsilon_n} \hat{\mathbf{u}}_{\varepsilon_n})\big]_\kappa + \gamma \nabla \big[\eta_\varepsilon \hat{r}_{\varepsilon_n}\big]_\kappa\mathrm{d}t 
= \big[ A_{\varepsilon}(m,\alpha)\,
\mathrm{div}\,\mathbb{F}_{\varepsilon_n,\kappa} 
 \\&+
B_{\varepsilon}(m,\alpha)\,\mathbf{F}_{\varepsilon_n,\kappa}\big]\mathrm{d}t
 +
  \varepsilon^m \,\big[\eta_\varepsilon\Phi(\hat{\varrho}_{\varepsilon_n},\hat{\varrho}_{\varepsilon_n} \hat{\mathbf{u}}_{\varepsilon_n})\big]_\kappa\mathrm{d}\tilde{W}_{n}.
\end{aligned}
\end{equation}
respectively with
\begin{equation}
\begin{aligned}
\label{aAndb}
A_{\varepsilon}(m,\alpha)= \varepsilon^m + \varepsilon^{2(m-1-\alpha)},
\quad
B_{\varepsilon}(m,\alpha)=  \varepsilon^{m-1}+\varepsilon^\alpha + \varepsilon^{2(m-1-\alpha)}.
\end{aligned}
\end{equation} 
Let  $\mathcal{Q}= \nabla\Delta^{-1}_{\mathcal{O}}\mathrm{div}$ and $\mathcal{P}$ be respectively, the gradient and solenoidal parts according to Helmholtz decomposition with the identity operator satisfying $\mathrm{Id} = \mathcal{Q}+ \mathcal{P}$. Now define the function $ \hat{\Psi}_{\varepsilon_n,\kappa} =\Delta^{-1}_{\mathcal{O}}\mathrm{div}\big[\eta_\varepsilon (\hat{\varrho}_{\varepsilon_n}\hat{\mathbf{u}}_{\varepsilon_n}) \big]_\kappa$ so that  the relation  $\nabla \hat{\Psi}_{\varepsilon_n,\kappa} =\mathcal{Q}\big[\eta_\varepsilon (\hat{\varrho}_{\varepsilon_n}\hat{\mathbf{u}}_{\varepsilon_n}) \big]_\kappa$ holds. We can then recast \eqref{contEqCutoff2}  as\footnote{The right-hand side of \eqref{contEqCutoff3}, which is not a vector-valued function, may be viewed as the \textit{mononopolar} in Lighthill's interpretation of the inhomogeneous acoustic wave equation  \cite[Eq. 4.36]{feireisl2009singular}.}
\begin{equation}
\begin{aligned}
\label{contEqCutoff3}
\varepsilon^m \,\mathrm{d} \hat{\varphi}_{\varepsilon_n,\kappa} 
+  \Delta\hat{\Psi}_{\varepsilon_n,\kappa} \,\mathrm{d}t   &=\varepsilon^{\alpha}\hat{F}_{\varepsilon_n,\kappa} \,\mathrm{d}t
\end{aligned}
\end{equation}
where
\begin{align*}  \hat{\varphi}_{\varepsilon_n,\kappa} :=\big[\eta_\varepsilon  \hat{r}_{\varepsilon_n}\big]_\kappa =\Big[ \eta_\varepsilon \frac{\hat{\varrho}_{\varepsilon_n} - \overline{\varrho}_{\varepsilon}}{\varepsilon^m}\Big]_\kappa.
\end{align*}
Applying the operator $\mathcal{Q}$ to \eqref{momEqCutoff2} also yields
\begin{equation}
\begin{aligned}
\label{momEqCutoff3}
 &\varepsilon^m\, \mathrm{d}\big[\nabla \hat{\Psi}_{\varepsilon_n,\kappa}\big] 
 + 
 \gamma \nabla  \hat{\varphi}_{\varepsilon_n,\kappa} \mathrm{d}t 
 =
\big[ A_{\varepsilon}(m,\alpha)\mathcal{Q}\mathrm{div}\,\mathbb{F}_{\varepsilon_n,\kappa} 
 \\&+
B_{\varepsilon}(m,\alpha)\mathcal{Q}\,\mathbf{F}_{\varepsilon_n,\kappa} \big]\mathrm{d}t
 +
  \varepsilon^m \,\mathcal{Q}\big[\eta_\varepsilon\Phi(\hat{\varrho}_{\varepsilon_n},\hat{\varrho}_{\varepsilon_n} \hat{\mathbf{u}}_{\varepsilon_n})\big]_\kappa\mathrm{d}\tilde{W}_{n}.
\end{aligned}
\end{equation}
Equations \eqref{contEqCutoff3} and \eqref{momEqCutoff3} is equivalent to the system
\begin{equation}
\begin{aligned}
\label{acousticSPD111}
&\varepsilon^m\mathrm{d}
\begin{bmatrix}
       \hat{\varphi}_{\varepsilon_n,\kappa}        \\[0.3em]
       \nabla\hat{\Psi}_{\varepsilon_n,\kappa}
\end{bmatrix}  
=
\mathcal{A}
\begin{bmatrix}
        \hat{\varphi}_{\varepsilon_n,\kappa}        \\[0.3em]
       \nabla\hat{\Psi}_{\varepsilon_n,\kappa}
\end{bmatrix}  \mathrm{d}t 
+\varepsilon^m
\begin{bmatrix}
      \frac{\varepsilon^{\alpha}}{\varepsilon^m} \hat{F}_{\varepsilon_n,\kappa} \\[0.3em]
      0
\end{bmatrix}  \mathrm{d}t
\\
&+\varepsilon^m
\begin{bmatrix}
     0 \\[0.3em]
      \frac{B_\varepsilon}{\varepsilon^m}\mathcal{Q} \hat{\mathbf{F}}_{\varepsilon_n,\kappa}
\end{bmatrix}  \mathrm{d}t
+\varepsilon^m
\begin{bmatrix}
     0 \\[0.3em]
      \frac{A_\varepsilon}{\varepsilon^m}\mathcal{Q} \mathrm{div}\,\hat{\mathbb{F}}_{\varepsilon_n,\kappa}
\end{bmatrix}  \mathrm{d}t
+\varepsilon^m
\begin{bmatrix}
       0        \\[0.3em]
      \mathcal{Q}\eta_\varepsilon\hat{\Phi}_{\varepsilon_n,\kappa}
\end{bmatrix}  \mathrm{d}W_{n}.
\end{aligned}
\end{equation}
Here $\hat{\Phi}_{\varepsilon_n,\kappa}:=\Phi\left(\hat{\varrho}_{\varepsilon_n}, \hat{\varrho}_{\varepsilon_n}\hat{\mathbf{u}}_{\varepsilon_n} \right)_\kappa$ and the operator $\mathcal{A}$ is given by 
\begin{align}
\mathcal{A} := -
\begin{bmatrix}
       0 & \mathrm{div}       \\[0.3em]
       \gamma \nabla & 0
\end{bmatrix}.
\end{align}
Now let $E=L^2(\mathcal{O})\times L^2(\mathcal{O};\mathbb{R}^N)$ and consider the operator given by $
S(t)=e^{t\mathcal{A}}$.
We observe that $\left(S(t)\right)_{t\geq0}$, as a function of $t$, is a strongly continuous semigroup since,
\begin{align*} 
S(0)=\mathbbm{1},\quad S(t+s)=S(t)S(s),\quad  \lim_{t\downarrow 0} S(t)\mathbf{f}=\mathbf{f}
\end{align*}
for all $ \mathbf{f}=[\varphi
,\nabla\Psi]^T\in E$.  
Moreover  $\mathcal{A}$ is linear and
\begin{align*}
 \mathcal{A}\mathbf{f}=\lim_{t\downarrow 0}\left.\frac{S(t)\mathbf{f}-\mathbf{f}}{t}=\frac{\mathrm{d}S(t)\mathbf{f}}{\mathrm{d}t}\right\vert_{t=0}=e^{t\mathcal{A}}\mathcal{A}\mathbf{f}\left.\right\vert_{t=0}.
\end{align*}
Hence,  $\mathcal{A}$ is an infinitesimal generator of the strongly continuous semigroup $S(t)$ with domain
\begin{align*}
\mathrm{Dom}(\mathcal{A}) &=\left\{ \mathbf{f}\in E \, :\,\lim_{t\downarrow 0}\frac{S(t)\mathbf{f}-\mathbf{f}}{t} \text{ exists } \right\}
\\
&=\left\{ \mathbf{f}=[\varphi,\nabla\Psi]^T \, :\, \varphi\in W^{1,2}(\mathcal{O}),\, \nabla\Psi\in L^2(\mathcal{O}),\,\mathrm{div}\,\nabla\Psi\in L^2(\mathcal{O})\, \right\}
\end{align*}
\begin{prop}
\label{weakToMild}
Assume that $\mathcal{A} \, :\, \mathrm{Dom}(\mathcal{A})\subset E \longrightarrow E$ is an infinitesimal generator of a strongly continuous semigroup $(S(t))_{t\geq0}$ on $E$ and that\footnote{The following space means that $\hat{\Phi}_{\varepsilon_n,\kappa}$ is $L_2(\mathfrak{U};W^{-l,2}(\mathcal{O}))$-predictable and $\mathbb{E} \int_0^T \Vert \hat{\Phi}_{\varepsilon_n,\kappa} \Vert^2_{L_2(\mathfrak{U};W^{-l,2}(\mathcal{O}))} \mathrm{d}t<\infty$}
\begin{align*}
\hat{\Phi}_{\varepsilon_n,\kappa}\in\mathcal{N}^2_W\big(0,T; L_2(\mathfrak{U};W^{-l,2}(\mathcal{O})) \big) ,
\quad l>5/2
\end{align*}
Then a weak solution of  \eqref{acousticSPD111}  is also a mild solution.
\end{prop}
\begin{proof}
This is just a special case of a standard theorem. We refer the reader to \cite[Theorem 6.5]{da2014stochastic} for example.
\end{proof}
As a result of Proposition \ref{weakToMild}, we can rewrite Eq. \eqref{acousticSPD111}, after rescaling in time, in the mild form\footnote{This mild formulation is essentially a stochastic version of the Duhamel's formula with the added stochastic convolution term given by the noise.}
\begin{equation}
\begin{aligned}
\label{duhamel}
&\begin{bmatrix}
       \hat{\varphi}_{\varepsilon_n,\kappa}        \\[0.3em]
       \nabla\hat{\Psi}_{\varepsilon_n,\kappa}
\end{bmatrix}  (t) 
=
S\left(\frac{t}{\varepsilon^m}\right)
\begin{bmatrix}
       \hat{\varphi}_{\varepsilon_n,\kappa}(0)        \\[0.3em]
       \nabla\hat{\Psi}_{\varepsilon_n,\kappa}(0)
\end{bmatrix} 
+\int_0^t S\left(\frac{t-s}{\varepsilon^m}\right)
\begin{bmatrix}
      \frac{\varepsilon^{\alpha}}{\varepsilon^m} \hat{F}_{\varepsilon_n,\kappa} \\[0.3em]
      0
\end{bmatrix}
\mathrm{d}s
\\
&+\int_0^t S\left(\frac{t-s}{\varepsilon^m}\right)
\begin{bmatrix}
     0 \\[0.3em]
      \frac{B_\varepsilon}{\varepsilon^m} \mathcal{Q}\hat{\mathbf{F}}_{\varepsilon_n,\kappa}
\end{bmatrix}
\mathrm{d}s
+
\int_0^t S\left(\frac{t-s}{\varepsilon^m}\right)
\begin{bmatrix}
     0 \\[0.3em]
      \frac{A_\varepsilon}{\varepsilon^m} \mathcal{Q}\mathrm{div}\,\hat{\mathbb{F}}_{\varepsilon_n,\kappa}
\end{bmatrix}
\mathrm{d}s
\\
&+\int_0^t S\left(\frac{t-s}{\varepsilon^m}\right)
\begin{bmatrix}
       0     \\[0.3em]
       \mathcal{Q}\eta_\varepsilon \hat{\Phi}_{\varepsilon_n,\kappa}
\end{bmatrix}
\mathrm{d}\tilde{W}_{s,n}
\\
&=:J_1 + J_2+ J_3 + J_4 + J_5.
\end{aligned}
\end{equation}
where
\begin{equation}
\label{semigroup}
S\left(t\right)
\begin{bmatrix}
       \hat{\varphi}_{0,\varepsilon_n,\kappa} (\cdot)       \\[0.3em]
       \nabla\hat{\Psi}_{0,\varepsilon_n,\kappa} (\cdot)
\end{bmatrix} 
=
\begin{bmatrix}
       \hat{\varphi}_{\varepsilon_n,\kappa}   (\cdot, t)     \\[0.3em]
     \nabla  \hat{\Psi}_{\varepsilon_n,\kappa}(\cdot, t)
\end{bmatrix}  
\end{equation}
is the solution to the homogeneous  PDE
\begin{equation}
\begin{aligned}
\label{homogeAcoustic}
\mathrm{d}\hat{\varphi}_{\varepsilon_n,\kappa} +  \Delta \hat{\Psi}_{\varepsilon_n,\kappa}\,\mathrm{d}t   &=  0,   \\
 \mathrm{d}\nabla\hat{\Psi}_{\varepsilon_n,\kappa}  +  \gamma \nabla \hat{\varphi}_{\varepsilon_n,\kappa}\,\mathrm{d}t   &=    0,
 \\
\hat{\varphi}_{\varepsilon_n,\kappa}(0) = \hat{\varphi}_{0,\varepsilon_n,\kappa}; \quad \nabla\hat{\Psi}_{\varepsilon_n,\kappa}(0) &= \nabla\hat{\Psi}_{0,\varepsilon_n,\kappa}.
\end{aligned}
\end{equation}
C.f. the purely deterministic case \cite[Eq. 8.111]{feireisl2009singular}. Using Fourier transforms (in space), we obtain for a.e. $(\omega,x)\in\Omega\times \mathcal{O}$, an exact solution to \eqref{homogeAcoustic} given by the pair
\begin{equation}
\begin{aligned}
\label{solution}
\nabla\hat{\Psi}_{\varepsilon_n,\kappa}(\cdot,t)  &=  \frac{1}{2}\exp\big(i\sqrt{-\gamma\Delta}t\big)\left(\nabla\hat{\Psi}_{0,\varepsilon_n,\kappa} +\frac{i\sqrt{\gamma}}{\sqrt{-\Delta}}\nabla \hat{\varphi}_{0,\varepsilon_n,\kappa} \right)    
\\
&+   \frac{1}{2}\exp\big(-i\sqrt{-\gamma\Delta}t\big)\left( \nabla\hat{\Psi}_{0,\varepsilon_n,\kappa}  -  \frac{i\sqrt{\gamma}}{\sqrt{-\Delta}}\nabla \hat{\varphi}_{0,\varepsilon_n,\kappa}\right),
\\
\hat{\varphi}_{\varepsilon_n,\kappa}(\cdot, t)  &=  \frac{1}{2}\exp\big(i\sqrt{-\gamma\Delta}t\big)\left( \hat{\varphi}_{0,\varepsilon_n,\kappa}  -  \frac{i\sqrt{-\Delta}}{\sqrt{\gamma}}\hat{\Psi}_{0,\varepsilon_n,\kappa} \right)   
\\
&+ \frac{1}{2}\exp\big(-i\sqrt{-\gamma\Delta}t\big)\left( \hat{\varphi}_{0,\varepsilon_n,\kappa} + \frac{i\sqrt{-\Delta}}{\sqrt{\gamma}} \hat{\Psi}_{0,\varepsilon_n,\kappa}  \right) .
\end{aligned}
\end{equation}
\begin{rem}
Note that by substitution, the problem \eqref{homogeAcoustic} ( and thus its solution \eqref{solution}), may be recast as a single system of PDEs. A similar remark holds for the inhomogeneous counterpart of \eqref{homogeAcoustic}.
\end{rem}
We now state a lemma, the proof of which  follows by taking expectation in \cite[Lemma 3.1]{feireisl2012multi}.
\begin{lem}
\label{lem:strichartz}
Let $\phi\in C_c^\infty(\mathbb{R}^2)$. Then the inequality
\begin{align*}
\mathbb{E}\,\big\Vert \phi(x_h)\,\exp(i\sqrt{-\gamma\Delta}t)\,[\mathbf{f}] \big\Vert^2_{L^2(\mathbb{R}\times \mathcal{O})} \lesssim_\phi 
\mathbb{E}\,\Vert \,\mathbf{f} \, \Vert^2_{L^2( \mathcal{O})}
\end{align*}
holds for $\mathbf{f}\in L^2(\Omega\times \mathcal{O})$.
\end{lem}
\begin{rem}
The proof of Lemma \ref{lem:strichartz} uses space-time Fourier transform. One can therefore interpret $\mathbf{f}$ as the extension by zero outside of its time interval to the whole line $\mathbb{R}$, assuming it is only defined on $[0,T]$. This remark also applies to the analysis below.
\end{rem}
With Lemma \ref{lem:strichartz} in hand, we are able to estimate the right-hand of \eqref{duhamel}. To see this, we first notice that by rescaling \eqref{semigroup} in time, we obtain
\begin{equation}
\begin{aligned}
\label{semigroup1}
&\mathbb{E}\Bigg\Vert S\left(\frac{t}{\varepsilon^m}\right)
\begin{bmatrix}
       \hat{\varphi}_{0,\varepsilon_n,\kappa} 
        \\[0.3em]
       \nabla\hat{\Psi}_{0,\varepsilon_n,\kappa} 
\end{bmatrix} 
       \Bigg\Vert^2_{L^2((0,T)\times \mathcal{O})}
\leq
\mathbb{E}\Bigg\Vert S\left(\frac{t}{\varepsilon^m}\right)
\begin{bmatrix}
       \hat{\varphi}_{0,\varepsilon_n,\kappa}
        \\[0.3em]
       \nabla\hat{\Psi}_{0,\varepsilon_n,\kappa} 
\end{bmatrix} 
       \Bigg\Vert^2_{L^2(\mathbb{R}\times \mathcal{O})}
\\
&\leq \varepsilon^m\,
\mathbb{E}\,\Bigg\Vert S\left(t\right)
\begin{bmatrix}
       \hat{\varphi}_{0,\varepsilon_n,\kappa}
        \\[0.3em]
       \nabla\hat{\Psi}_{0,\varepsilon_n,\kappa} 
\end{bmatrix} 
       \Bigg\Vert^2_{L^2(\mathbb{R}\times \mathcal{O})}
= \varepsilon^m\,
\mathbb{E}\,\Bigg\Vert 
\begin{bmatrix}
       \hat{\varphi}_{\varepsilon_n,\kappa}(t)
        \\[0.3em]
       \nabla\hat{\Psi}_{\varepsilon_n,\kappa}(t) 
\end{bmatrix} 
       \Bigg\Vert^2_{L^2(\mathbb{R}\times \mathcal{O})}.
\end{aligned}
\end{equation}
However since the final term in \eqref{semigroup1} above has entries satisfying \eqref{solution}, we can use Lemma \ref{lem:strichartz} to show that these terms  are controlled. In particular, for ${g}_{0,\varepsilon_n,\kappa}$
denoting the initial data on the right-hand side of \eqref{solution}, we gain by using Lemma \ref{lem:strichartz} that the estimate
\begin{equation}
\begin{aligned}
\label{semigroup2}
\mathbb{E}\,\bigg\Vert   \exp\Big(i\sqrt{-\gamma \Delta}t\Big) \big[\eta_\varepsilon {g}_{0,\varepsilon_n}\big]_\kappa \bigg\Vert^2_{L^2(\mathbb{R}\times \mathcal{O})}
&\lesssim
\mathbb{E}\,\Vert \,{g}_{0,\varepsilon_n,\kappa} \Vert^2_{L^2(\mathcal{O})}
\\
&\lesssim \Vert \,\wp_\kappa\Vert^2_{L^p( \mathcal{O})}\, \mathbb{E}\,\Vert \,{g}_{0,\varepsilon_n} \Vert^2_{L^q(\mathcal{O})}
\end{aligned}
\end{equation}
holds for $1/p+1/q=3/2$. Subsequently we can use the boundedness of the initial law \eqref{momentsBounded} and the mollifier $\wp_\kappa$ to conclude that
\begin{equation}
\begin{aligned}
\mathbb{E}\,\Bigg\Vert S&\left(\frac{t}{\varepsilon^m}\right)
\begin{bmatrix}
       \hat{\varphi}_{0,\varepsilon_n,\kappa} 
        \\[0.3em]
       \nabla\hat{\Psi}_{0,\varepsilon_n,\kappa} 
\end{bmatrix} 
       \Bigg\Vert^2_{L^2((0,T)\times \mathcal{O})}
\lesssim_\kappa\varepsilon^m.
\end{aligned}
\end{equation}
Uniform estimates for the terms $J_2,\ldots, J_4$ in \eqref{duhamel} follows  a similar argument as in \cite[Eq 60]{mensah2016existence} (c.f. \cite[Eq. 3.21]{feireisl2012multi}). Indeed with the uniform estimates Lemma \ref{uniformBounds}, \eqref{allForces}, \eqref{limDenseAndCentri} as well as the equality of laws given in Proposition \ref{prop:Jakubow0} in hand, we have that for any $K\Subset\mathcal{O}$,
\begin{equation}
\begin{aligned}
\label{semigroup3}
\mathbb{E}\,\Bigg\Vert \int_0^tS&\left(\frac{t-s}{\varepsilon^m}\right)
      D_\varepsilon\, \mathbf{f}_{\varepsilon_n,\kappa} 
\,\mathrm{d}s
       \Bigg\Vert^2_{L^2((0,T)\times \mathcal{O})}
\\
&\lesssim_T
D_\varepsilon^2\,
\mathbb{E}\,\Bigg\Vert S\left(\frac{t-s}{\varepsilon^m}\right)
       \mathbf{f}_{\varepsilon_n,\kappa} 
       \Bigg\Vert^2_{L^2((0,T)^2\times \mathcal{O})}
 \\
&\lesssim D_\varepsilon^2\,\varepsilon^m\,
\mathbb{E}\,\Bigg\Vert S\left(\frac{-s}{\varepsilon^m}\right)
       \mathbf{f}_{\varepsilon_n,\kappa} 
       \Bigg\Vert^2_{L^2((0,T)\times \mathcal{O})}
\\
&= c\, D_\varepsilon^2\,\varepsilon^m\,
\mathbb{E}\,\Vert 
      \, \mathbf{f}_{\varepsilon_n,\kappa} 
\,
      \Vert^2_{L^2((0,T)\times \mathcal{O})}.
\end{aligned}
\end{equation}
Here $\mathbf{f}_{\varepsilon_n,\kappa}$ is comparable to the terms in  $J_2,\ldots, J_4$ and we may choose 
\begin{align*}
D_\varepsilon^2=\max\big\{ (\varepsilon^{-m}A_\varepsilon)^2, (\varepsilon^{-m}B_\varepsilon)^2 , \varepsilon^{2(\alpha-m)}\big\}
\end{align*} 
for the entries defined in \eqref{aAndb} so that given \eqref{alpha}, we have that $D_\varepsilon^2<\varepsilon$. Thus given \eqref{forceScaler} and \eqref{forcetensorAndmatrix}, we can conclude  from \eqref{semigroup3} and the properties of convolution that
\begin{equation}
\begin{aligned}
\label{semigroup4}
\mathbb{E}\,\Vert    J_2+\ldots + J_4\Vert^2_{L^2(\mathbb{R}\times K)}
\lesssim_\kappa \varepsilon.
\end{aligned}
\end{equation}
To estimate the stochastic term, we first define the term $\hat{\Phi}_{\varepsilon_n,\kappa}(e_i):  =\hat{\mathbf{g}}_i^{\varepsilon_n,\kappa}:=$ \\$ \mathbf{g}_i\left(\cdot, \hat{\varrho}_{\varepsilon_n}(\cdot), (\hat{\mathbf{m}}_{\varepsilon_n} )(\cdot) \right)_\kappa$. Then similar to the estimate for the noise term in \cite[Page 2128]{mensah2016existence}, we can use  It\^{o} isometry, Fubini's theorem, the properties of the semigroup, Lemma \ref{lem:strichartz} and the continuity of the operator $\mathcal{Q}$ to get that
\begin{align*}
\tilde{\mathbb{E}}\,\Bigg\Vert
\int\limits_0^t 
 &S\left(\frac{t-s}{\varepsilon^m}\right)
       \mathcal{Q}\hat{\Phi}_{\varepsilon_n,\kappa}
 \,\mathrm{d}\tilde{W}_{n}(s) \Bigg\Vert^2_{L^2((0,T)\times K)} 
 \\
 &=
\tilde{\mathbb{E}}\,\int\limits_0^t \sum_{i\in\mathbb{N}} \left\Vert 
 S\left(\frac{t-s}{\varepsilon^m}\right)
       \mathcal{Q}\hat{\mathbf{g}}_i^{\varepsilon_n,\kappa}
       \right\Vert^2_{L^2((0,T)\times K)}\mathrm{d}s
\\
&\lesssim
\int\limits_0^T
\sum\limits_{i\in\mathbb{N}}  \int_{\mathbb{R}}\hat{\mathbb{E}}
\left\Vert 
 S\left(\frac{t-s}{\varepsilon^m}\right)
       \mathcal{Q}\hat{\mathbf{g}}_i^{\varepsilon_n,\kappa}
          \right\Vert^2_{L^2(K)}
\mathrm{d}s\,\mathrm{d}t.
\\
&
\lesssim \varepsilon^m\,
\int\limits_0^T
\sum\limits_{i\in\mathbb{N}}   \tilde{\mathbb{E}}
\left\Vert 
    \hat{\mathbf{g}}_i^{\varepsilon_n,\kappa}
          \right\Vert^2_{L^2(\mathcal{O})}
\mathrm{d}t 
\lesssim_\kappa \varepsilon^m
\end{align*}
for any $K\Subset\mathcal{O}$. We have shown the following lemma.
\begin{lem}
\label{lem:uniformGrad}
There exists a constant $c$ uniform in $\varepsilon$ such that
\begin{align*}
\tilde{\mathbb{E}}\,\Vert \hat{\varphi}_{\varepsilon_n,\kappa} (t) \Vert^2_{L^2((0,T)\times K)}
+ \tilde{\mathbb{E}}\,\Vert \nabla \hat{\Psi}_{\varepsilon_n,\kappa}(t) \Vert^2_{L^2((0,T)\times K)}
\lesssim_\kappa \varepsilon .
\end{align*}
for any $K \Subset\mathcal{O}$.
\end{lem}

Having shown Lemma \ref{lem:uniformGrad}, we can combine it with \cite[Eq. 64]{mensah2016existence} for an arbitrary ball \eqref{specialball} to get  \eqref{acousticMomen}, for at least a subsequence.
\end{proof}
To continue, we now identify the structure of the limit process for the velocity. 
\begin{lem}
$\tilde{\mathbb{P}}$-a.s., we have that
\begin{align}
\label{solenodalVelocity}
\hat{\mathbf{u}}_{\varepsilon_n} \, \rightharpoonup \, \hat{\mathbf{U}} &\quad\text{in}\quad L^2(0,T;W^{1,2}(\mathcal{O}))
\end{align}
and with this limit being of the form
\begin{align}
\label{2dVelocity}
\hat{\mathbf{U}}=\hat{\mathbf{U}}_h(x_h,0)= (\hat{U}^1(x_h,0), \hat{U}^2(x_h,0), 0).
\end{align}
\end{lem}
\begin{proof}
The first part is precisely contained in Proposition \ref{prop:Jakubow0}. The important bit is showing \eqref{2dVelocity}. To see this, we observe that by the equality of laws given in Proposition \ref{prop:Jakubow0} as well as Lemma \ref{lem:strongDensity}, one can conclude that the sequence $\hat{\varrho}_{\varepsilon_n}$ as $n\rightarrow\infty$, converges to
\begin{align}
\label{limitDensityOne}
\hat{\varrho}=1
\end{align}
$\tilde{\mathbb{P}}$-a.s. 
%
Due to Theorem \cite[Theorem 2.9.1]{breit2017stoch}, on the new probability space, one can then use this strong convergence of the density to \eqref{limitDensityOne} in order to pass to the limit in the corresponding continuity equation \eqref{comprSPDE0}$_1$ to get
\begin{align}
\label{incomCondition}
\mathrm{div}\,\hat{\mathbf{U}}=0, \quad \tilde{\mathbb{P}}\text{-a.e.} \text{ in }[0,T]\times\mathcal{O}
\end{align}
which implies that
\begin{align}
\label{incomCondition00}
\partial_{x_3}\hat{U}^3 = -  \partial_{x_1}\hat{U}^1 - \partial_{x_2}\hat{U}^2
\end{align}
$\tilde{\mathbb{P}}$-a.e. in $[0,T]\times\mathcal{O}$. Furthermore, it follows from Section \ref{sec:Coriolis} that $\mathcal{P}(\mathbf{e}_3 \times \mathbf{\hat{U}})=0$ and hence there exist a potential $\hat{\psi}=\hat{\psi}(x_h,x_3)$ such that $\mathbf{e_3}\times \hat{\mathbf{U}} =\nabla \hat{\psi}$. Or equivalently,
\begin{align}
\label{potentials}
\partial_{x_1}\hat{\psi}=-\hat{U}^2,\quad \partial_{x_2}\hat{\psi}=\hat{U}^1,\quad \partial_{x_3}\hat{\psi}=0.
\end{align}
By combining \eqref{incomCondition00} and \eqref{potentials}, we obtain
\begin{align*}
\partial_{x_3}\hat{U}^3=0
\end{align*}
which implies that $\hat{U}_3=\hat{U}_3(x_h,c)$ for a constant $c$ independent of $x_h$. But since by \eqref{boundaryCond1},  $\hat{U}_3(x_h,1)=0$, it implies that $\hat{U}_3 \equiv 0$.
\end{proof}

\subsection{Strong convergence for the vertical average of the solenoidal part of momentum}
\label{sec:verticalSolenoidal}
We now wish to control the vertical average for the solenoidal part of momentum. 
Since the spatial regularity of \eqref{momentum} just fails to belong to the class of Hilbert spaces, the aim is to improve this. We will require an $L^2$-spatial regularity in order to pass to the limit in the convective term of the momentum balance equation.
\\
First of all, let set $\hat{\mathbf{Y}}_{\varepsilon_n,\kappa}:=\mathcal{P}\left(\eta_\varepsilon\hat{\varrho}_{\varepsilon_n} \hat{\mathbf{u}}_{\varepsilon_n}\right)_\kappa$. Since the embedding $$C_w\left( [0,T];L^\frac{2\gamma}{\gamma-1}(K)\right)\hookrightarrow L^2\left( 0,T;W^{-1,2}(K) \right)$$ is continuous for $K\Subset \mathcal{O}$, we can conclude from \eqref{kolmogorovEst2} and Proposition \ref{prop:Jakubow0} that
\begin{equation}
\begin{aligned}
\ulcorner \hat{\mathbf{Y}}_{\varepsilon_n}\urcorner 
\rightarrow
 \hat{\mathbf{U}}\quad
 \text{in}\quad
 L^2\left( 0,T;W^{-1,2}(K)\right)
\end{aligned}
\end{equation}
$\tilde{\mathbb{P}}$-a.s. (for at least a subsequence). Furthermore, for fixed $\kappa>0$, one can find a constant $c>0$ independent of $\varepsilon_n$ such that
\begin{align}
\big\Vert
\ulcorner \hat{\mathbf{Y}}_{\varepsilon_n,\kappa }\urcorner -  \hat{\mathbf{U}}_\kappa\big\Vert_{L^2\left((0,T)\times K\right)}
\leq c(\kappa)
\big\Vert
\ulcorner \hat{\mathbf{Y}}_{\varepsilon_n }\urcorner -  \hat{\mathbf{U}}\big\Vert_{L^2\left(0,T;W^{-1,2}( K)\right)}
\end{align}
so that we obtain
\begin{align}
\label{vertSolenoiStrongConv}
\ulcorner \hat{\mathbf{Y}}_{\varepsilon_n,\kappa }\urcorner  \rightarrow  \hat{\mathbf{U}}_\kappa \quad\text{in} \quad L^2\left((0,T)\times K\right)
\end{align}
$\tilde{\mathbb{P}}$-a.s. for any fixed $\kappa>0$ as $n\rightarrow \infty$. 
\begin{rem}
Here we remind the reader of the structure of the limit velocity \eqref{2dVelocity}. 
\end{rem}

\subsection{Oscillatory part of momentum}
\label{sec:oscillatoryMomentum}
As a result of compactness of the vertical average of the solenoidal part of momentum, any source of oscillation will inherently stem from the vertical coordinate dependent component of momentum . However, we can show that this oscillatory component does not interfere with the analysis of the convective term in the momentum balance equation. 

\begin{prop}
\label{prop:avgConvDecom}
For all $\varepsilon_n>0$, we let $\mathcal{P}\left(\eta_\varepsilon\hat{\varrho}_{\varepsilon_n} \hat{\mathbf{u}}_{\varepsilon_n}\right)=:\hat{\mathbf{Y}}_{\varepsilon_n}$ be the solenoidal part of momentum solving \eqref{comprSPDE0}$_2$ in the sense of distributions. Now let \\$\left( \hat{\mathbf{Y}}_{\varepsilon_n, \kappa} \right)_{\kappa>0}$ be its regularized family obtain by convolution with the usual mollifier $\wp_\kappa$. Then  for any $\bm{\phi}=[\underline{\bm{\phi}}(x_h),0]$ with $\underline{\bm{\phi}}\in C^\infty_{c,\mathrm{div}_h}(\mathbb{R}^2)$ and $t\in[0,T]$, we have that
\begin{align*}
&-\lim_{n \rightarrow \infty}\int_0^t\left\langle  \hat{\mathbf{Y}}_{\varepsilon_n,\kappa} \otimes \hat{\mathbf{Y}}_{\varepsilon_n,\kappa}\,,\, \nabla \bm{\phi} \right\rangle\,\mathrm{d}\tau
 =
\int_0^t\int_{\mathbb{R}^2} \mathrm{div}\left( \hat{\mathbf{U}}_{\kappa} \otimes \hat{\mathbf{U}}_{\kappa}\right)
 \cdot \bm{\phi}\,\mathrm{d}x_h\,\mathrm{d}\tau
\end{align*}
holds $\tilde{\mathbb{P}}$-a.s.
\end{prop}
\begin{proof}
For the following decomposition
\begin{align}
\label{verticalAvgAndRest1}
\hat{\mathbf{Y}}_{\varepsilon_n,\kappa } (x) = \left\ulcorner \hat{\mathbf{Y}}_{\varepsilon_n,\kappa } \right\urcorner  (x_h)+
\llcorner \hat{\mathbf{Y}}_{\varepsilon_n,\kappa } \lrcorner (x),
\end{align}
we observe that
\begin{align}
\label{llulzero}
\left\ulcorner \llcorner \hat{\mathbf{Y}}_{\varepsilon_n,\kappa }\lrcorner \right\urcorner  =0
\end{align}
a.s. and as  such, we can find a $I\big(\hat{\mathbf{Y}}_{\varepsilon_n,\kappa }\big)$ such that
\begin{align}
\label{verticalAvgAndRest2}
\llcorner \hat{\mathbf{Y}}_{\varepsilon_n,\kappa }\lrcorner =\partial_{x_3}I\big(\hat{\mathbf{Y}}_{\varepsilon_n,\kappa }\big),\quad \left\ulcorner I\big(\hat{\mathbf{Y}}_{\varepsilon_n,\kappa }\big)\right\urcorner  =0
\end{align}
c.f. \cite[Section 3.3.2]{feireisl2012multi} and \cite[Section 3.2]{gallagher2006weak}. 

Let now rewrite \eqref{momEqCutoff} and mollify the resultant system to obtain the  following  
\begin{equation}
\begin{aligned}
\label{momEqCutoff01}
 \varepsilon \, \mathrm{d}\big[\eta_\varepsilon &(\hat{\varrho}_{\varepsilon_n} \hat{\mathbf{u}}_{\varepsilon_n})\big]_\kappa
 +
\mathbf{e}_3\times \big[\eta_\varepsilon  \big(\hat{\varrho}_{\varepsilon_n} \hat{\mathbf{u}}_{\varepsilon_n} \big) \big]_\kappa\,\mathrm{d}t
=  \varepsilon^{m-1}\big[\eta_\varepsilon (\hat{r}_{\varepsilon_n} \nabla G)\big]_\kappa\,\mathrm{d}t
 \\&
 +  \varepsilon \, \mathrm{div}\big[\eta_\varepsilon \mathbb{S}(\nabla \hat{\mathbf{u}}_{\varepsilon_n})
 -
 \eta_\varepsilon  (\hat{\varrho}_{\varepsilon_n} \hat{\mathbf{u}}_{\varepsilon_n} \otimes \hat{\mathbf{u}}_{\varepsilon_n})\big]_\kappa\mathrm{d}t
 \\
 &-
 \varepsilon\,\big[\nabla \eta_\varepsilon  \cdot \big( \mathbb{S}(\nabla \hat{\mathbf{u}}_{\varepsilon_n})-
 (\hat{\varrho}_{\varepsilon_n} \hat{\mathbf{u}}_{\varepsilon_n} \otimes \hat{\mathbf{u}}_{\varepsilon_n})\big) \big]_\kappa \,\mathrm{d}t
\\
&
 +
 \varepsilon^{1-2m}\Big( \big[\nabla\eta_\varepsilon \, (\hat{\varrho}_{\varepsilon_n}^\gamma - \overline{\varrho}_{\varepsilon}^\gamma)\big]_\kappa
-
\nabla \big[\eta_\varepsilon \big(\hat{\varrho}_{\varepsilon_n}^\gamma 
- \overline{\varrho}_{\varepsilon}^\gamma\big)\big]_\kappa \Big)\mathrm{d}t 
\\
& +
  \varepsilon\, \big[ \eta_\varepsilon  \,\Phi(\hat{\varrho}_{\varepsilon_n},\hat{\varrho}_{\varepsilon_n} \hat{\mathbf{u}}_{\varepsilon_n})\big]_\kappa\mathrm{d}\tilde{W}_{\varepsilon_n}
\\
 &=:
\varepsilon\, \mathbf{Q}_{\varepsilon_n,\kappa}\,\mathrm{d}t
 +
 \varepsilon^{1-2m}\mathbf{P}_{\varepsilon_n,\kappa}\,\mathrm{d}t
  +
  \varepsilon\, \Phi_{\varepsilon_n,\kappa}\mathrm{d}\tilde{W}_{n}
\end{aligned}
\end{equation}
where
\begin{equation}
\begin{aligned}
\mathbf{Q}_{\varepsilon_n,\kappa}&:=
  \varepsilon^{m-2}\big[\eta_\varepsilon (\hat{r}_{\varepsilon_n} \nabla G)\big]_\kappa
 \\&
 + \mathrm{div}\big[\eta_\varepsilon \mathbb{S}(\nabla \hat{\mathbf{u}}_{\varepsilon_n})
 -
 \eta_\varepsilon  (\hat{\varrho}_{\varepsilon_n} \hat{\mathbf{u}}_{\varepsilon_n} \otimes \hat{\mathbf{u}}_{\varepsilon_n})\big]_\kappa
 \\
 &-
\big[\nabla \eta_\varepsilon  \cdot \big( \mathbb{S}(\nabla \hat{\mathbf{u}}_{\varepsilon_n})-
 (\hat{\varrho}_{\varepsilon_n} \hat{\mathbf{u}}_{\varepsilon_n} \otimes \hat{\mathbf{u}}_{\varepsilon_n})\big) \big]_\kappa, 
\\
\mathbf{P}_{\varepsilon_n,\kappa}
&:=\Big( \big[\nabla\eta_\varepsilon \, (\hat{\varrho}_{\varepsilon_n}^\gamma - \overline{\varrho}_{\varepsilon}^\gamma)\big]_\kappa
-
\nabla \big[\eta_\varepsilon \big(\hat{\varrho}_{\varepsilon_n}^\gamma 
- \overline{\varrho}_{\varepsilon}^\gamma\big)\big]_\kappa \Big),
\\
\Phi_{\varepsilon_n,\kappa}
&:=  \varepsilon\, \big[ \eta_\varepsilon  \,\Phi(\hat{\varrho}_{\varepsilon_n},\hat{\varrho}_{\varepsilon_n} \hat{\mathbf{u}}_{\varepsilon_n})\big]_\kappa.
\end{aligned}
\end{equation}
Furthermore, rewriting \eqref{momEqCutoff01} in component form and differentiating results in the following
\begin{align}
 \varepsilon \, \mathrm{d}\partial_{x_i}\big(\eta_\varepsilon \hat{\varrho}_{\varepsilon_n} \hat{u}_{\varepsilon_n}\big)_\kappa^1
 &-
\partial_{x_i}\big(\eta_\varepsilon \hat{\varrho}_{\varepsilon_n} \hat{u}_{\varepsilon_n}\big)_\kappa^2 \,\mathrm{d}t 
=
\varepsilon\, \partial_{x_i}Q_{\varepsilon_n,\kappa}^1\,\mathrm{d}t
 \nonumber \\&+
 \varepsilon^{1-2m}\, \partial_{x_i}P_{\varepsilon_n,\kappa}^1\,\mathrm{d}t
 +
  \varepsilon\, \partial_{x_i}\Phi_{\varepsilon_n,\kappa}^1\mathrm{d}\tilde{W}_{n}
  \label{1d1a} \\
   \varepsilon \,  \mathrm{d}\partial_{x_i}\big(\eta_\varepsilon \hat{\varrho}_{\varepsilon_n} \hat{u}_{\varepsilon_n}\big)_\kappa^2
 &+
\partial_{x_i}\big(\eta_\varepsilon \hat{\varrho}_{\varepsilon_n} \hat{u}_{\varepsilon_n}\big)_\kappa^1 \,\mathrm{d}t 
=
\varepsilon\, \partial_{x_i}Q_{\varepsilon_n,\kappa}^2\,\mathrm{d}t
 \nonumber  \\&+
 \varepsilon^{1-2m}\, \partial_{x_i}P_{\varepsilon_n,\kappa}^2\,\mathrm{d}t
 +
  \varepsilon\,\partial_{x_i} \Phi_{\varepsilon_n,\kappa}^2\mathrm{d}\tilde{W}_{n}
   \label{1d2a}  \\
   &\varepsilon \, \mathrm{d}\partial_{x_i}\big(\eta_\varepsilon \hat{\varrho}_{\varepsilon_n} \hat{u}_{\varepsilon_n}\big)_\kappa^3
=
\varepsilon\, \partial_{x_i}Q_{\varepsilon_n,\kappa}^3\,\mathrm{d}t
 \nonumber \\&+
 \varepsilon^{1-2m}\, \partial_{x_i}P_{\varepsilon_n,\kappa}^3\,\mathrm{d}t
 +
  \varepsilon\, \partial_{x_i}\Phi_{\varepsilon_n,\kappa}^3\mathrm{d}\tilde{W}_{n}  \label{1d3a}
\end{align}
for  $i=1,2,3$.

We now recall that in terms of coordinates, we can write the decomposition of momentum into its solenoidal and gradient parts as 
\begin{align}
\label{sov-grad-part}
\big(\eta_\varepsilon\hat{\varrho}_{\varepsilon_n}\hat{u}_{\varepsilon_n}\big)^i_\kappa= \hat{Y}^i_{\varepsilon_n,\kappa} + \partial_{x_i}\hat{\Psi}_{\varepsilon_n,\kappa}, \quad i=1,2,3.
\end{align}
As such, the symmetry of the Hessian means that we can define the traceless skew-symmetric operator
\begin{align}
\label{momentumCommutator}
\hat{\Upsilon}^{ij}_{\varepsilon_n,\kappa}:=\partial_{x_i}\big(\eta_\varepsilon\hat{\varrho}_{\varepsilon_n}\hat{u}_{\varepsilon_n}\big)^j_\kappa - \partial_{x_j}\big(\eta_\varepsilon\hat{\varrho}_{\varepsilon_n}\hat{u}_{\varepsilon_n}\big)^i_\kappa=\partial_{x_i}\hat{Y}^j_{\varepsilon_n,\kappa}-\partial_{x_j}\hat{Y}^i_{\varepsilon_n,\kappa}
\end{align}
for $ i,j=1,2,3$.

By using the relation \eqref{momentumCommutator}, we get that \eqref{1d3a} with $i=2$ minus \eqref{1d2a} with $i=3$ yields
\begin{align}
\label{1dcommutator1a}
\varepsilon \mathrm{d}\,\hat{\Upsilon}_{\varepsilon_n,\kappa}^{23}-\partial_{x_3}\hat{Y}^1_{\varepsilon_n,\kappa}\mathrm{d}t  =\partial_{x_3x_1}\hat{\Psi}_{\varepsilon_n,\kappa} \mathrm{d}t 
+ 
\varepsilon Q^{23}_{\varepsilon_n,\kappa}\mathrm{d}t
 +
  \varepsilon  \Phi^{23}_{\varepsilon_n,\kappa}\,\mathrm{d}\tilde{W}_{n}
\end{align}
having used the additional decomposition \eqref{sov-grad-part}. In analogy with \eqref{momentumCommutator},
\begin{align*}
Q^{ij}_{\varepsilon_n,\kappa}:=\partial_{x_i}Q^j_{\varepsilon_n,\kappa}-\partial_{x_j}Q^i_{\varepsilon_n,\kappa}
\end{align*}
with an identical notation for the noise term. Similar to \eqref{1dcommutator1a}, \eqref{1d3a} with $i=1$ minus \eqref{1d1a} with $i=3$ yields
\begin{align}
\label{1dcommutator2a}
\varepsilon \mathrm{d}\,\hat{\Upsilon}_{\varepsilon_n,\kappa}^{13} + \partial_{x_3}\hat{Y}_{\varepsilon_n,\kappa}^2\mathrm{d}t   =-\partial_{x_3x_2}  \hat{\Psi}_{\varepsilon_n,\kappa} \mathrm{d}t+ 
 \varepsilon Q^{13}_{\varepsilon_n,\kappa}\mathrm{d}t
 +
  \varepsilon  \Phi^{13}_{\varepsilon_n,\kappa} \,\mathrm{d}\tilde{W}_{n}.
\end{align}
Lastly, \eqref{1d2a} with $i=1$ minus \eqref{1d1a}with $i=2$ gives
\begin{align}
\label{1dcommutator3a}
\varepsilon \mathrm{d}\,\hat{\Upsilon}_{\varepsilon_n,\kappa}^{12}+ \mathrm{div}_h \big[ \hat{\mathbf{Y}}_{\varepsilon_n,\kappa} \big]_h \mathrm{d}t   =  -\Delta_h \hat{\Psi}_{\varepsilon_n,\kappa} \mathrm{d}t
+
 \varepsilon Q^{12}_{\varepsilon_n,\kappa} \mathrm{d}t
 +
 \varepsilon  \Phi^{12}_{\varepsilon_n,\kappa} \mathrm{d}\tilde{W}_{n}
\end{align}
by the use of \eqref{sov-grad-part}  above. Now using the fact that $\mathrm{div}\,\hat{\mathbf{Y}}_{\varepsilon_n,\kappa}=0$, we have that
\begin{equation}
\begin{aligned}
\label{messy1}
\mathrm{div}\big( \hat{\mathbf{Y}}_{\varepsilon_n,\kappa} &\otimes \hat{\mathbf{Y}}_{\varepsilon_n,\kappa}\big)=
\hat{\mathbf{Y}}_{\varepsilon_n,\kappa} \cdot\nabla \hat{\mathbf{Y}}_{\varepsilon_n,\kappa}
\\
&=\frac{1}{2}\nabla \left\vert \hat{\mathbf{Y}}_{\varepsilon_n,\kappa}\right\vert^2 - \hat{\mathbf{Y}}_{\varepsilon_n,\kappa}\times\mathrm{curl}\left( \hat{\mathbf{Y}}_{\varepsilon_n,\kappa}\right).
\end{aligned}
\end{equation}
Furthermore, by linearity and commutativity of the curl and derivative operators, we can use \eqref{verticalAvgAndRest1}--\eqref{verticalAvgAndRest2} to get
\begin{equation}
\begin{aligned}
\label{messy2}
\mathrm{div}&\left( \hat{\mathbf{Y}}_{\varepsilon_n,\kappa} \otimes \hat{\mathbf{Y}}_{\varepsilon_n,\kappa}\right)
=
\frac{1}{2}\nabla \left\vert \hat{\mathbf{Y}}_{\varepsilon_n,\kappa}\right\vert^2 
- \ulcorner \hat{\mathbf{Y}}_{\varepsilon_n,\kappa} \urcorner \times\mathrm{curl}\,\ulcorner \hat{\mathbf{Y}}_{\varepsilon_n,\kappa} \urcorner
\\&
- \partial_{x_3}  \left[   I\big(\hat{\mathbf{Y}}_{\varepsilon_n,\kappa} \big)  \times\mathrm{curl}\,\ulcorner \hat{\mathbf{Y}}_{\varepsilon_n,\kappa} \urcorner   +    \ulcorner \hat{\mathbf{Y}}_{\varepsilon_n,\kappa} \urcorner \times   \mathrm{curl}\left(I\big( \hat{\mathbf{Y}}_{\varepsilon_n,\kappa} \big) \right)  \right] 
\\&
- \partial_{x_3}I\big(\hat{\mathbf{Y}}_{\varepsilon_n,\kappa} \big) \times   \mathrm{curl}\left(\partial_{x_3} I\big( \hat{\mathbf{Y}}_{\varepsilon_n,\kappa} \big) \right)
\end{aligned}
\end{equation}
where
\begin{align}
\label{messy2a}
\ulcorner 
\partial_{x_3}  \left[   I\big(\hat{\mathbf{Y}}_{\varepsilon_n,\kappa} \big)  \times\mathrm{curl}\,\ulcorner \hat{\mathbf{Y}}_{\varepsilon_n,\kappa} \urcorner   +    \ulcorner \hat{\mathbf{Y}}_{\varepsilon_n,\kappa} \urcorner \times   \mathrm{curl}\left(I\big( \hat{\mathbf{Y}}_{\varepsilon_n,\kappa} \big) \right)  \right] 
 \urcorner=0.
\end{align}
With \eqref{messy2a} in hand, we wish to show that the last term in \eqref{messy2} above converges to zero in a suitable sense. To see this, we perform a direct computation using \eqref{momentumCommutator} to gain
\begin{equation}
\begin{aligned}
\label{curlMess}
  \Big[ \partial_{x_3}&I\big(\hat{\mathbf{Y}}_{\varepsilon_n,\kappa} \big) \times   \mathrm{curl}\left(\partial_{x_3} I\big( \hat{\mathbf{Y}}_{\varepsilon_n,\kappa} \big) \right) \Big]^1
\\&=
\big(\partial_{x_3}I\big(\hat{Y}_{\varepsilon_n,\kappa}^2 \big) \big)
 \partial_{x_3}
\big[\partial_{x_1}I\big(\hat{Y}_{\varepsilon_n,\kappa}^2 \big) 
-
\partial_{x_2}I\big(\hat{Y}_{\varepsilon_n,\kappa}^1 \big) \big]
\\&-
\big(\partial_{x_3}I\big(\hat{Y}_{\varepsilon_n,\kappa}^3 \big) \big)
 \partial_{x_3}
\big[\partial_{x_3}I\big(\hat{Y}_{\varepsilon_n,\kappa}^1 \big) 
-
\partial_{x_1}I\big(\hat{Y}_{\varepsilon_n,\kappa}^3 \big) \big]
\\&=
\big(\partial_{x_3}I\big(\hat{Y}_{\varepsilon_n,\kappa}^2 \big) \big)
 \partial_{x_3}
I\big(\hat{\Upsilon}_{\varepsilon_n,\kappa}^{12} \big) -
\big(\partial_{x_3}I\big(\hat{Y}_{\varepsilon_n,\kappa}^3 \big) \big)
 \partial_{x_3}
I\big(\hat{\Upsilon}_{\varepsilon_n,\kappa}^{31} \big)
\\&=
 \partial_{x_3}\Big[\big(\partial_{x_3}I\big(\hat{Y}_{\varepsilon_n,\kappa}^2 \big) \big)
I\big(\hat{\Upsilon}_{\varepsilon_n,\kappa}^{12} \big) \Big]
-
\big(\partial_{x_3x_3}I\big(\hat{Y}_{\varepsilon_n,\kappa}^2 \big) \big)
I\big(\hat{\Upsilon}_{\varepsilon_n,\kappa}^{12} \big) 
\\&+
I\big(\mathrm{div}_h\big[\mathbf{\hat{Y}}_{\varepsilon_n,\kappa} \big]_h \big)
 \partial_{x_3}
I\big(\hat{\Upsilon}_{\varepsilon_n,\kappa}^{31} \big)
\end{aligned}
\end{equation}
where we have used $\mathrm{div}\big[\mathbf{\hat{Y}}_{\varepsilon_n,\kappa} \big]=0 $ and hence $\mathrm{div}_h\big[\mathbf{\hat{Y}}_{\varepsilon_n,\kappa} \big]_h = -\partial_{x_3}{\hat{Y}}_{\varepsilon_n,\kappa}^3$ in the last step above.

Since the vertical average of the first term on the right-hand side of \eqref{curlMess} vanishes, we concentrate on the last two terms. For this, we first use \eqref{1dcommutator2a} to formally obtain
\begin{equation}
\begin{aligned}
\label{1dcommutator2ax}
-\varepsilon &\Big[\mathrm{d}\,
 \partial_{x_3}I \big(\hat{\Upsilon}_{\varepsilon_n,\kappa}^{13} \big)\Big] I\big(\hat{\Upsilon}_{\varepsilon_n,\kappa}^{12} \big) 
 =  
\big(\partial_{x_3x_3}I\big(\hat{Y}_{\varepsilon_n,\kappa}^2 \big) \big)
I\big(\hat{\Upsilon}_{\varepsilon_n,\kappa}^{12} \big) \,\mathrm{d}t
 \\
& 
  +
 \partial_{x_3x_2}  \hat{\Psi}_{\varepsilon_n,\kappa} I\big(\hat{\Upsilon}_{\varepsilon_n,\kappa}^{12} \big)\mathrm{d}t
- 
 \varepsilon\big(  Q^{13}_{\varepsilon_n,\kappa} - \ulcorner \partial_{x_1} Q^3_{\varepsilon_n,\kappa} \urcorner \big) I\big(\hat{\Upsilon}_{\varepsilon_n,\kappa}^{12} \big) \mathrm{d}t
\\ 
&-
  \varepsilon \big(  \Phi^{13}_{\varepsilon_n,\kappa} - \ulcorner \partial_{x_1} \Phi^3_{\varepsilon_n,\kappa} \urcorner \big) I\big(\hat{\Upsilon}_{\varepsilon_n,\kappa}^{12} \big) \,\mathrm{d}\tilde{W}_{n}
\end{aligned}
\end{equation}
since
\begin{align*}
\partial_{x_3}I\Big[ \partial_{x_3x_2}  \hat{\Psi}_{\varepsilon_n,\kappa} \Big] &= \partial_{x_3x_2}  \hat{\Psi}_{\varepsilon_n,\kappa} 
-
\partial_{x_3}  \ulcorner\partial_{x_2}  \hat{\Psi}_{\varepsilon_n,\kappa} \urcorner = \partial_{x_3x_2}  \hat{\Psi}_{\varepsilon_n,\kappa} 
\\
\partial_{x_3}I \big[Q^{13}_{\varepsilon_n,\kappa} \big]
&=
Q^{13}_{\varepsilon_n,\kappa} - \ulcorner \partial_{x_1}Q^{3}_{\varepsilon_n,\kappa} \urcorner +  \partial_{x_3}\ulcorner Q^{1}_{\varepsilon_n,\kappa} \urcorner
\\&=
Q^{13}_{\varepsilon_n,\kappa}  - \ulcorner \partial_{x_1}Q^{3}_{\varepsilon_n,\kappa} \urcorner
\end{align*}
and similarly for the noise term. By using \eqref{1dcommutator3a}, we also formally  obtain 
\begin{equation}
\begin{aligned}
\label{1dcommutator2ay}
-&\varepsilon \big[ \mathrm{d}\,I\big(\hat{\Upsilon}_{\varepsilon_n,\kappa}^{12} \big) \big] \partial_{x_3}
I\big(\hat{\Upsilon}_{\varepsilon_n,\kappa}^{13} \big)
=
\varepsilon \big[ \mathrm{d}\,I\big(\hat{\Upsilon}_{\varepsilon_n,\kappa}^{12} \big) \big] \partial_{x_3}
I\big(\hat{\Upsilon}_{\varepsilon_n,\kappa}^{31} \big)
\\&=
-I\big(\mathrm{div}_h\big[\mathbf{\hat{Y}}_{\varepsilon_n,\kappa} \big]_h \big)
 \partial_{x_3}
I\big(\hat{\Upsilon}_{\varepsilon_n,\kappa}^{31} \big) \,\mathrm{d}t
-
I \big(\Delta_h \hat{\Psi}_{\varepsilon_n,\kappa} \big)
\partial_{x_3}
I\big(\hat{\Upsilon}_{\varepsilon_n,\kappa}^{31} \big) \,\mathrm{d}t
\\&+
 \varepsilon I\big( Q^{12}_{\varepsilon_n,\kappa}) \partial_{x_3}
I\big(\hat{\Upsilon}_{\varepsilon_n,\kappa}^{31} \big)\,\mathrm{d}t
+
 \varepsilon I \big( \Phi^{12}_{\varepsilon_n,\kappa}\big)\,
 \partial_{x_3}
I\big(\hat{\Upsilon}_{\varepsilon_n,\kappa}^{31} \big)\,\mathrm{d}\tilde{W}_{n}.
\end{aligned}
\end{equation}
Now by I\^to's formula, \cite[Theorem A.4.1]{breit2017stoch}, it follows from \eqref{1dcommutator2ax} and \eqref{1dcommutator2ay} that
\begin{equation}
\begin{aligned}
\label{1dcommutator2az}
-\varepsilon  &\mathrm{d}\,\big[ I\big(\hat{\Upsilon}_{\varepsilon_n,\kappa}^{12} \big)  \partial_{x_3}
I\big(\hat{\Upsilon}_{\varepsilon_n,\kappa}^{13} \big) \big]
=
-I\big(\mathrm{div}_h\big[\mathbf{\hat{Y}}_{\varepsilon_n,\kappa} \big]_h \big)
 \partial_{x_3}
I\big(\hat{\Upsilon}_{\varepsilon_n,\kappa}^{31} \big) \,\mathrm{d}t
\\&-
I \big(\Delta_h \hat{\Psi}_{\varepsilon_n,\kappa} \big)
\partial_{x_3}
I\big(\hat{\Upsilon}_{\varepsilon_n,\kappa}^{31} \big) \,\mathrm{d}t
+
 \varepsilon I\big( Q^{12}_{\varepsilon_n,\kappa}) \partial_{x_3}
I\big(\hat{\Upsilon}_{\varepsilon_n,\kappa}^{31} \big)\mathrm{d}t
\\&+
\big(\partial_{x_3x_3}I\big(\hat{Y}_{\varepsilon_n,\kappa}^2 \big) \big)
I\big(\hat{\Upsilon}_{\varepsilon_n,\kappa}^{12} \big) \,\mathrm{d}t 
  +
 \partial_{x_3x_2}  \hat{\Psi}_{\varepsilon_n,\kappa} I\big(\hat{\Upsilon}_{\varepsilon_n,\kappa}^{12} \big)\mathrm{d}t
 \\&- 
 \varepsilon\big(  Q^{13}_{\varepsilon_n,\kappa} - \ulcorner \partial_{x_1} Q^3_{\varepsilon_n,\kappa} \urcorner \big) I\big(\hat{\Upsilon}_{\varepsilon_n,\kappa}^{12} \big) \mathrm{d}t
\\&+
 \varepsilon \,
\sum_{k\in \mathbb{N}}
 I \big(  \mathbf{g}_{\varepsilon_n,\kappa, k}^{12}\big)\,
 \partial_{x_3}
I\big(\hat{\Upsilon}_{\varepsilon_n,\kappa}^{31} \big)\, \mathrm{d}\tilde{\beta}_{k}^{n}
\\ 
&-
  \varepsilon\,
\sum_{k\in \mathbb{N}} \big(  \mathbf{g}_{\varepsilon_n,\kappa,k}^{13} - \ulcorner \partial_{x_1} \mathbf{g}_{\varepsilon_n,\kappa,k}^3 \urcorner \big) I\big(\hat{\Upsilon}_{\varepsilon_n,\kappa}^{12} \big) \, \mathrm{d}\tilde{\beta}_{k}^{n}
  \\ 
&-\varepsilon\,
\sum_{k\in \mathbb{N}}
I(\mathbf{g}_{\varepsilon_n,\kappa,k}^{12})\big[\mathbf{g}_{\varepsilon_n,\kappa,k}^{13}
-
\ulcorner \partial_{x_1}\mathbf{g}_{\varepsilon_n,\kappa,k}^{3} \urcorner \big] \, 
\mathrm{d}t
\end{aligned}
\end{equation}
where
\begin{align*}
\mathbf{g}_{\varepsilon_n,\kappa,k}^{i} =[\Phi_{\varepsilon_n,\kappa}(e_k)]^i, \quad \mathbf{g}_{\varepsilon_n,\kappa,k}^{ij} = \partial_{x_i}[\Phi_{\varepsilon_n,\kappa}(e_k)]^j - \partial_{x_j}[\Phi_{\varepsilon_n,\kappa}(e_k)]^i.
\end{align*}
We can now rearrange \eqref{1dcommutator2az} to get
\begin{equation}
\begin{aligned}
\label{1dcommutator2azx}
-&
\big(\partial_{x_3x_3}I\big(\hat{Y}_{\varepsilon_n,\kappa}^2 \big) \big)
I\big(\hat{\Upsilon}_{\varepsilon_n,\kappa}^{12} \big) \,\mathrm{d}t 
+
I\big(\mathrm{div}_h\big[\mathbf{\hat{Y}}_{\varepsilon_n,\kappa} \big]_h \big)
 \partial_{x_3}
I\big(\hat{\Upsilon}_{\varepsilon_n,\kappa}^{31} \big) \,\mathrm{d}t
\\&=
\varepsilon  \mathrm{d}\,\big[ I\big(\hat{\Upsilon}_{\varepsilon_n,\kappa}^{12} \big)  \partial_{x_3}
I\big(\hat{\Upsilon}_{\varepsilon_n,\kappa}^{13} \big) \big]
+
 \varepsilon I\big( Q^{12}_{\varepsilon_n,\kappa}) \partial_{x_3}
I\big(\hat{\Upsilon}_{\varepsilon_n,\kappa}^{31} \big)\mathrm{d}t
 \\&- 
 \varepsilon\big(  Q^{13}_{\varepsilon_n,\kappa} - \ulcorner \partial_{x_1} Q^3_{\varepsilon_n,\kappa} \urcorner \big) I\big(\hat{\Upsilon}_{\varepsilon_n,\kappa}^{12} \big) \mathrm{d}t
\\&+
 \varepsilon \,
\sum_{k\in \mathbb{N}}
 I \big(  \mathbf{g}_{\varepsilon_n,\kappa,k}^{12}\big)\,
 \partial_{x_3}
I\big(\hat{\Upsilon}_{\varepsilon_n,\kappa}^{31} \big)\, \mathrm{d}\tilde{\beta}_{k}^{n}
\\ 
&-
  \varepsilon\,
\sum_{k\in \mathbb{N}} \big(  \mathbf{g}_{\varepsilon_n,\kappa,k}^{13} - \ulcorner \partial_{x_1} \mathbf{g}_{\varepsilon_n,\kappa,k}^3 \urcorner \big) I\big(\hat{\Upsilon}_{\varepsilon_n,\kappa}^{12} \big) \, \mathrm{d}\tilde{\beta}_{k}^{n}
  \\ 
&-\varepsilon\,
\sum_{k\in \mathbb{N}}
I(\mathbf{g}_{\varepsilon_n,\kappa,k}^{12})\big[\mathbf{g}_{\varepsilon_n,\kappa,k}^{13}
-
\ulcorner \partial_{x_1}\mathbf{g}_{\varepsilon_n,\kappa,k}^{3} \urcorner \big] \, \mathrm{d}t
\\&-
I \big(\Delta_h \hat{\Psi}_{\varepsilon_n,\kappa} \big)
\partial_{x_3}
I\big(\hat{\Upsilon}_{\varepsilon_n,\kappa}^{31} \big) \,\mathrm{d}t
  +
 \partial_{x_3x_2}  \hat{\Psi}_{\varepsilon_n,\kappa} I\big(\hat{\Upsilon}_{\varepsilon_n,\kappa}^{12} \big)\mathrm{d}t.
\end{aligned}
\end{equation}
Substituting  \eqref{1dcommutator2ax}--\eqref{1dcommutator2azx} into \eqref{curlMess}, we have shown that
\begin{equation}
\begin{aligned}
\label{curl0Mess}
  \Big[ &\partial_{x_3}I\big(\hat{\mathbf{Y}}_{\varepsilon_n,\kappa} \big) \times   \mathrm{curl}\left(\partial_{x_3} I\big( \hat{\mathbf{Y}}_{\varepsilon_n,\kappa} \big) \right) \Big]^1 \,\mathrm{d}t
\\&=
 \partial_{x_3}\Big[\big(\partial_{x_3}I\big(\hat{Y}_{\varepsilon_n,\kappa}^2 \big) \big)
I\big(\hat{\Upsilon}_{\varepsilon_n,\kappa}^{12} \big) \Big] \,\mathrm{d}t
+
\varepsilon  \mathrm{d}\,\big[ I\big(\hat{\Upsilon}_{\varepsilon_n,\kappa}^{12} \big)  \partial_{x_3}
I\big(\hat{\Upsilon}_{\varepsilon_n,\kappa}^{13} \big) \big]
\\&+
 \varepsilon I\big( Q^{12}_{\varepsilon_n,\kappa}) \partial_{x_3}
I\big(\hat{\Upsilon}_{\varepsilon_n,\kappa}^{31} \big)\mathrm{d}t
 - 
 \varepsilon\big(  Q^{13}_{\varepsilon_n,\kappa} - \ulcorner \partial_{x_1} Q^3_{\varepsilon_n,\kappa} \urcorner \big) I\big(\hat{\Upsilon}_{\varepsilon_n,\kappa}^{12} \big) \mathrm{d}t
\\&+
 \varepsilon \,
\sum_{k\in \mathbb{N}}
 I \big(  \mathbf{g}_{\varepsilon_n,\kappa,k}^{12}\big)\,
 \partial_{x_3}
I\big(\hat{\Upsilon}_{\varepsilon_n,\kappa}^{31} \big)\, \mathrm{d}\tilde{\beta}_{k}^{n}
\\ 
&-
  \varepsilon\,
\sum_{k\in \mathbb{N}} \big(  \mathbf{g}_{\varepsilon_n,\kappa,k}^{13} - \ulcorner \partial_{x_1} \mathbf{g}_{\varepsilon_n,\kappa,k}^3 \urcorner \big) I\big(\hat{\Upsilon}_{\varepsilon_n,\kappa}^{12} \big) \, \mathrm{d}\tilde{\beta}_{k}^{n}
  \\ 
&-\varepsilon\,
\sum_{k\in \mathbb{N}}
I(\mathbf{g}_{\varepsilon_n,\kappa,k}^{12})\big[\mathbf{g}_{\varepsilon_n,\kappa,k}^{13}
-
\ulcorner \partial_{x_1}\mathbf{g}_{\varepsilon_n,\kappa,k}^{3} \urcorner \big]  \, \mathrm{d}t
\\&-
I \big(\Delta_h \hat{\Psi}_{\varepsilon_n,\kappa} \big)
\partial_{x_3}
I\big(\hat{\Upsilon}_{\varepsilon_n,\kappa}^{31} \big) \,\mathrm{d}t
  +
 \partial_{x_3x_2}  \hat{\Psi}_{\varepsilon_n,\kappa} I\big(\hat{\Upsilon}_{\varepsilon_n,\kappa}^{12} \big)\mathrm{d}t
 \\&
 =:J_1 +\ldots + J_9.
\end{aligned}
\end{equation}
Similarly, we gain
\begin{equation}
\begin{aligned}
\label{curl1Mess}
  \Big[ &\partial_{x_3}I\big(\hat{\mathbf{Y}}_{\varepsilon_n,\kappa} \big) \times   \mathrm{curl}\left(\partial_{x_3} I\big( \hat{\mathbf{Y}}_{\varepsilon_n,\kappa} \big) \right) \Big]^2 \,\mathrm{d}t
\\&=
 \partial_{x_3}\Big[\big(\partial_{x_3}I\big(\hat{Y}_{\varepsilon_n,\kappa}^1 \big) \big)
I\big(\hat{\Upsilon}_{\varepsilon_n,\kappa}^{21} \big) \Big] \,\mathrm{d}t
+
\varepsilon  \mathrm{d}\,\big[ I\big(\hat{\Upsilon}_{\varepsilon_n,\kappa}^{12} \big)  \partial_{x_3}
I\big(\hat{\Upsilon}_{\varepsilon_n,\kappa}^{23} \big) \big]
\\&+
 \varepsilon I\big( Q^{12}_{\varepsilon_n,\kappa}) \partial_{x_3}
I\big(\hat{\Upsilon}_{\varepsilon_n,\kappa}^{23} \big)\mathrm{d}t
 - 
 \varepsilon\big(  Q^{23}_{\varepsilon_n,\kappa} - \ulcorner \partial_{x_2} Q^3_{\varepsilon_n,\kappa} \urcorner \big) I\big(\hat{\Upsilon}_{\varepsilon_n,\kappa}^{12} \big) \mathrm{d}t
\\&+
 \varepsilon \,
\sum_{k\in \mathbb{N}}
 I \big(  \mathbf{g}_{\varepsilon_n,\kappa,k}^{12}\big)\,
 \partial_{x_3}
I\big(\hat{\Upsilon}_{\varepsilon_n,\kappa}^{23} \big)\, \mathrm{d}\tilde{\beta}_{k}^{n}
\\ 
&-
  \varepsilon\,
\sum_{k\in \mathbb{N}} \big(  \mathbf{g}_{\varepsilon_n,\kappa,k}^{23} - \ulcorner \partial_{x_2} \mathbf{g}_{\varepsilon_n,\kappa,k}^3 \urcorner \big) I\big(\hat{\Upsilon}_{\varepsilon_n,\kappa}^{12} \big) \, \mathrm{d}\tilde{\beta}_{k}^{n}
  \\ 
&-\varepsilon\,
\sum_{k\in \mathbb{N}}
I(\mathbf{g}_{\varepsilon_n,\kappa,k}^{12})\big[\mathbf{g}_{\varepsilon_n,\kappa,k}^{23}
-
\ulcorner \partial_{x_2}\mathbf{g}_{\varepsilon_n,\kappa,k}^{3} \urcorner \big]  \, \mathrm{d}t
\\&+
I \big(\Delta_h \hat{\Psi}_{\varepsilon_n,\kappa} \big)
\partial_{x_3}
I\big(\hat{\Upsilon}_{\varepsilon_n,\kappa}^{23} \big) \,\mathrm{d}t
  -
 \partial_{x_3x_1}  \hat{\Psi}_{\varepsilon_n,\kappa} I\big(\hat{\Upsilon}_{\varepsilon_n,\kappa}^{12} \big)\mathrm{d}t 
 \\&
 =:K_1 +\ldots + K_9
\end{aligned}
\end{equation}
where both $J_1$ and $K_1$ vanishes after the taking of vertical averages. Furthermore, given  that by Proposition \ref{prop:Jakubow0}, the smoothness of $\mathbf{g}_{\varepsilon_n,\kappa,k}$ is preserved and that the terms $J_2, \ldots, J_7$ and $K_2, \ldots, K_7$ are smooth and hence uniformly bounded in $\varepsilon_n$ in suitable Bochner spaces, recall Section \ref{sec:proofofmain}, we are left to worry about the terms $J_8,J_9,K_8$ and $K_9$. This is because, per the explanation for  $J_2, \ldots, J_7$ and $K_2, \ldots, K_7$ above, we obtain that for any $\bm{\phi}=[\underline{\bm{\phi}}(x_h),0]$ with $\underline{\bm{\phi}}\in C^\infty_{c,\mathrm{div}_h}(\mathbb{R}^2)$, any $\psi\in L^2(\Omega \times (0,t))$ and any $t\in[0,T]$,
 \begin{equation}
\begin{aligned}
\tilde{\mathbb{E}}\, \int_0^t\big\langle J_i\,,\,  \bm{\phi}\, \psi \big\rangle\,\mathrm{d}\tau
&\lesssim \varepsilon \big\Vert J_i \big\Vert_{L^{p_1}_\omega L^{p_2}_t L^{p_3}_x}  \big\Vert \bm{\phi} \big\Vert_{L^{q}_x}  \big\Vert \psi \big\Vert_{L^{2}_{\omega,t}}
 \\
 &\lesssim \varepsilon \rightarrow 0
\end{aligned}
\end{equation}
for suitable $p_i \geq 1$ and $q \geq 1$ as $\varepsilon \rightarrow 0$. The same applies for the $K_i$'s with $i=2,\ldots,7$. The noise terms follow in a similar manner by the use of It\^o isometry and the fact that squared-integrable functions (in probability) are integrable. For these extra terms $J_8,J_9,K_8$ and $K_9$, we observe that for any $t\in(0,T)$ and any $i,j=1,2,3$,
\begin{equation}
\begin{aligned}
\label{vanishVanish}
\int_0^t \Vert \partial_{x_jx_i}  \hat{\Psi}_{\varepsilon_n,\kappa} &\Vert^2_{L^2(K)} \, \mathrm{d}t
\lesssim
\int_0^T \Vert \partial_{x_i} \wp_\kappa \Vert^2_{L^\frac{2\gamma}{2\gamma-1}(K)}\Vert \nabla \hat{\Psi}_{\varepsilon_n} \Vert^2_{L^\frac{2\gamma}{\gamma+1}(K)} \, \mathrm{d}t 
\\&
\lesssim_\kappa
\int_0^T \Vert \mathcal{Q}(\hat{\varrho}_{\varepsilon_n}\hat{\mathbf{u}}_{\varepsilon_n})  \Vert^2_{L^\frac{2\gamma}{\gamma+1}(K)} \, \mathrm{d}t \rightarrow0
\end{aligned}
\end{equation}
$\tilde{\mathbb{P}}$-a.s. for $K\Subset \mathcal{O}$ as $n\rightarrow\infty$. This follows by \eqref{acousticMomen} hence terms $J_8,J_9,K_8$ and $K_9$ also vanishes in the limit.

We can now take the vertical average in \eqref{messy1} and combine it with \eqref{messy2}, \eqref{curl0Mess}, \eqref{curl1Mess} (keeping in mind, the argument after \eqref{curl1Mess}) and \eqref{vanishVanish} to get
 that for any $\bm{\phi}=[\underline{\bm{\phi}}(x_h),0]$ with $\underline{\bm{\phi}}\in C^\infty_{c,\mathrm{div}_h}(\mathbb{R}^2)$, any $\psi\in L^2(\Omega \times (0,t))$ and $t\in[0,T]$ that
 \begin{equation}
\begin{aligned}
-\lim_{n\rightarrow\infty}&\tilde{\mathbb{E}}\, \int_0^t\left\langle  \hat{\mathbf{Y}}_{\varepsilon_n,\kappa} \otimes \hat{\mathbf{Y}}_{\varepsilon_n,\kappa}\,,\, \nabla \bm{\phi}\, \psi \right\rangle\,\mathrm{d}\tau
 \\
 &=
\lim_{n\rightarrow\infty} \tilde{\mathbb{E}}\,  \int_0^t\int_{\mathbb{R}^2}\ulcorner   \mathrm{div}\big( \hat{\mathbf{Y}}_{\varepsilon_n,\kappa} \otimes \hat{\mathbf{Y}}_{\varepsilon_n,\kappa}\big)
 \urcorner\cdot \bm{\phi} \, \psi\,\mathrm{d}x_h\,\mathrm{d}\tau
 \\
 &=
 \lim_{n\rightarrow\infty} \tilde{\mathbb{E}}\, 
\int_0^t\int_{\mathbb{R}^2}\ulcorner \Bigg[   \frac{1}{2}\nabla \left\vert \hat{\mathbf{Y}}_{\varepsilon_n,\kappa}\right\vert^2 
\\&
- \ulcorner \hat{\mathbf{Y}}_{\varepsilon_n,\kappa} \urcorner \times\mathrm{curl}\ulcorner \hat{\mathbf{Y}}_{\varepsilon_n,\kappa} \urcorner
\Bigg] \urcorner\cdot \bm{\phi}\,\psi\,\mathrm{d}x_h\mathrm{d}\tau
\\
&=\tilde{\mathbb{E}}\, 
\int_0^t\int_{\mathbb{R}^2} \mathrm{div}\big( \hat{\mathbf{U}}_{\kappa} \otimes \hat{\mathbf{U}}_{\kappa}\big)
 \cdot \bm{\phi} \, \psi\,\mathrm{d}x_h\,\mathrm{d}\tau.
\end{aligned}
\end{equation}
This is because the gradient term vanishes after integration by part whereas  convergence to velocity holds for the vertical average of the solenoidal part of momentum. cf.  \eqref{vertSolenoiStrongConv}.
Subsequently, since $\psi$ is arbitrary, we gain
\begin{equation}
\begin{aligned}
\label{messy4}
-&\lim_{n\rightarrow\infty}\int_0^t\left\langle  \hat{\mathbf{Y}}_{\varepsilon_n,\kappa} \otimes \hat{\mathbf{Y}}_{\varepsilon_n,\kappa} , \nabla \bm{\phi} \right\rangle \mathrm{d}\tau
\\&=
\int_0^t\int_{\mathbb{R}^2} \mathrm{div}\big( \hat{\mathbf{U}}_{\kappa} \otimes \hat{\mathbf{U}}_{\kappa}\big)
 \cdot \bm{\phi} \mathrm{d}x_h \mathrm{d}\tau
\end{aligned}
\end{equation}
$\tilde{\mathbb{P}}$-a.s. possibly for a further subsequence. 
\end{proof}
\subsection{Identifying the convective term}
\label{sec:convectiveTerm}
We show in this section that in identifying the limit of the convective term, we may essentially replace the sequence  $\left(\hat{\varrho}_{\varepsilon_n}\hat{\mathbf{u}}_{\varepsilon_n}  \otimes \hat{\mathbf{u}}_{\varepsilon_n}\right)$ by the mollified term $\left(\hat{\varrho}_{\varepsilon_n}\hat{\mathbf{u}}_{\varepsilon_n} \right)_\kappa \otimes\left( \hat{\varrho}_{\varepsilon_n}\hat{\mathbf{u}}_{\varepsilon_n} \right)_\kappa$ since their limits coincide. This follows from the proof of the lemma below.
\begin{lem}
\label{lem:convectiveConvergence}
For any $\bm{\phi}=[\underline{\bm{\phi}}(x_h),0]$ with $\underline{\bm{\phi}}\in C^\infty_{c,\mathrm{div}_h}(\mathbb{R}^2)$ and for all $t\in[0,T]$, the converges
\begin{equation}
\begin{aligned}
\label{convectiveConver}
\lim_{n\rightarrow \infty}\int_0^t\left\langle  \hat{\varrho}_{\varepsilon_n}\hat{\mathbf{u}}_{\varepsilon_n} \otimes \hat{\mathbf{u}}_{\varepsilon_n}\,,\, \nabla\bm{\phi}\right\rangle\,\mathrm{d}\tau
&=
\int_0^t\left\langle  \hat{\mathbf{U}} \otimes \hat{\mathbf{U}}\,,\, \nabla\bm{\phi}\right\rangle\,\mathrm{d}\tau
\end{aligned}
\end{equation}
holds $\tilde{\mathbb{P}}$-a.s.
\end{lem}

\begin{proof}
To show \eqref{convectiveConver}, we first make the following decomposition\footnote{This decomposition is not unique and in fact, a much simpler variant suffices.}
\begin{equation}
\begin{aligned}
\label{decompositionConvective}
&\hat{\varrho}_{\varepsilon_n}\hat{\mathbf{u}}_{\varepsilon_n} \otimes \hat{\mathbf{u}}_{\varepsilon_n} = \left[ (1-\hat{\varrho}_{\varepsilon_n}) \hat{\mathbf{u}}_{\varepsilon_n}\right]_\kappa
\otimes \left[ (\hat{\varrho}_{\varepsilon_n}-1)
\hat{\mathbf{u}}_{\varepsilon_n}\right]_\kappa
\\
&+
 (\hat{\varrho}_{\varepsilon_n}-1)
\hat{\mathbf{u}}_{\varepsilon_n}
\otimes \hat{\mathbf{u}}_{\varepsilon_n}
+
\left[ (1-\hat{\varrho}_{\varepsilon_n})
\hat{\mathbf{u}}_{\varepsilon_n}\right]_\kappa
\otimes \hat{\mathbf{u}}_{\varepsilon_n,\kappa}
\\
&+
\hat{\mathbf{u}}_{\varepsilon_n,\kappa} \otimes \left[ (1-\hat{\varrho}_{\varepsilon_n})
\hat{\mathbf{u}}_{\varepsilon_n}\right]_\kappa
+
\hat{\mathbf{u}}_{\varepsilon_n,\kappa}\otimes \left(\hat{\mathbf{u}}_{\varepsilon_n}-\hat{\mathbf{u}}_{\varepsilon_n,\kappa} \right)
\\
&+
\big(\hat{\mathbf{u}}_{\varepsilon_n} - \hat{\mathbf{u}}_{\varepsilon_n,\kappa}  \big) \otimes \hat{\mathbf{u}}_{\varepsilon_n}
+
\big(\hat{\varrho}_{\varepsilon_n}\hat{\mathbf{u}}_{\varepsilon_n}\big)_\kappa  \otimes\big(\hat{\varrho}_{\varepsilon_n}\hat{\mathbf{u}}_{\varepsilon_n}\big)_\kappa 
=:\sum_{i=1}^7J_i.
\end{aligned}
\end{equation}
Then we observe that for any ball $B_k$ of  radius $k>0$ and fixed regularizing kernel $\kappa>0$,
\begin{align*}
J_i \rightarrow 0 \quad \text{ in }\quad L^1((0,T)\times B_k)
\end{align*}
$\tilde{\mathbb{P}}$-a.s. for each $i\in\{ 1,2,3,4\}$ as $n\rightarrow \infty$. This follows from Lemma \ref{lem:strongDensity} and the equality of laws given by Proposition \ref{prop:Jakubow0}. 

Now we notice that for the random variable $\hat{\mathbf{u}}_{\varepsilon_n}\in  L^2(0,T;W^{1,2}(B_k))$,  one can find a  constant $c>0$ independent of both $\kappa$  and  $\varepsilon_n$ such that
\begin{align}
\label{mollVelToVel}
\left\Vert \hat{\mathbf{u}}_{\varepsilon_n}-\hat{\mathbf{u}}_{\varepsilon_n,\kappa}\right\Vert_{L^2(\Omega;L^2(0,T;L^p(B_k)))}
\lesssim\kappa^{{3}/{p}-{1}/{2}}\, 
\end{align}
holds uniformly in $\varepsilon_n>0$ for all $p\in[2,6)$, cf. \cite[Section 3.2]{gallagher2006weak}. 
It follows from a density argument and the Riesz representation theorem for integrable functions
  that for all $\bm{\phi}(x)\in C_c^\infty( \mathcal{O})$ and all $\bm{\phi}_1(\omega)\in L^\infty(\Omega)$, $\bm{\phi}_2(t)\in L^\infty(0,T)$, one can find a generic constant $c>0$ that is uniform in both $\varepsilon_n$ and $\kappa$ such that
\begin{equation}
\begin{aligned}
\Big\Vert&\big\langle \hat{\mathbf{u}}_{\varepsilon_n,\kappa}\otimes \big(\hat{\mathbf{u}}_{\varepsilon_n}-\hat{\mathbf{u}}_{\varepsilon_n,\kappa} \big)  \,,\, \nabla \bm{\phi} \big\rangle
 \Big\Vert_{L^1_{\omega,t}}
\\
& =
 \sup_{\Vert\bm{\phi}_1\bm{\phi}_2\Vert_\infty\leq1}\tilde{\mathbb{E}}\int_0^T\big\langle \hat{\mathbf{u}}_{\varepsilon_n,\kappa}\otimes \big(\hat{\mathbf{u}}_{\varepsilon_n}-\hat{\mathbf{u}}_{\varepsilon_n,\kappa} \big)  \,,\, \nabla \bm{\phi}\big\rangle\,\bm{\phi}_1(\omega)\bm{\phi}_2(t)\,\mathrm{d}t
 \\
&\lesssim
\Vert \nabla\bm{\phi}\Vert_{L^2_x}
\Vert \hat{\mathbf{u}}_{\varepsilon_n,\kappa} \Vert_{L^2_{\omega,t}L^6_x}
\,\Vert \hat{\mathbf{u}}_{\varepsilon_n}-\hat{\mathbf{u}}_{\varepsilon_n,\kappa}  \Vert_{L^2_{\omega,t}L^3_x}
\\
&\lesssim
\Vert \hat{\mathbf{u}}_{\varepsilon_n} \Vert_{L^2_{\omega,t}L^6_x}
\,\Vert \hat{\mathbf{u}}_{\varepsilon_n}-\hat{\mathbf{u}}_{\varepsilon_n,\kappa}  \Vert_{L^2_{\omega,t}W^{1,2}_x}^{1/2}
\,\Vert \hat{\mathbf{u}}_{\varepsilon_n}-\hat{\mathbf{u}}_{\varepsilon_n,\kappa}  \Vert_{L^2_{\omega,t}L^2_x}^{1/2}
\\&\lesssim 
\sqrt{\kappa}\,
\Vert \hat{\mathbf{u}}_{\varepsilon_n} \Vert_{L^2_{\omega,t}W^{1,2}_x}
\,\lesssim\sqrt{\kappa}.
\end{aligned}
\end{equation}
It follows that  as $n\rightarrow \infty$ and $\kappa\rightarrow0$, $\int_0^T\langle J_5\,,\, \nabla\bm{\phi}\rangle\,\mathrm{d}t$ vanishes in mean and hence in probability.  The same argument holds for $J_6$. 

As a consequence of  \eqref{acousticMomen}, \eqref{messy4}, \eqref{decompositionConvective} and the above convergence to zero results, one has  that
\begin{equation}
\begin{aligned}
\label{lastMess}
\lim_{n \rightarrow \infty}&\int_0^t\big\langle  \hat{\varrho}_{\varepsilon_n}\hat{\mathbf{u}}_{\varepsilon_n} \otimes \hat{\mathbf{u}}_{\varepsilon_n}\,,\, \nabla\bm{\phi}\big\rangle\,\mathrm{d}\tau
\\&=
\lim_{\kappa\rightarrow0}\lim_{n\rightarrow \infty} \int_0^t\left\langle  \hat{\mathbf{Y}}_{\varepsilon_n,\kappa} \otimes \hat{\mathbf{Y}}_{\varepsilon_n,\kappa}  \,,\, \nabla\bm{\phi}\right\rangle\,\mathrm{d}\tau
\\
&=
\int_0^t\left\langle  \hat{\mathbf{U}} \otimes \hat{\mathbf{U}}  \,,\, \nabla\bm{\phi}\right\rangle\,\mathrm{d}\tau
\end{aligned}
\end{equation}
for any $t\in(0,T]$.
\end{proof}

\begin{proof}[Proof of Theorem \ref{the:important}]
From Lemma \ref{lem:WeinerSeq}, $\tilde{W}_n$ are cylindrical Wiener processes  and as such, we can find a collection of mutually independent 1-D $(\tilde{\mathscr{F}}_t)$-Brownian motions $(\tilde{\beta}_k)_{k\in\mathbb{N}}$ and orthonormal basis $(e_k)_{k\in\mathbb{N}}$ such that $\tilde{W}=\sum_{k\in\mathbb{N}}\tilde{\beta}_k e_k$. We refer the reader to \cite[Lemma 2.1.35, Corollary 2.1.36]{breit2017stoch} for further details.

To show that  $[ (\tilde{\Omega}, \tilde{\mathscr{F}}_t,(\tilde{\mathscr{F}}_t),\tilde{\mathbb{P}}),\hat{\mathbf{U}}, \tilde{W}]$ satisfies \eqref{2dIncom} in the distributional sense, we consider the functionals
\begin{align*}
M(\rho, \mathbf{u}, \mathbf{m})_t   &=  \langle \mathbf{m}(t), \bm{\phi} \rangle   -  \langle \mathbf{m}(0), \bm{\phi} \rangle 
  -  \int_0^t  \langle \mathbf{m} \otimes \mathbf{u}
-\nu  \nabla \mathbf{u}, \nabla\bm{\phi} \rangle\mathrm{d}r
\\
& 
+(\lambda+\nu)\int_0^t  \langle \mathrm{div}\, \mathbf{u}, \mathrm{div}\,\bm{\phi} \rangle\mathrm{d}r
-
\frac{1}{\varepsilon^{2m}}\int_0^t  \langle \rho^\gamma, \mathrm{div}\,\bm{\phi} \rangle\mathrm{d}r
\\
&+\frac{1}{\varepsilon}\int_0^t  \langle \mathbf{e}_3\times \mathbf{m}, \,\bm{\phi}  \rangle\mathrm{d}r
-
\frac{1}{\varepsilon^2}\int_0^t  \langle \rho\nabla G\,, \,\bm{\phi}  \rangle\mathrm{d}r,
\\
N(\rho,  \mathbf{m})_t   &=
\sum_{k\in\mathbb{N}}\int_0^t\langle \mathbf{g}_k(\rho,  \mathbf{m})\,,\, \bm{\phi}\rangle^2\,\mathrm{d}r,
\\
N_k(\rho,  \mathbf{m})_t   &=
\int_0^t\langle \mathbf{g}_k(\rho,  \mathbf{m})\,,\, \bm{\phi} \rangle\,\mathrm{d}r,
\end{align*}
for all $\bm{\phi} =( \underline{\bm{\phi}} ,0)\in C^{\infty}_{c,\mathrm{div}}(\mathcal{O})$ where $ \underline{\bm{\phi}}\in C^{\infty}_{c,\mathrm{div}_h}(\mathbb{R}^2)$. Then by  Section \ref{sec:Coriolis} and integration by parts, we have that 
\begin{equation}
\begin{aligned}
\frac{1}{\varepsilon}\int_0^t  \langle \mathbf{e}_3\times \hat{\varrho}_{\varepsilon_n}\hat{\mathbf{u}}_{\varepsilon_n}, \,\bm{\phi}  \rangle\mathrm{d}r 
&=\frac{1}{\varepsilon}\int_0^t  \big\langle \,\ulcorner \mathbf{e}_3\times \hat{\varrho}_{\varepsilon_n}\hat{\mathbf{u}}_{\varepsilon_n}\urcorner\,, \,\underline{\bm{\phi}}  \big\rangle_h\mathrm{d}r 
=0.
\end{aligned}
\end{equation}
Also  since $\mathrm{div}\,\bm{\phi}=0$, we can use \eqref{staticProblem} and $\hat{r}_{\varepsilon_n}=\frac{\hat{\varrho}_{\varepsilon_n}-\overline{\varrho}_\varepsilon}{\varepsilon^m}$ to show that
 \begin{equation}
 \begin{aligned}
\frac{1}{\varepsilon^2}\int_0^t  \langle  \hat{\varrho}_{\varepsilon_n}\nabla G&\,, \, \bm{\phi}  \rangle\mathrm{d}r 
=
\frac{1}{\varepsilon^2}\int_0^t  \langle  \overline{\varrho}_{\varepsilon}\nabla G\,, \, \bm{\phi}  \rangle\mathrm{d}r 
\\
&+
\frac{1}{\varepsilon^2}\int_0^t  \langle(\hat{\varrho}_{\varepsilon_n}-\overline{\varrho}_\varepsilon)\nabla G\,, \, \bm{\phi}  \big\rangle\mathrm{d}r 
\\
&=
\frac{1}{\varepsilon^{2m}}\int_0^t  \langle  \nabla \overline{\varrho}_{\varepsilon}^\gamma \,, \, \bm{\phi}  \rangle\mathrm{d}r 
+
\varepsilon^{m-2}\int_0^t  \langle\hat{r}_{\varepsilon_n}\nabla G\,, \, \bm{\phi}  \big\rangle\mathrm{d}r 
\\
&=\varepsilon^{m-2}\int_0^t  \langle\hat{r}_{\varepsilon_n}\nabla G\,, \, \bm{\phi}  \big\rangle\mathrm{d}r 
\end{aligned}
\end{equation}
by integration by parts.  Finally since $\langle \mathrm{div}\,\hat{\mathbf{u}}_{\varepsilon_n}\,,\, \mathrm{div}\,\bm{\phi} \rangle= \langle \hat{\varrho}_{\varepsilon_n}^\gamma\,,\, \mathrm{div}\,\bm{\phi} \rangle=0$, we can therefore conclude that for $m>10$,
\begin{align*}
&M\big(\hat{\varrho}_{\varepsilon_n}, \hat{\mathbf{u}}_{\varepsilon_n}, \hat{\varrho}_{\varepsilon_n}\hat{\mathbf{u}}_{\varepsilon_n}\big)_t   =  \big\langle \big(\hat{\varrho}_{\varepsilon_n}\hat{\mathbf{u}}_{\varepsilon_n}\big)(t), \bm{\phi} \big\rangle   -  \big\langle \big(\hat{\varrho}_{\varepsilon_n}\hat{\mathbf{u}}_{\varepsilon_n}\big)(0), \bm{\phi}  \big\rangle 
\\
&-  \int_0^t  \langle \hat{\varrho}_{\varepsilon_n}\hat{\mathbf{u}}_{\varepsilon_n} \otimes \hat{\mathbf{u}}_{\varepsilon_n}
-\nu  \nabla \hat{\mathbf{u}}_{\varepsilon_n}, \nabla \bm{\phi}  \rangle\mathrm{d}r
+
\varepsilon^{m-2}\int_0^t  \langle\hat{r}_{\varepsilon_n}\nabla G\,, \, \bm{\phi}  \big\rangle\mathrm{d}r.
\end{align*}
As a result of \eqref{limDenseAndCentri} and Proposition \ref{prop:Jakubow0}, we can conclude that the last  term above  vanishes $\tilde{\mathbb{P}}$-a.s. in the limit as $\varepsilon\rightarrow0$ since $m>10$.

Now passing to the limit $n \rightarrow \infty$ in the function $M(\cdot)_t$ and keeping in mind \eqref{convectiveConver}  and Proposition \ref{prop:Jakubow0} we obtain\footnote{We denote by $M(\cdot)_{s,t}$, the difference $M(\cdot)_{t}-M(\cdot)_{s}$ and similarly for the other functionals. Also $\textbf{r}_t$ is a continuous map that restrict functions to time $t$ whereas $h$ is a continuous function. See \cite[Section 3]{breit2015incompressible}  for further details.}
\begin{equation}
\begin{aligned}
\label{continuousMartingaleLimit}
\tilde{\mathbb{E}}\,h( \textbf{r}_t\hat{\mathbf{U}}, \textbf{r}_t\tilde{W})\left[M(1, \hat{\mathbf{U}}, \hat{\mathbf{U}})_{s,t}  \right]
=0,
\\
\tilde{\mathbb{E}}\,h( \textbf{r}_t\hat{\mathbf{U}}, \textbf{r}_t\tilde{W})\left[\left[ M(1, \hat{\mathbf{U}}, \hat{\mathbf{U}})^2\right]_{s,t}- N(1, \hat{\mathbf{U}})_{s,t} \right]
=0,
\\
\tilde{\mathbb{E}}\,h( \textbf{r}_t\hat{\mathbf{U}}, \textbf{r}_t\tilde{W})\left[\left[ M(1, \hat{\mathbf{U}}, \hat{\mathbf{U}})\tilde{\beta}_k\right]_{s,t}- N(1, \hat{\mathbf{U}})_{s,t} \right]
=0.
\end{aligned}
\end{equation}
where the $\tilde{\mathbb{P}}$-a.s. limit process satisfies 
\begin{align*}
&M(1, \hat{\mathbf{U}}, \hat{\mathbf{U}})_t   = 
 \langle  \hat{\mathbf{U}}(t), \bm{\phi}  \rangle   -  \langle \hat{\mathbf{U}}(0), \bm{\phi} \rangle 
  -  \int_0^t  \langle \hat{\mathbf{U}}\otimes\hat{\mathbf{U}}
-\nu  \nabla \hat{\mathbf{U}}, \nabla \bm{\phi} \rangle\mathrm{d}r
\\
&= 
\left\langle  \hat{\mathbf{U}}_h(t), \bm{\phi} _h \right\rangle_h   -  \left\langle \hat{\mathbf{U}}_h(0), \bm{\phi} _h \right\rangle_h 
  -  \int_0^t  \left\langle \hat{\mathbf{U}}_h\otimes\hat{\mathbf{U}}_h
-\nu  \nabla_h \hat{\mathbf{U}}_h\,, \nabla_h\bm{\phi} _h \right\rangle_h\mathrm{d}r.
\end{align*}
Eq. \eqref{continuousMartingaleLimit} implies that $M(1, \hat{\mathbf{U}}_h, \hat{\mathbf{U}}_h)$ is an $(\tilde{\mathscr{F}}_t)$-martingale with quadratic and cross variations given by
\begin{align*}
\langle\langle M(1, \hat{\mathbf{U}}_h, \hat{\mathbf{U}}_h) \rangle\rangle= N(1, \hat{\mathbf{U}}_h), 
\quad 
\langle\langle M(1, \hat{\mathbf{U}}_h, \hat{\mathbf{U}}_h), \tilde{\beta}_k \rangle\rangle= N_k(1, \hat{\mathbf{U}}_h)
\end{align*}
respectively and thus we obtain
\begin{align*}
\Big\langle\Big\langle 
M(1, \hat{\mathbf{U}}_h, \hat{\mathbf{U}}_h)
-
\int_0^\cdot
\langle \Phi(1, \hat{\mathbf{U}}_h)\mathrm{d}\tilde{W}, \phi  \rangle
 \Big\rangle\Big\rangle=0.
\end{align*}
It therefore follows that \eqref{2dIncom}$_1$ is satisfied in the distributional sense.
\\
The same can be done for  $( (\tilde{\Omega}, \tilde{\mathscr{F}}_t,(\tilde{\mathscr{F}}_t),\tilde{\mathbb{P}}),\check{\mathbf{U}}, \tilde{W})$.
\end{proof}

\subsection{Pathwise solvability of the limit problem}
\label{subsec:pathwise}
\begin{proof}[Proof of Theorem \ref{thm:mainRo}]
With Theorem \ref{the:important} in hand, we can now use the assumption on the initial law and Theorem \ref{thm:2duniqueness}  to get that $\tilde{\mathbb{P}}$-a.s., $\hat{\mathbf{U}}_h(0)=\check{\mathbf{U}}_h(0)=\mathbf{U}_{h,0}$ and as such, the pair of solutions $\hat{\mathbf{U}}_h$ and $\check{\mathbf{U}}_h$ coincide $\tilde{\mathbb{P}}$-a.s. with law
\begin{align*}
&\mu\left( \left(\hat{\varrho}, \hat{\mathbf{U}},\hat{\mathbf{m}}, \check{\varrho},\check{\mathbf{U}}, \check{\mathbf{m}} \right)\in \chi \, : \,
(\hat{\varrho}, \hat{\mathbf{U}},\hat{\mathbf{m}})
=( \check{\varrho},\check{\mathbf{U}}, \check{\mathbf{m}})
\right)
\\
&
=\tilde{\mathbb{P}}\left((\hat{\varrho}, \hat{\mathbf{U}},\hat{\mathbf{m}})
=( \check{\varrho},\check{\mathbf{U}}, \check{\mathbf{m}})
\right)
=
\tilde{\mathbb{P}}\left(\hat{\mathbf{U}}_h =\check{\mathbf{U}}_h \right)=1.
\end{align*}
We can now use the generalization to quasi-Polish spaces of the Gy\"{o}ngy-Krylov characterization of convergence given in \cite[Theorem 2.10.3]{breit2017stoch} to show that the original sequence $\left(\varrho_\varepsilon,\mathbf{u}_\varepsilon,\ulcorner \mathcal{P}\,(\varrho_\varepsilon\mathbf{u}_\varepsilon)\urcorner \right)$ defined on the original probability space $\left( \Omega,\mathscr{F},\mathbb{P}\right)$ converges in probability to some random variables $(\varrho,\mathbf{U}_h, \mathbf{m}_h)$ in the topology  of $\chi_\varrho \times \chi_\mathbf{u} \times \chi_{\ulcorner\varrho\mathbf{u}\urcorner}$. 

We can now repeat Section \ref{subsec:limits} for $\left(\varrho_\varepsilon,\mathbf{u}_\varepsilon,\ulcorner\mathcal{P}\,(\varrho_\varepsilon\mathbf{u}_\varepsilon)\urcorner \right)$ and finally get that $\mathbf{U}=\mathbf{U}_h$ is a pathwise solution of 
\eqref{2dIncom} according to Definition \ref{def:pathSol}. This repetition is comparatively simpler since we are dealing with the original sequence. As a consequence, it suffices to take only one Wiener process in Lemma \ref{lem:WeinerSeq} for example.
\end{proof}

\subsection*{Acknowledgment}
The author is grateful to D. Breit and E. Feireisl for their many useful discussions, suggestions and corrections.


\begin{thebibliography}{10}

\bibitem{babin1999global}
A.~Babin, A.~Mahalov, B.~Nicolaenko :
\newblock Global regularity of 3{D} rotating {N}avier-{S}tokes equations for
  resonant domains.
\newblock {\em Indiana Univ. Math. J.}, 48(3):1133--1176, (1999).

\bibitem{babin20013d}
A.~Babin, A.~Mahalov, B.~Nicolaenko :
\newblock 3{D} {N}avier-{S}tokes and {E}uler equations with initial data
  characterized by uniformly large vorticity.
\newblock {\em Indiana Univ. Math. J.}, 50(Special Issue):1--35, (2001).
  

\bibitem{breit2015incompressible}
D.~Breit, E.~Feireisl, M.~Hofmanov{\'a} :
\newblock Incompressible limit for compressible fluids with stochastic forcing.
\newblock {\em Arch. Ration. Mech. Anal.}, 222(2):895--926, (2016).

\bibitem{breit2015compressible}
D.~Breit, E.~Feireisl, M.~Hofmanov{\'a} :
\newblock Compressible fluids driven by stochastic forcing: The relative energy
  inequality and applications.
\newblock {\em Comm. Math. Phys.}, 350(2):443--473, (2017).

\bibitem{breit2017stoch}
D.~Breit, E.~Feireisl, M.~Hofmanov{\'a} :
\newblock {\em Stochastically Forced Compressible Fluid Flows}.
\newblock Berlin, Boston: De Gruyter, (2018).

\bibitem{Hof}
D.~Breit, M.~Hofmanov\'{a} :
\newblock Stochastic {N}avier--{S}tokes equations for compressible fluids.
\newblock {\em Indiana Univ. Math. J.}, 65(4):1183--1250, (2016).

\bibitem{chemin2006mathematical}
J.-Y. Chemin, B.~Desjardins, I.~Gallagher, E.~Grenier :
\newblock {\em Mathematical geophysics}, volume~32 of {\em Oxford Lecture
  Series in Mathematics and its Applications}.
\newblock The Clarendon Press, Oxford University Press, Oxford, (2006).


\bibitem{da2014stochastic}
G.~Da~Prato, J.~Zabczyk :
\newblock {\em Stochastic equations in infinite dimensions}, volume 152 of {\em
  Encyclopedia of Mathematics and its Applications}.
\newblock Cambridge University Press, Cambridge, second edition, (2014).

\bibitem{ebin1983viscous}
D.~G. Ebin :
\newblock Viscous fluids in a domain with frictionless boundary.
\newblock In {\em Global analysis--analysis on manifolds}, volume~57 of {\em
  Teubner-Texte Math.}, pages 93--110. Teubner, Leipzig, (1983).

\bibitem{feireisl2012multi}
E.~Feireisl, I.~Gallagher, D.~Gerard-Varet, A.~n. Novotn\'y :
\newblock Multi-scale analysis of compressible viscous and rotating fluids.
\newblock {\em Comm. Math. Phys.}, 314(3):641--670, (2012).

\bibitem{feireisl2012a}
E.~Feireisl, I.~Gallagher,  A.~n. Novotn\'y :
\newblock A singular limit for compressible rotating fluids.
\newblock {\em SIAM J. Math. Anal.}, 44(1):192--205, (2012).

\bibitem{feireisl2009singular}
E.~Feireisl, A.~Novotn{\'y} :
\newblock {\em Singular limits in thermodynamics of viscous fluids}.
\newblock Springer Science \& Business Media, (2009).

\bibitem{flandoli2012stochastic}
F.~Flandoli, A.~Mahalov :
\newblock Stochastic three-dimensional rotating {N}avier--{S}tokes equations:
  averaging, convergence and regularity.
\newblock {\em Arch. Ration. Mech. Anal.}, 205(1):195--237, (2012).

\bibitem{gallagher2006weak}
I.~Gallagher, L.~Saint-Raymond :
\newblock Weak convergence results for inhomogeneous rotating fluid equations.
\newblock {\em J. Anal. Math.}, 99:1--34, (2006).

\bibitem{gyongy1996existence}
I.~Gy\"ongy, N.~Krylov :
\newblock Existence of strong solutions for {I}t\^o's stochastic equations via
  approximations.
\newblock {\em Probab. Theory Related Fields}, 105(2):143--158, (1996).

\bibitem{jakubowski1998short}
A.~Jakubowski :
\newblock The almost sure skorokhod representation for
  subsequences in nonmetric spaces.
\newblock {\em Theory Probab. Appl.}, 42(1):167--174, (1998).


\bibitem{menaldi2002stochastic}
J.-L. Menaldi, S.~S. Sritharan :
\newblock Stochastic 2-{D} {N}avier--{S}tokes equation.
\newblock {\em Appl. Math. Optim.}, 46(1):31--53, (2002).

\bibitem{mensah2016existence}
P.~R. Mensah :
\newblock Existence of martingale solutions and the incompressible limit for
  stochastic compressible flows on the whole space.
\newblock {\em Ann. Mat. Pura Appl. (4)}, 196(6):2105--2133, (2017).

\bibitem{ngo2016dispersive}
V.-S. Ngo, S.~Scrobogna :
\newblock Dispersive effects of weakly compressible and fast rotating inviscid
  fluids.
\newblock {\em Discrete Contin. Dyn. Syst.}, 38(2):749--789, (2018).

\bibitem{ondrejat2010stochastic}
M.~Ondrej\'at :
\newblock Stochastic nonlinear wave equations in local {S}obolev spaces.
\newblock {\em Electron. J. Probab.}, 15:no. 33, 1041--1091, (2010).


\end{thebibliography}
\end{document}